\documentclass[11pt]{amsart}

\usepackage{amsmath}
\usepackage{amsxtra}
\usepackage{bbm}
\usepackage{amscd}
\usepackage{amsthm}
\usepackage{amsfonts}
\usepackage{amssymb}
\usepackage{eucal}
\usepackage{xcolor}
\usepackage{url}
\usepackage[all]{xy}
\usepackage{tikz-cd}
\usepackage{enumitem}

\usepackage{hyperref}

\newtheorem{theorem}{Theorem}[section]
\newtheorem{cor}[theorem]{Corollary}
\newtheorem{lem}[theorem]{Lemma}
\newtheorem{prop}[theorem]{Proposition}

\newtheorem{thm}[theorem]{Theorem}

\newtheorem{example}[theorem]{Example}
\newtheorem*{question}{Question}
\newenvironment{lemma*}
 {\pushQED{\qed}\lem}
 {\popQED\endlem}

\theoremstyle{definition}
\newtheorem{defn}[theorem]{Definition}

\theoremstyle{remark}
\newtheorem{rem}[theorem]{Remark}

\numberwithin{equation}{section}
\makeatletter
\def\l@subsection{\@tocline{2}{0pt}{1.5pc}{4pc}{}}
\makeatother

\setlist[itemize]{topsep = 0pt}

\begin{document}

\newcommand{\nc}{\newcommand}
\nc{\thmref}[1]{theorem~\ref{#1}}
\nc{\secref}[1]{Sect.~\ref{#1}}
\nc{\lemref}[1]{Lemma~\ref{#1}}
\nc{\propref}[1]{Proposition~\ref{#1}}
\nc{\corref}[1]{Corollary~\ref{#1}}
\nc{\remref}[1]{Remark~\ref{#1}}
\nc{\conjref}[1]{Conjecture~\ref{#1}}
\nc{\Z}{{\mathbb Z}}
\nc{\C}{{\mathbb C}}
\nc{\CC}{{\mathcal C}}
\nc{\arr}{\rightarrow}
\nc{\ri}{\rangle}
\nc{\lef}{\langle}
\nc{\su}{\widehat{\mathfrak{sl}}_2}
\nc{\g}{{\mathfrak g}}
\nc{\h}{{\mathfrak h}}
\nc{\ttt}{{\mathfrak t}}
\nc{\bor}{{\mathfrak b}}
\nc{\bi}{\bibitem}
\nc{\ds}{\displaystyle}
\nc{\uqg}{U_q(\g)}
\nc{\uqgm}{U_{q^{-1}}(\g)}
\nc{\uqbm}{U_{q^{-1}}(\bor)}
\nc{\uqgpm}{U_{q^{\pm 1}}(\g)}
\nc{\uqbpm}{U_{q^{\pm 1}}(\bor)}
\nc{\uqb}{U_q(\bor)}
\nc{\uqhp}{U_q(\h^+)}
\nc{\F}{\mathcal{F}}
\nc{\Om}{\mathcal{O}^-}
\nc{\Op}{\mathcal{O}^+}
\nc{\qbin}[3]{\genfrac{[}{]}{0pt}{}{#1}{#2}_{#3}}
\nc{\pp}{\sigma}
\nc\Eval[3]{\left.#1\right\rvert_{#2}^{#3}}
\nc{\uqgmu}[1]{U_q^{#1}(\g)}
\nc{\ups}{\Omega}

\title[A functor for $R$-matrices in the category $\mathcal{O}$]{A functor for constructing $R$-matrices in the\\ category $\mathcal{O}$ of Borel quantum loop algebras}

\author{Théo Pinet}

\address{\parbox{\linewidth}{Universit\'e de Paris and Sorbonne Universit\'e, CNRS,
  IMJ-PRG, F-75006, Paris, France\\ Universit{\'e} de Montr{\'e}al, DMS,
Montr{\'e}al, Qu{\'e}bec, Canada, H3C 3J7\\ \textit{Email adress}: \textnormal{tpinet@imj-prg.fr}}}

\begin{abstract} We tackle the problem of constructing $R$-matrices for the category $\mathcal{O}$ associated to the Borel subalgebra of an arbitrary untwisted quantum loop algebra $\uqg$. For this, we define an exact functor $\F_q$ from the category $\mathcal{O}$ linked to $\uqgm$ to the one linked to $\uqg$. This functor $\F_q$ is compatible with tensor products, preserves irreducibility and interchanges the subcategories $\Op$ and $\Om$ of (D. Hernandez, B. Leclerc, Algebra Number Theory, 2016). We construct $R$-matrices for $\Op$ by applying $\F_q$ on the braidings already found for $\Om$ in (D. Hernandez, Rep. Theory, 2022). We also use the factorization of the latter intertwiners in terms of stable maps to deduce an analogous factorization for our new braidings. We finally obtain as byproducts new relations for the Grothendieck ring $K_0(\mathcal{O})$ as well as a functorial interpretation of a remarkable ring isomorphism $K_0(\Op)\simeq K_0(\Om)$ of Hernandez--Leclerc.
\end{abstract} 
\maketitle
\thispagestyle{empty}
{\small
\setlength{\parskip}{0in}
\tableofcontents}
\section{Introduction}\label{sec:Intro} Consider $q \in \mathbb{C}^{\times}$ not a root of unity with $\uqg$ an untwisted quantum loop algebra (i.e.~the quotient of some untwisted quantum affine algebra at level 0). It is well-known (see e.g.~\cite{dr1}) that the category $\CC$ of finite-dimensional $\uqg$-modules admits isomorphisms $V\otimes W \simeq W\otimes V$ for generic simple objects $V$ and $W$. These generic 
braidings, called \textit{$R$-matrices}, give solutions to the Yang--Baxter equation and are typically obtained using the universal $R$-matrix of $\uqg$.\par
To obtain such braidings for infinite-dimensional modules, a natural path is to replace the category $\mathcal{C}$ with the category $\mathcal{O}$ associated to the Borel subalgebra $\uqb$ of $\uqg$. The latter category, introduced in \cite{hj}, is however not generically braided (see e.g.~\cite{bjmst, her}) and we thus need to restrict the class of simple modules considered in order to define new $R$-matrices. This was done recently by Hernandez in \cite{her} who found a way to obtain braidings for a notable monoidal subcategory $\Om$ of $\mathcal{O}$. His approach relies crucially on the fact that, for every simple module $V$ in $\Om$, there exists a sequence $(V_k)_{k\geq 1}$ of finite-dimensional simple $\uqg$-modules for which the sequence of normalized $q$-characters $\overline{\chi}_q(V_k)$ (which are generating functions for the dimensions of the eigenspaces related to the action of a commutative subalgebra of $\uqb$) tends to the normalized $q$-character $\overline{\chi}_q(V)$ (in some ring of formal series, 
see Proposition \ref{prop:limitOm}). \par
The category $\mathcal{O}$ contains all finite-dimensional $\uqb$-modules as well as the prefundamental representations $L_{i,a}^{\pm}$ of \cite{hj}. The latter representations are parametrized by a sign $\pm$, a scalar $a \in \mathbb{C}^{\times}$ and an element $i$ of the set $I=\{1,...,n\}$ where $n$ is the rank of the finite-dimensional simple Lie algebra $\dot{\g}$ underlying $\uqg$. These representations play a central role in the study of Baxter $Q$-operators and were used by Frenkel--Hernandez \cite{fh1} to prove a conjecture of Frenkel--Reshitikhin \cite{fr} about the spectra of quantum integrable systems. \par 
The subcategory $\Om$ is generated by the negative prefundamental representations $L_{i,a}^-$ along with the finite-dimensional $\uqb$-modules. There is also a similar subcategory $\Op$ where the negative prefundamental representations are replaced by the positive ones. The subcategories $\mathcal{O}^{\pm}$ are of primordial importance in the study of monoidal categorifications of cluster algebras (see e.g.~\cite{hl2, kkop2,kkop3}) and are related by an isomorphism $D:K_0(\Om)\simeq K_0(\Op)$ of Grothendieck rings that sends classes of simple modules to classes of simple modules. We will call this isomorphism \textit{Hernandez--Leclerc's duality}.
\par
One can ask the following natural questions:
\begin{itemize}
\item[1.] Can we construct explicitly the $R$-matrices for simple modules of the positive subcategory $\mathcal{O}^+$? Are these explicit $R$-matrices related to the ones already constructed by Hernandez for the negative subcategory $\mathcal{O}^-$?
\item[2.] Can we relate $\Om$ and $\Op$ with some invertible exact functor that behaves well with respect to tensor product of modules? Are the subcategories $\mathcal{O}^{\pm}$ related by something deeper than Hernandez--Leclerc's duality $D:K_0(\Om) \simeq K_0(\Op)$? 
\end{itemize} 
The relation between the normalized $q$-characters of simple objects in $\Op$ and of simple finite-dimensional $\uqg$-modules is more intricate than in the case of the negative subcategory $\Om$ (see Section \ref{sec:qcaract}) and the approach taken in \cite{her} cannot be used directly for Question 1. This disrepancy regarding normalized $q$-characters is moreover not the only difference between the subcategories $\mathcal{O}^{\pm}$ and answering positively Question 2 may thus seem hopeless at first glance. For example, the stable maps defined in \cite{her} (which are remarkable automorphisms of tensor products of $\uqb$-modules) are uniquely determined when the underlying modules are simple objects of $\mathcal{O}^{-}$, but can be non-uniquely defined for simple modules in $\mathcal{O}^+$ (see Section \ref{sec:Stab}). The categories  $\mathcal{O}^{\pm}$ also relate differently to representations of shifted quantum affine algebras. \par
The shifted quantum affine algebra $U_q^{\mu}(\g)$ is a variation of the quantum loop algebra $\uqg$. It originated in the context of quantized $K$-theoretic Coulomb branches of 3d $N=4$ SUSY quiver gauge theories (see \cite{ft}) and is parametrized by a coweight $\mu$ of the finite-dimensional Lie algebra $\dot{\g}$ underlying $\uqg$. In \cite{hshift}, Hernandez defined a category $\mathcal{O}$ (denoted by $\mathcal{O}_{\mu}$) for $U_q^{\mu}(\g)$ and constructed analogs of positive (negative) prefundamental representations for this algebra when $\mu=\omega_i^{\vee}$ (resp.~ when $\mu=-\omega_i^{\vee}$) where $\omega_i^{\vee}$ is the $i$th-fundamental coweight of $\dot{\g}$ (see also \cite{z1} for an analogous construction for shifted yangians). The so-constructed negative prefundamental representations are infinite-dimensional simple modules whereas the positive ones all have dimension 1. This asymmetry comes from the fact that the shifted algebra $U_q^{\mu}(\g)$ contains a copy of the Borel subalgebra $\uqb\subseteq \uqg$ when the coweight $\mu$ is antidominant (while this is not true when $\mu$ is dominant). The $\uqb$-module $L_{i,a}^-$ of $\Om$ can thus be realized as the restriction of a simple $U_q^{\mu}(\g)$-module (for $\mu = -\omega_i^{\vee}$) whereas this cannot be done for the positive prefundamental representations of $\uqb$ in $\Op$.

However, even with these technical differences, one can relate the subcategories $\mathcal{O}^{\pm}$ corresponding to distinct quantum parameters $q$. This was already partially done in \cite{hj} where a procedure is given for constructing a positive prefundamental representation of $\uqb$ from a negative prefundamental representation of $\uqbm$. Unfortunately, the given procedure is not functorial and cannot \textit{a priori} be extended to all objects of $\Op$.  \par
We resolve this problem in the present paper and define a functor $\F_q$ from the category $\mathcal{O}$ linked to $\uqgm$ to the one linked to $\uqg$. This functor sends negative (positive) prefundamental representations of $\uqbm$ to positive (negative, resp.) prefundamental representations of $\uqb$ and satisfies the conditions given in Question 2. It arises naturally as the pullback by an isomorphism of algebras $\pp_q:\uqg\arr\uqgm$ which is given on the usual Drinfeld--Jimbo generating set $\{e_i,f_i,k_i^{\pm 1}\}_{i=0}^n$ of $\uqg$ by 
$$\pp_q(e_i)=-k_i^{-1}e_i,\quad \pp_q(f_i) = -f_ik_i \ \ \text{and} \ \  \pp_q(k_i^{\pm 1}) = k_i^{\mp 1}.$$ 
It is clear that $\pp_q$ restricts to an isomorphism $\uqb\simeq \uqbm$ and that the pullback $\pp_q^*$ by $\pp_q$ gives an exact invertible functor which preserves dimension and irreducibility of modules. This pullback is moreover compatible with the notion of category $\mathcal{O}$ and thus induces a functor $\F_q : \mathcal{O}_{q^{-1}} \arr \mathcal{O}_q$ with $\mathcal{O}_{q^{\pm 1}}$ the category $\mathcal{O}$ associated to $\uqgpm$. Our first main results are:
\begin{theorem}\label{thm:ComonIntro} Fix $V,W$ in $\mathcal{O}_{q^{-1}}$. Then $\F_q(V\otimes W)\simeq \F_q(W)\otimes \F_q(V)$ as $\uqb$-modules.
\end{theorem}
\begin{theorem}\label{thm:PosNeg} The functor $\F_q$ maps the full subcategory $\mathcal{O}^{\pm}_{q^{-1}}$ of $\mathcal{O}_{q^{-1}}$ to the subcategory $\mathcal{O}^{\mp}_{q}$ of $\mathcal{O}_q$. Moreover there exists a $\gamma\in \C^{\times}$ such that $ \F_q(L_{i,a}^{\pm, q^{-1}})\simeq L_{i,a\gamma}^{\mp,q} $ for all $i\in I$ and $a\in \C^{\times}$ (with $L_{i,a}^{\pm, q^{-1}}$ and $L_{i,a}^{\pm, q}$ the prefundamental representations of $\uqbm$ and $\uqb$, respectively). 
\end{theorem} 
In addition to answering positively (a slightly modified version of) Question 2, these results also answer Question 1 as they allow us to find explicit braidings for the positive subcategory $\Op_q$ of $\mathcal{O}_q$ by using $\F_q$ and the braidings of $\Om_{q^{-1}}$ given in \cite{her}. Let us clarify. \par 
It is well-known (see e.g.~\cite{her}) that a $\uqb$-module $V$ can be deformed into a module $V(u)$ over $U_{q,u}(\bor) = \uqb\otimes \C(u)$ where $u$ is a formal variable called the \textit{spectral parameter}. (This is done using an automorphism $\tau_u$ of $U_{q,u}(\bor)$, see Section \ref{sec:DefQaff}.) With this formalism, the affine $R$-matrix given in \cite{her} for the simple objects $V,W$ of $\mathcal{O}_{q^{-1}}^{-}$ is a $U_{q,u}(\bor)$-linear isomorphism
$$R_{V,W}(u) : V(u)\otimes W \rightarrow W\otimes V(u)$$
which is meromorphic in $u$ and specializes to a $\uqb$-linear braiding $V\otimes W\simeq W\otimes V$ if $u=1$ is neither a singularity of $R_{V,W}(u)$ nor of its inverse. \par 
Take now $V$ and $W$ simple objects in $\Op_q$. Let $V'$ and $W'$ be simple modules in $\Om_{q^{-1}}$ with $V=\F_q(V')$ and $W=\F_q(W')$ (see Theorem \ref{thm:PosNeg}). Let also $$R'(u) : W'(u)\otimes V' \rightarrow V'\otimes W'(u)$$
be the affine $R$-matrix given in \cite{her}. Then Theorem 1 can be used to get a invertible map $$V\otimes W(u^{-1}) \simeq \F_q(W'(u^{-1})\otimes V')\xrightarrow{R'(u^{-1})} \F_q(V'\otimes W'(u^{-1})) \simeq W(u^{-1})\otimes V$$
which induces a meromorphic $U_{q,u}(\bor)$-linear isomorphism $R_{V,W}(u):V(u)\otimes W \arr W\otimes V(u)$
by spectral deformation. The latter isomorphism is the wanted explicit $R$-matrix for $V,W$. It has exactly the same set of poles as the map $R'(u^{-1})$ above and specializes to a $\uqb$-linear braiding $V\otimes W \simeq W\otimes V$ if $u=1$ is neither a pole of $R'(u)$ nor of its inverse. \par 

The functor $\F_q$ can also be used to obtain a functorial interpretation of Hernandez--Leclerc's duality $D$ on an interesting subcategory $\mathcal{O}_{\Z,q}$ of $\mathcal{O}_q$ (see Section \ref{sec:FuncD}). More precisely, (assuming that $q$ is a formal variable and changing the base field) we can compose the functor $\F_q$ with a canonical functor $\mathcal{B}_{q^{-1},q}$ to get an autofunctor $\mathcal{D}_q$ of $\mathcal{O}_q$. This leads to our third main result:
\begin{theorem}\label{thm:DCatIntro}
The autofunctor $\mathcal{D}_q$ preserves irreducibility and induces a ring automorphism of $K_0(\mathcal{O}_q)$ that is equal to $D^{\pm 1}$ on the Grothendieck ring of the intersection of $\mathcal{O}^{\pm}_q$ and $\mathcal{O}_{\Z,q}$.
\end{theorem}
This produces a partial extension of the duality
$D$ on $K_0(\mathcal{O}_{\Z,q})$. We can however go further than this and construct an extension $\tilde{D}_q$ of $D$ on the full Grothendieck ring $K_0(\mathcal{O}_q)$ (hence answering a question of Hernandez--Leclerc). This is done in Appendix \ref{app:K0D} (for $q\in \C^{\times}$ not a root of unity or $q$ a formal variable) using again 
$\F_q$ with the intrinsic structure of $K_{0}(\mathcal{O}_q)$. The extension thus obtained is nevertheless not induced from an autofunctor of $\mathcal{O}_q$ and it is not clear whether or not it sends simple modules to simple modules. \par
We conclude this introduction by describing two additional uses for the functor $\F_q$. Other applications are given at the end of Section \ref{sec:RMat} but their study is kept for future work.
\par First, the 
isomorphism $K_0(\mathcal{O}_{q^{-1}})\simeq K_0(\mathcal{O}_q)$ of Grothendieck rings given by $[V] \mapsto [\F_q(V)]$ on equivalence classes of modules can be used to induce new relations for $K_0(\mathcal{O}_q)$ from already known relations of $K_0(\mathcal{O}_{q^{-1}})$. This lead is followed in Section \ref{sec:FGroth} of the present work and we apply the above isomorphism to two remarkable systems of relations in $K_0(\mathcal{O}_{q^{-1}})$, namely:
\begin{itemize}
\item[1.] the $Q\widetilde{Q}$-system of \cite{fh2} (which is deeply related to Bethe Ansatz equations) and 
\item[2.] the $QQ^*$-system of \cite{hl2} (which describes a cluster algebra structure on $K_0(\mathcal{O}_q^+)$).
\end{itemize} Only the first application yields new relations for $K_0(\mathcal{O}_q)$. The last one (with results of \cite{her}) nevertheless produces a novel ``categorified version'' of the $QQ^*$-system in terms of a non-split short exact sequence with simple extremal terms.
\par 
On the other hand, it was shown by Hernandez that affine $R$-matrices for the category $\mathcal{O}^-_{q}$ factorize in terms of stable maps. Applying the functor $\F_q$ to the corresponding factorization enables us to factorize the $R$-matrices of $\mathcal{O}_q^+$ obtained above in terms of what we call \textit{modified stable maps}. Unlike Hernandez's stable maps (for simple objects of $\mathcal{O}_q^-$), these modified maps are not morphisms for the action of the Cartan--Drinfeld subalgebra $\uqhp$ of $\uqb$, but are rather morphisms for the action of the subalgebra $\pp_q(U_{q^{-1}}(\mathfrak{h}^+))\subseteq \uqb$. This is the only real distinction between the stable maps (for simple objects of $\mathcal{O}_q^-$) and their modified analogs. \par
\begin{rem} 
We were informed, after completing the first version of this work, of a similarity between some of our results 
and those presented in \cite[Lemma 1.9]{z2} for the Lie superalgebra $\dot{\g}=\mathfrak{gl}(M|N)$ (where $M,N \geq 0$). Indeed, in \cite{z2}, Zhang uses a well-chosen algebra morphism to define a functor relating the category $\mathcal{O}$ associated to $\dot{\g}=\mathfrak{gl}(M|N)$ with the one associated to $\dot{\g}=\mathfrak{gl}(N|M)$ (for the same quantum parameter $q$). Zhang then shows, using the so-called RTT (or RLL) realization of $\uqg$ for $\g$ of type A (see e.g.~\cite{df,frt,z4}), that the functor thus defined reverses tensor products and inverts the sign of prefundamental representations. He also expects that his results can be understood as a categorification of Hernandez--Leclerc's duality $D$ between $K_0(\Op)$ and $K_0(\Om)$. This resembles our own Theorems \ref{thm:ComonIntro}, \ref{thm:PosNeg} and \ref{thm:DCatIntro}. However, the algebra morphism given in \cite{z2} is different from our isomorphism $\pp_q$ and cannot be used for Lie algebras $\g$ outside of type A. Our results, which treat uniformly the case of all Lie algebras $\g$, are hence of interest even in type A (and have no known analogs in general).
\end{rem}
The paper is organized as follows. Section \ref{sec:Qaff} recalls facts concerning quantum loop algebras, their Borel subalgebras and the category $\mathcal{O}$. It also recalls the definition of the categories $\mathcal{O}^{\pm}$ and that of Hernandez--Leclerc's duality $D:K_0(\mathcal{O}^-)\simeq K_0(\mathcal{O}^+)$. In Section \ref{sec:Fq}, we define the functor $\F_q:\mathcal{O}_{q^{-1}}\arr \mathcal{O}_q$ and show Theorems \ref{thm:ComonIntro} and \ref{thm:PosNeg}. We moreover apply the functor $\F_q$ to relations of $K_0(\mathcal{O}_{q^{-1}})$ and study the relations of $K_0(\mathcal{O}_q)$ hence obtained. The section then concludes with the functorial interpretation of the duality $D$. Finally, Section \ref{sec:RMat} is devoted to the explicit construction (and factorization) of braidings for the category $\mathcal{O}^+$. The associated constructions (and factorizations) for finite-dimensional modules and objects of the negative category $\mathcal{O}^-$ are recalled for completeness. An appendix ends the text with a proof that the map $D$ can be extended to a ring automorphism of all of $K_0(\mathcal{O})$.
\addtocontents{toc}{\setcounter{tocdepth}{-10}}
\subsection*{Acknowledgements} This work would have been impossible without the guidance of David Hernandez whom we warmly thank. We are also grateful to Ryo Fujita, Yvan Saint-Aubin and Huafeng Zhang for useful discussions. The author finally thanks Alexis Langlois-Rémillard for his reading of the first draft of this paper. This project has received funding from the European Union’s Horizon 2020 research and innovation programme under the Marie Skłodowska-Curie grant agreement \includegraphics[scale=0.07]{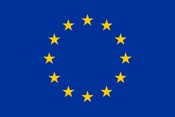} No 945332. This work was also supported by a scholarship (CGS-D) from the Natural Sciences and Engineering Research Council of Canada (NSERC). This support is gratefully acknowledged.
\addtocontents{toc}{\setcounter{tocdepth}{2}}
\section{Quantum loop algebras and the category $\mathcal{O}$}\label{sec:Qaff}
We recall here some known definitions and results about quantum loop algebras and the representation theory of their Borel subalgebras. We remind in particular the definition of the categories $\mathcal{O}$ and $\mathcal{O}^{\pm}$. We refer to \cite{hj,fh1} for details and to \cite{cp1} for a proper introduction to quantum affine algebras. We also refer to \cite{hbourb,kash,mo} for applications of the representation theory of quantum loop algebras. \par 
In this paper, $\mathbb{N} = \Z_{\geq 0}$ and all vector spaces, algebras and tensor products are defined over $\C$ unless otherwise specified.
\subsection{Definition of the algebras}\label{sec:DefQaff} Let $C = (C_{i,j})_{0\leq i,j\leq n}$ be an indecomposable Cartan matrix of untwisted affine type and let $\g$ be the corresponding affine Kac-Moody algebra. Set $D$ to be the unique diagonal matrix with relatively prime diagonal entries $d_0,\dots ,d_n \in \mathbb{N}_{>0}$ such that $DC$ is symmetric. There is a unique vector $\vec{a}$ with relatively prime entries $a_0,\dots ,a_n \in \mathbb{N}$ such that $a_0 = 1$ and $C \vec{a} = 0$ (see \cite[Chapter 4]{kac}). \par
Let $I = \{1,\dots ,n\}$ and denote by $\dot{\g}$ the simple finite-dimensional Lie algebra with Cartan matrix $(C_{i,j})_{i,j\in I}$. We let $\{\alpha_i\}_{i\in I}$ and $\{\omega_i\}_{i\in I}$ be respectively the simple roots and fundamental weights of $\dot{\g}$ with $\dot{\h}$ its Cartan subalgebra. Set $\alpha_0 = -(a_1\alpha_1+\dots +a_n\alpha_n)$ with 
$$ Q = \oplus_{i\in I} \Z\alpha_i,\, Q^+ = \oplus_{i\in I} \mathbb{N} \alpha_i\text{ and }P = \oplus_{i\in I} \Z \omega_i. $$
We also set $P_{\mathbb{Q}} = P\otimes_{\Z}\mathbb{Q}$ with its partial ordering given by $\omega' \leq \omega$ if and only if $\omega-\omega' \in Q^+$. \par 
Fix a quantum parameter $q \in \C^{\times}$ which is not a root of unity and suppose that $q=e^h$ for some $h \in \C$ so that $q^r$ is well-defined for any $r \in \C$. We will use $q_i = q^{d_i}$ for $i\in I$ with
$$\ds [m]_x = \frac{x^m-x^{-m}}{x-x^{-1}},\, \ds [m]_x! = \prod_{r=1}^m [r]_x\text{ and }\ds \qbin{m}{p}{x} = \frac{[m]_x!}{[p]_x![m-p]_x!}$$
for $x\in \C^{\times}$ not a root of unity and $m,p \in \mathbb{N}$ satisfying $m\geq p$. We let $\qbin{m}{p}{x}=0$ if $m<p$. \par\noindent
The quantum loop algebra $\uqg$ is the (unital) $\C$-algebra generated by $\{e_i,f_i,k_i^{\pm 1}\}_{i=0}^n$ with 
\begin{center}
$k_ik_j = k_jk_i, \quad k_i^{\pm 1}k_i^{\mp 1} = 1 = k_0^{a_0}k_1^{a_1}\dots k_n^{a_n}, \quad k_ie_jk_i^{-1} = q_i^{C_{i,j}}e_j,\quad k_if_jk_i^{-1} = q_i^{-C_{i,j}}f_j,$\smallskip\\
$\ds[e_i,f_j] = \delta_{i,j}\frac{k_i-k_i^{-1}}{q_i-q_i^{-1}}$, \quad $\ds\sum_{r=0}^{1-C_{i,j}}(-1)^re_i^{(1-C_{i,j}-r)}e_je_i^{(r)}=\sum_{r=0}^{1-C_{i,j}}(-1)^rf_i^{(1-C_{i,j}-r)}f_jf_i^{(r)} = 0$
\end{center}
for $0\leq i,j \leq n$ and where, for the last relation, $i\neq j$ and $x_i^{(r)} = x_i^r/[r]_{q_i}!$ (for $x_i = e_i,f_i$). It is a Hopf algebra for the coproduct and antipode (which is an anti-automorphism of $\uqg$) given by (for $0\leq i\leq n$)
\begin{center}
$\Delta(e_i) = e_i\otimes 1+k_i\otimes e_i, \quad \Delta(f_i) = f_i\otimes k_i^{-1}+1\otimes f_i, \quad \Delta(k_i) = k_i\otimes k_i$ \smallskip\\
$ S(e_i) = -k_i^{-1}e_i, \quad S(f_i) = -f_ik_i, \quad S(k_i) = k_i^{-1}.$
\end{center}
Since the seminal work of Drinfeld, Beck and Damiani (cf.~\cite{dr1,b,da1,da2}) it is known that $\uqg$ has another set of generators $\{x_{i,r}^{\pm}, \phi_{i,r}^{\pm}\,|\, i\in I,\, r \in \Z\}$ with $\phi_{i,0}^{\pm}=k_i^{\pm 1}$ and $\phi_{i,\pm m}^{\pm} = 0$ if $m < 0$. The defining relations for these generators are $k_i^{\pm 1}k_i^{\mp 1} = 1 = k_0^{a_0}k_1^{a_1}\dots k_n^{a_n}$ with the additional relations ($i,j \in I$ and $r,s,m \in \mathbb{Z}$ with $m\neq 0$)
\begin{center}
$[\phi_{i,r}^{\pm},\phi_{j,s}^{\pm}] = [\phi_{i,r}^{\pm},\phi_{j,s}^{\mp}] = 0$, \ \ 
$k_i x_{j,r}^{\pm} k_i^{-1} = q_i^{\pm C_{i,j}}x_{j,r}^{\pm}$, \ \ 
$[h_{i,m},x_{j,r}^{\pm}] = \pm \frac{1}{m}[mC_{i,j}]_{q_i}x_{j,r+m}^{\pm}$, \\ $\ds[x_{i,r}^+,x_{j,s}^-] = \delta_{i,j}\frac{\phi_{i,r+s}^+-\phi_{i,r+s}^-}{q_i-q_i^{-1}}$, \quad
$x_{i,r+1}^{\pm}x_{j,s}^{\pm}-q_i^{\pm C_{i,j}}x_{j,s}^{\pm}x_{i,r+1}^{\pm}=q_i^{\pm C_{i,j}}x_{i,r}^{\pm}x_{j,s+1}^{\pm}-x_{j,s+1}^{\pm}x_{i,r}^{\pm}$
\end{center}
as well as, for $r',r_1,\dots ,r_p \in \Z$ and $i\neq j$ with $p = 1-C_{i,j}$,
\begin{center}
$\ds \sum_{\pi \in \Sigma_p}\sum_{0\leq l\leq p}(-1)^l \qbin{p}{l}{q_i}x_{i,r_{\pi(1)}}^{\pm}\dots x_{i,r_{\pi(l)}}^{\pm}x_{j,r'}^{\pm}x_{i,r_{\pi(l+1)}}^{\pm}\dots x_{i,r_{\pi(p)}}^{\pm}=0$
\end{center}
where $\Sigma_p$ is the symmetric group on $p$ letters. The elements $h_{i,m}$ appearing in these relations are defined by the equality of generating functions
\begin{center}
$\ds\phi_i^{\pm}(z)=\sum_{m\geq 0} \phi_{i,\pm m}^{\pm} z^{\pm m} = k_i^{\pm 1}\exp\left(\pm(q_i-q_i^{-1})\sum_{m>0}h_{i,\pm m}z^{\pm m}\right)$.
\end{center}
\begin{example}[see e.g.~{\cite[Example 2.1]{fh1}}]\label{ex:DJDsl2} For $\g = \su$, a correspondence between the two generating sets is given by
\begin{center}
$e_0 = k_1^{-1}x_{1,1}^-$, \quad $e_1 = x_{1,0}^+$, \quad $f_0 = x_{1,-1}^+k_1$ \ and \ $f_1 = x_{1,0}^-$.
\end{center}
In general, we still have $e_i = x_{i,0}^+$ and $f_i = x_{i,0}^-$ for $i\in I$, but the possible expressions for $e_0$ and $f_0$ are somewhat more complicated (and not unique, see \cite{cp2} for details).
\end{example}
The algebra $\uqg$ is $Q$-graded by $\deg x_{i,r}^{\pm} = \pm \alpha_i$ and $\deg \phi_{i,r}^{\pm} = 0$ for $r \in \Z$. It also admits a $\Z$-grading given by $\deg e_0 = -\deg f_0 = 1$ with $\deg e_i = \deg f_i = \deg k_i^{\pm 1} = 0$ for $i\in I$. The latter $\Z$-grading verifies $\deg x_{i,r}^{\pm}= \deg \phi_{i,r}^{\pm}=r$ for any $i\in I$ and $m\in \Z$. Fixing $a \in \C^{\times}$, we have a Hopf algebra automorphism $\tau_{a}$ of $\uqg$ such that $\tau_{a}(x)=a^m x$ when $x$ is homogeneous of $\Z$-degree $m$. We can also replace $a$ by a formal variable $u$ to get an automorphism of the algebra $U_{q,u}(\g) = \uqg\otimes \C(u)$. (We always imply that the base field is $\C(u)$ when using this variable $u$. We will thus in particular write $\otimes$ for $\otimes_{\C(u)}$ when a tensor product of $\C(u)$-vector spaces is considered. This is the only exception to the rule mentioned at the beginning of the section.) The pullback of a module $V$ with respect to $\tau_{a}$ and $\tau_{u}$ are denoted $V(a)$ and $V(u)$. \par
In this paper, we focus on the representation theory of the Borel subalgebra $\uqb\subseteq \uqg$ which is defined as the subalgebra generated by the set $\{e_i,k_i^{\pm 1}\}_{i=0}^n$. It contains the elements $\phi_{i,r}^+$, $x_{i,r}^{+}$ and $x_{i,m}^{-}$ for $i\in I$, $r \geq 0$ and $m>0$ with the commutative subalgebra $\uqhp\subseteq \uqg$ generated by $\{k_i^{\pm 1},h_{i,m}\,|\, i\in I,\, m>0\}$ (cf.~\cite{hj}). The algebra $\uqb$ is also a Hopf subalgebra of $\uqg$ and inherits the two gradings defined above. In particular, the \textit{shift} $V(a)$ of a given $\uqb$-module $V$ with respect to $a \in \C^{\times}$ is a well-defined $\uqb$-module. \par 
We have a vector-space isomorphism \cite{b,bcp}
\begin{center}
$\uqg \simeq U^-_q(\g)\otimes U^0_q(\g) \otimes U^+_q(\g)$
\end{center}
where $U_q^{\pm}(\g)$ and $U_q^0(\g)$ are respectively the subalgebras generated by $\{x_{i,r}^{\pm}\,|\,i\in I,\, r \in \Z\}$ and $\{\phi_{i,r}^{+},\phi_{i,-r}^-\,|\, i\in I,\, r\in \mathbb{N}\}$. There is also a triangular decomposition \cite{b,hj}
\begin{center}
$\uqb \simeq U^-_q(\bor)\otimes U^0_q(\bor) \otimes U^+_q(\bor)$
\end{center}
where $U_q^{\pm}(\bor) = \uqb\cap U_q^{\pm}(\g)$ and $U_q^0(\bor) = \uqb\cap U_q^0(\g)$. By the results of \cite{hj}, we have
\begin{center}
$U_q^{+}(\bor)= \langle x_{i,m}^+\rangle_{i\in I,m\geq 0}$ and $U_q^0(\bor) = \langle \phi_{i,m}^+, k_i^{\pm 1}\rangle_{i\in I,m>0}$
\end{center}
but the subalgebra $U_q^-(\bor)$ has no such simple description in terms of the Drinfeld generators (except when $\g = \su$ since then $U_q^-(\bor) = \langle x_{1,m}^-\rangle_{m>0}$).
\subsection{Representation theory and the category $\mathcal{O}$}\label{sec:O}
 Denote by $\ttt\subseteq \uqb$ the commutative subalgebra generated by $\{k_i^{\pm 1}\}_{i\in I}$ and consider $\ttt^{\times} = (\C^{\times})^I$ with the group structure induced from pointwise multiplication. For a $\uqb$-module $V$ and $\mu = (\mu_i)_{i\in I} \in \ttt^{\times}$, the weight-space $V_{\mu}$ of $V$ associated to $\mu$ is the $\ttt$-eigenspace
\begin{center}
$V_{\mu} = \{v\in V\,|\, k_i v = \mu_i v\text{ for any } i\in I\}$.
\end{center}
We say that $\mu$ is a weight of $V$ if $V_{\mu}\neq 0$ and we denote by $P(V)$ the set of weights of $V$. \par We have an injective morphism of groups $\overline{\textcolor{white}{\alpha}} : P_{\mathbb{Q}} \rightarrow \ttt^{\times}$ defined on fundamental weights by $\overline{\omega_i} = (\overline{\omega_i}(j))_{j\in I}$ with $\overline{\omega_i}(j) = q_i ^{\delta_{ij}}$. We use this morphism to carry the natural order on $P_{\mathbb{Q}}$ to $\ttt^{\times}$. In other terms, we write $\mu \leq \nu$ in $\ttt^{\times}$ if $\nu\mu^{-1}\in\overline{Q_+}$. Clearly,
\begin{center}
$\phi_{i,r}^{\pm} V_{\mu} \subseteq V_{\mu}$ and $x_{i,r}^{\pm} V_{\mu} \subseteq V_{\mu\overline{\alpha_i}^{\pm 1}}$.
\end{center}
We have the following analogue of the BGG category $\mathcal{O}$ of usual Lie algebra theory.
\begin{defn}[{\cite[Definition 3.8]{hj}}]\label{def:O} The category $\mathcal{O}$ is the full monoidal subcategory of the category of all $\uqb$-modules whose objects are the modules $V$ satisfying
\begin{itemize}
\item[(i)] $V$ is $\ttt$-diagonalizable, that is $V = \bigoplus_{\mu \in \ttt^{\times}} V_{\mu}$;
\item[(ii)] $\dim V_{\mu} < \infty$ for all $\mu \in \ttt^{\times}$ and
\item[(iii)] $P(V) \subseteq \bigcup_{i=1}^s D(\lambda_i)$ for some $\lambda_1,\dots ,\lambda_s\in \ttt^{\times}$ where $D(\lambda) = \{\mu \in \ttt^{\times}\,|\, \mu\leq \lambda\}$.
\end{itemize}
\end{defn}
In other words, a $\uqb$-module $V$ is in $\mathcal{O}$ if and only if it decomposes into a direct sum of finite-dimensional weight spaces with appropriately bounded above weights (for $\leq$). \par An important notion in the study of the category $\mathcal{O}$ is the one of $\ell$-weight spaces (where $\ell$ stands for \textit{loop}). These are the (simultaneous) generalized eigenspaces of a module $V$ with respect to the commutative algebra $\uqhp$. More precisely, denote by $\ttt_{\ell}^{\times}$ the set of sequences $\Psi \in (\Psi_{i,r})_{i\in I,r\geq 0} \subseteq \C$ satisfying $\Psi_{i,0} \neq 0$ for each $i\in I$. Then, for such a $\Psi \in \ttt^{\times}_{\ell}$ and a fixed $\uqb$-module $V$, we call the subspace 
\begin{center}
$V_{\Psi} = \{v\in V\,|\, \text{there is } p \in \mathbb{N} \text{ such that } (\phi_{i,r}^+-\Psi_{i,r})^pv = 0 \text{ for all } i\in I \text{ and } r \geq 0\}$
\end{center}
the $\ell$-weight space of $V$ associated to $\Psi$. We say that $\Psi$ is a $\ell$-weight of $V$ if $V_{\Psi} \neq 0$. \par
We associate a given sequence $\Psi = (\Psi_{i,r})_{i\in I,r\geq 0}\in \ttt_{\ell}^{\times}$ with the corresponding sequence $(\Psi_i(z))_{i\in I}$ of generating functions $\Psi_i(z) = \sum_{r\geq 0} \Psi_{i,r}z^r$. We moreover endow $\ttt^{\times}_{\ell}$ with a group structure by using the standard multiplication of formal power series. There is then a natural surjective group morphism $\varpi:\ttt^{\times}_{\ell}\arr \ttt^{\times}$ given by the evaluation at $z =0$
\begin{center}
$\varpi((\Psi_{i}(z))_{i\in I})=(\Psi_i(0))_{i\in I}= (\Psi_{i,0})_{i\in I}$
\end{center}
and we have, for every $\Psi \in \ttt^{\times}_{\ell}$, a factorization $\Psi = \varpi(\Psi)\Psi_{\text{norm}}$ with $\Psi_{\text{norm}}$ such that $\varpi(\Psi_{\text{norm}})$ is trivial. We call $\varpi(\Psi)$ the constant part of $\Psi$.  \par 
A first reason justifying the use of $\ell$-weights in our study is that every object $V$ of $\mathcal{O}$ is the sum of its $\ell$-weight spaces (see e.g.~\cite{her,fr}). These spaces are also all finite-dimensional as
$V_{\Psi}\subseteq V_{\varpi(\Psi)}$ for every $\Psi\in \ttt^{\times}_{\ell}$. However, the most important fact regarding $\ell$-weights is that they can be used to parametrize the simple objects of the category $\mathcal{O}$. 
\begin{defn}[{\cite[Definition 3.2]{hj}}] Let $\Psi = (\Psi_{i,r})_{i\in I,r\geq 0}\in \ttt^{\times}_{\ell}$ and fix $V$ a $\uqb$-module. Then $V$ is of highest $\ell$-weight $\Psi$ if there is a non-zero $v \in V$ verifying $V = \uqb v$ and $e_iv = 0$ with $\phi_{i,r}^+ v = \Psi_{i,r}v$ for all $i \in I$ and $r \geq 0$. This $v$ is then said to be a highest $\ell$-weight vector.
\end{defn}
The highest $\ell$-weight $\Psi\in \ttt^{\times}_{\ell}$ of a highest $\ell$-weight module $V$ is uniquely determined and the associated $\ell$-weight space $V_{\Psi}$ is of dimension 1. In particular, a highest $\ell$-weight vector $v\in V$ is uniquely determined up to a scalar factor. Such a vector $v$ also satisfies $U_q^+(\bor) v = 0$ so that $V = \uqb v = U_q^-(\bor)v$ by the triangular decomposition given in the precedent subsection.

 We can define Verma modules for the notion of $\ell$-weight and show the following result as in the usual Kac-Moody theory (with the associated notion of weight).
\begin{thm}[{\cite[Proposition 3.4]{hj}}]\label{thm:PropSimpleOEx} Let $\Psi \in \ttt^{\times}_{\ell}$. Then there is a unique (up to isomorphism) simple $\uqb$-module $L(\Psi)$ of highest $\ell$-weight $\Psi$.
\end{thm}
A simple object $V$ of $\mathcal{O}$ is necessarily of highest $\ell$-weight. Indeed, we can choose $\mu \in \ttt^{\times}$ maximal in $P(V)$ and use the $\ell$-weight space decomposition of $V$ to find $\Psi \in \ttt^{\times}_{\ell}$ with $V_{\Psi}\neq 0$ and $\varpi(\Psi) = \mu$. As $\dim V_{\Psi} < \infty$, there is a $v \in V_{\Psi}$ that is a simultaneous eigenvector for all the $\phi_{i,r}^+$ with $i\in I$ and $r\geq 0$. (Remark that it suffices to find an eigenvector for finitely many $\phi_{i,m}^{+}$ since $\dim V_{\Psi} < \infty$. This is easily done as the operators $\phi_{i,r}^+-\Psi_{i,r}$ with $\Psi = (\Psi_{i,r})_{i\in I,r\geq 0}$ all act nilpotently on $V_{\Psi}$.) Such a vector $v\in V_{\Psi}$ is then a highest $\ell$-weight vector for $V$ since $e_iv \in V_{\mu \overline{\alpha_i}}= 0$ and $\uqb v = V$ (as $V$ is simple with $\mu$ maximal in $P(V)$). \par The above fact and the results of \cite{hj} imply the following important theorem. Set 
$$\mathfrak{r} = \{\Psi=(\Psi_i(z))_{i\in I}\in \ttt^{\times}_{\ell}\,|\, \Psi_i(z) \text{ is a rational function for all } i\in I\}.$$
\begin{thm}[{\cite[Proposition 3.10 and Theorem 3.11]{hj}}]\label{thm:SimplesO} A complete list of non-isomorphic simple objects of $\mathcal{O}$ is $\{L(\Psi)\}_{\Psi\in\mathfrak{r}}$. Also, if $V$ is in $\mathcal{O}$ with $V_{\Psi}\neq 0$ for a $\Psi\in \ttt^{\times}_{\ell}$, then $\Psi\in \mathfrak{r}$.
\end{thm}
\begin{rem}\label{rem:Psishift} For $\Psi = (\Psi_i(z))_{i\in I} \in \mathfrak{r}$ with $a \in \mathbb{C}^{\times}$, we define $\Psi(a)=(\Psi_i(az))_{i\in I}$. Then the pullback of $L(\Psi)$ by the automorphism $\tau_{a}$ of Section \ref{sec:DefQaff} is isomorphic to $L(\Psi(a))$.
\end{rem}
\begin{example}[{\cite[Sections 4 and 5]{hj}}]\label{ex:Lpmsl2} For $a \in \C^{\times}$ and $i\in I$, let $\Psi_{i,a} \in \mathfrak{r}$ denote the $\ell$-weight defined by
$$(\Psi_{i,a})_j(z) = \left\{\begin{array}{ll}
1-az & \text{if } i=j,\\
1 & \text{else.}
\end{array}\right.$$
The modules $L_{i,a}^{+}=L(\Psi_{i,a})$ and $L_{i,a}^- = L(\Psi_{i,a}^{-1})$ are called positive and negative prefundamental representations (resp.). They are of first and foremost importance in the study of Baxter $Q$-operators and of quantum integrable systems (see \cite{fh1}).\par For $\g = \su$, a realization of $L_{1,a}^{+}$ on the vector space with basis $\{w_j\}_{j\geq 0}$ is given via
$$k_1w_j = q^{-2j}w_j,\ e_1w_j = w_{j-1}\text{ and }(q-q^{-1})e_0w_j = -aq^{2+j}[j+1]_qw_{j+1}$$
with $w_{-1}=0$. This action verifies $x_{1,r}^+w_j = \delta_{r,0}w_{j-1}$, $(q-q^{-1})x_{1,m}^-w_j = -aq^{-j}[j+1]_q\delta_{r,1}w_{j+1}$ (for $r \geq 0$ and $m>0$) and $\phi_1^+(z)w_j = q^{-2j}(1-az)w_j$. \par An explicit realization of the negative prefundamental representation $L_{1,a}^{-}$ on the space with basis $\{z_j\}_{j\geq 0}$ is also given for $\g = \su$ via
$$k_1z_j = q^{-2j}z_j,\ e_1z_j = z_{j-1}\text{ and }(q-q^{-1})e_0z_j = aq^{2-j}[j+1]_qz_{j+1}$$
with $z_{-1}=0$. We have $x_{1,r}^+ z_j = a^r q^{2r(1-j)}z_{j-1}$ and $(q- q^{-1})x_{1,m}^-z_j = a^mq^{-(2m+1)j}[j+1]_qz_{j+1}$ (for $r\geq 0$ and $m>0$) with
\begin{center}
$\ds \phi_1^+(z)z_j = q^{-2j}\frac{(1-q^2 a z)}{(1-q^{2(1-j)}a z)(1-q^{-2j}a z)}z_j$.
\end{center} 
More details about the general construction of the modules $L_{i,a}^{\pm}$ are given in Section \ref{sec:qcaract}.
\end{example}
In addition to the remarkable prefundamental representations, the category $\mathcal{O}$ contains all finite-dimensional $\uqb$-modules (cf.~\cite{hj}). In particular, for all $\mu\in \ttt^{\times}$, the simple module $[\mu]$ with highest $\ell$-weight $\mu$ belongs to $\mathcal{O}$. (We associate $\mu = (\mu_i)_{i\in I}\in \ttt^{\times}$ with the element $\mu = (\mu_{i,r})_{i\in I,r\geq 0}\in\ttt_{\ell}^{\times}$ defined by $\mu_{i,r}= \mu_i\delta_{r,0}$. This association was also used in the definition of the constant part $\varpi(\Psi)$ and will be used again without mention.) The module $[\mu]$ is one-dimensional with trivial action of the elements $e_i$, $e_0$, $x_{i,r}^+$ and $x_{i,m}^-$ for $i\in I$, $r\geq 0$ and $m>0$. We will call these representations \textit{invertible} as they verify
$$ [\mu] \otimes [\mu^{-1}] \simeq [\mathbbm{1}]$$
where $\mathbbm{1}\in \ttt^{\times}$ is defined by $\mathbbm{1}=(1)_{i\in I}$.\par  
Let $a\in \C^{\times}$ and $i \in I$. We define the $\ell$-weight $Y_{i,a}\in \mathfrak{r}$ by $Y_{i,a} = \overline{\omega_i}\Psi_{i,a q_i^{-1}}\Psi_{i,aq_i}^{-1}$, that is 
\begin{center}
$(Y_{i,a})_j(z) = \left\{\begin{array}{ll}
q_i \frac{1-azq_i^{-1}}{1-azq_i} & \text{if } i=j,\\
1 & \text{else}.
\end{array}\right.$
\end{center}
The simple $L(\Psi)$ is finite-dimensional if $\Psi$ is a monomial in the $Y_{i,a}$'s and the $\uqb$-action on such a module can be uniquely extended to a $\uqg$-action. In fact, any (type I) simple finite-dimensional $\uqg$-module can be realized in this way (see e.g.~\cite{her,cp1,cp2} for these last results). We also have a similar characterization for simple finite-dimensional $\uqb$-modules:
\begin{prop}[{\cite[Remark 3.11]{fh1}}]\label{prop:CharacSimFD} Fix a simple finite-dimensional $\uqb$-module $L(\Psi)$. Then there is $\mu \in \ttt^{\times}$ and a monomial $\Psi'$ in the $Y_{i,a}$'s such that $\Psi=\mu\Psi'$.
\end{prop}
The following finite-dimensional modules are of capital importance in the study of $\uqg$.
\begin{example}[see e.g.~{\cite[Example 3.12]{fh1}}]\label{ex:KRsl2} For fixed $k\in \mathbb{N}_{>0}$, $i \in I$ and $a \in \C^{\times}$ we have the Kirillov--Reshitikhin module  
\begin{center}
$W_{k,a}^{(i)} = L(Y_{i,a}Y_{i,aq_i^2}\dots Y_{i,aq_i^{2(k-1)}})$
\end{center}\vspace*{-1mm}
and the fundamental representation $V_i(a) = W_{1,a}^{(i)} = L(Y_{i,a})$.\newpage These simple modules over $\uqb$ (and over $\uqg$) generalize the evaluation representations of \cite{cp1}. A realization of $W_{k,aq^{1-2k}}^{(1)}$ on the space with basis $\{v_j\}_{j=0}^k$ is given for $\g = \su$ by
$$k_1v_j = q^{k-2j}v_j, \quad e_1v_j = v_{j-1}, \quad e_0v_j = aq^{2-k}[j+1]_q[k-j]_qv_{j+1},$$
$$f_1v_j = [j+1]_q[k-j]_qv_{j+1} \ \text{  and }\  af_0v_j =q^{k-2}v_{j-1}$$
with $v_{-1} = v_{k+1} = 0$. We get $x_{1,r}^+v_j = a^rq^{2r(1-j)}v_{j-1}$ and $x_{1,r}^-v_j = a^rq^{-2jr}[j+1]_q[k-j]_qv_{j+1}$ for $r\in \mathbb{Z}$ as well as
$$\ds \phi_{1}^{\pm}(z)v_j= q^{k-2j}\frac{(1-azq^{-2k})(1-azq^2)}{(1-azq^{2(1-j)})(1-azq^{-2j})}v_j.$$
Note that there is no known explicit description of the modules $W_{k,a}^{(i)}$ for $\g$ of general type.
\end{example}
\begin{rem}[{\cite[Remark 3.6]{hj}}]\label{rem:L(PP')} For $\Psi,\Psi'\in \ttt^{\times}_{\ell}$, the submodule of $L(\Psi)\otimes L(\Psi')$ generated by a tensor product of highest $\ell$-weight vectors has highest $\ell$-weight $\Psi\Psi'$. The simple module $L(\Psi\Psi')$ is thus a subquotient of $L(\Psi)\otimes L(\Psi')$. This fact will be used in the next sections.
\end{rem}
\subsection{$q$-characters and prefundamental representations}\label{sec:qcaract} One of the principal tools in the study of the category $\mathcal{O}$ is the $q$-character map. To define it, let $\mathcal{E}_{\ell}$ be the ring of functions $c:\mathfrak{r}\rightarrow\Z$ satisfying the two conditions
\begin{itemize}
\item[(i)] $\{\varpi(\Psi)\,|\,\Psi\in \mathfrak{r},\, c(\Psi)\neq 0\}$ is contained in $\bigcup_{i=1}^r D(\lambda_i)$ for some $\lambda_1,\dots ,\lambda_r\in \mathfrak{t}^{\times}$ and
\item[(ii)] $\{\Psi\in \mathfrak{r}\,|\, c(\Psi)\neq 0,\, \varpi(\Psi)=\mu\}$ is a finite set for any $\mu \in \ttt^{\times}$
\end{itemize}
where the sets $D(\lambda)$ are the ones of Definition \ref{def:O}. We also let $\mathcal{E}$ be the ring of maps $d:\ttt^{\times}\rightarrow \Z$ such that $d(\mu) = 0$ if $\mu \in \ttt^{\times}$ is taken outside a finite union of sets $D(\lambda)$.  The ring structure for $\mathcal{E}_{\ell}$ and $\mathcal{E}$ are obtained via convolution, that is
\begin{center}
$\ds (c_1\cdot c_2)(\Psi) = \sum_{\Psi_1\Psi_2 = \Psi} c_1(\Psi_1)c_2(\Psi_2)$ and $\ds (d_1\cdot d_2)(\mu) = \sum_{\mu_1\mu_2 = \mu} d_1(\mu_1)d_2(\mu_2)$.
\end{center}
Fix $V$ in $\mathcal{O}$. Then the $q$-character $\chi_q(V)$ is the element of $\mathcal{E}_{\ell}$ defined by (cf.~\cite{fr,hj})
\begin{center}
$\ds \chi_q(V)= \sum_{\Psi\in \mathfrak{r}} \dim (V_{\Psi})[\Psi]$
\end{center}
where, for $\Psi \in \mathfrak{r}$, we wrote $[\Psi]$ for the map in $\mathcal{E}_{\ell}$ 
given by $[\Psi](\Psi') = \delta_{\Psi,\Psi'}$. This $q$-character can be seen as a generating function for the dimensions of the $\ell$-weight spaces of $V$. We may also define the ordinary character $\chi(V)$ as the generating function for the dimensions of the (ordinary) weight spaces of $V$, that is
$$\ds \chi(V)= \sum_{\mu \in \ttt^{\times}} \dim (V_{\mu}) [\mu]\in\mathcal{E}$$
where, for $\mu \in \ttt^{\times}$, we defined $[\mu]\in \mathcal{E}$ by $[\mu](\mu') = \delta_{\mu,\mu'}$. Note that the invertible representation $[\mu]$ of $\uqb$ satisfies $\chi([\mu]) = [\mu]\in \mathcal{E}$ so that this last notation should not lead to confusion.\par
If $V$ is of highest $\ell$-weight $\Psi$ (or more generally if $V$ has a unique $\ell$-weight $\Psi$ with $\varpi(\Psi)$ maximal in $P(V)$), we can also consider its normalized $q$-character $\overline{\chi}_q(V)$ given by
\begin{center}
$\overline{\chi}_q(V) = [\Psi^{-1}]\chi_q(V)$.
\end{center}
We have $\chi_q(V\otimes V') = \chi_q(V)\chi_q(V')$ with $\chi_q(V\oplus V') = \chi_q(V)+\chi_q(V')$ for any $V,V'$ in $\mathcal{O}$. This is in fact part of another result due to \cite{hj}.
\begin{thm}[{\cite[Proposition 3.12]{hj}} and {\cite[Theorem 3]{fr}}]\label{thm:qcaract} Taking $q$-characters induces an injective ring morphism $\chi_q : K_0(\mathcal{O}) \rightarrow \mathcal{E}_{\ell}$. In particular, $K_0(\mathcal{O})$ is a commutative ring by commutativity of the convolution product of $\mathcal{E}_{\ell}$.
\end{thm}
The $\ell$-weights of a finite-dimensional $\uqg$-module are monomials in the variables $(Y_{i,a})^{\pm 1}$'s (cf.~\cite{fr}) and the $q$-character of such a module may hence be seen as an element of the ring $\mathcal{Y}=\Z[(Y_{i,a})^{\pm 1}]_{i\in I,a\in \C^{\times}}$. We will therefore follow the convention of \cite{hj,fh1,hl2} and forget the brackets $[\cdot]$ when writing an element of $\mathcal{Y}$ in a $q$-character. (In other words, we associate $\Psi$ and $[\Psi] = \delta_{\Psi,\Psi'}$ in the image of the $q$-character map for any $\Psi \in \mathcal{Y}$.) \par We will moreover use 
$$\ds A_{i,a} = Y_{i,aq_i^{-1}}Y_{i,aq_i}\Bigg(\prod_{\{j\in I| C_{ji}=-1\}} Y_{j,a}\prod_{\{j\in I|C_{ji}=-2\}} Y_{j,aq^{-1}}Y_{j,aq}\prod_{\{j\in I|C_{ji}=-3\}} Y_{j,aq^{-2}}Y_{j,a}Y_{j,aq^2}\Bigg)^{-1}.$$
We have $\varpi(A_{i,a}) = \overline{\alpha_i}$ and $\varpi(Y_{i,a}) = \overline{\omega_i}$. Thus $A_{i,a}$ may be seen as the analog of a simple root in $\mathcal{Y}$. We also have the following result (which is used implicitly in Proposition \ref{prop:limitOm} below).
\begin{prop}[{\cite[Proposition 3]{fr}}] For $V$ a simple finite-dimensional $\uqg$-module $$\overline{\chi}_q(V)\in\Z[A_{i,a}^{-1}]_{i\in I,a\in \C^{\times}}.$$
\end{prop}
\begin{example}\label{ex:qcharW} Fix $\g = \su$ and consider $V = W_{k,aq^{1-2k}}^{(1)}$ where $a \in \C^{\times}$ and $k\geq 0$. Then, using Example \ref{ex:KRsl2} and the convention above, we get $\chi_q(V) = Y_{1,aq^{1-2k}}\dots Y_{1,aq^{-1}}\overline{\chi}_q(V)$ with
\begin{center}
$\overline{\chi}_q(V) = 1+ A_{1,a}^{-1} + A_{1,a}^{-1}A_{1,aq^{-2}}^{-1}+\dots +A_{1,a}^{-1}A_{1,aq^{-2}}^{-1}\dots A_{1,aq^{2(1-k)}}^{-1}$.
\end{center}
\end{example}
\begin{thm}[{\cite[Theorem 4.1]{fh1}}, {\cite[Theorem 5.3]{hl2}}]\label{prop:qcharprefond} Fix $a \in \C^{\times}$ and $i\in I$. Then 
\begin{center}
$\chi_q(L_{i,a}^+) = [\Psi_{i,a}]\chi(L_{i,a}^+)$ and $\chi_q(L_{i,a}^-) \in [\Psi_{i,a}^{-1}](1+A_{i,a}^{-1}\Z[[A_{j,b}^{-1}]]_{j\in I, b\in\C^{\times}})$
\end{center}
where the character $\chi(L_{i,a}^+)$ (seen as an element of $\mathcal{E}_{\ell}$ by the obvious inclusion $\mathcal{E}\subseteq \mathcal{E}_{\ell}$) is equal to the character $\chi(L_{i,a}^-)$ and does not depend on the parameter $a$.
\end{thm}
\begin{example}\label{ex:qcharLmsu} For $\g=\su$, Example \ref{ex:Lpmsl2} gives $\chi_q(L_{1,a}^+)=[\Psi_{1,a}](1-[\overline{\alpha_1}]^{-1})^{-1} = [\Psi_{1,a}]\chi(L_{1,a}^+)$ and $\chi_q(L_{1,a}^-) = [\Psi_{1,a}^{-1}](1+A_{1,a}^{-1}(1+\sum_{j\geq 1} A_{1,aq^{-2}}^{-1}\dots A_{1,aq^{-2j}}^{-1}))$. This agrees with Theorem \ref{prop:qcharprefond}.
\end{example}
We now recall the definition of the main categories of interest in this paper.
\begin{defn}[{\cite[Definition 3.9]{hl2}}]\label{def:Opm} An object $V$ of $\mathcal{O}$ belongs to $\mathcal{O}^+$ (resp. $\mathcal{O}^-$) if the image of any simple constituent of $V$ in 
 $K_0(\mathcal{O})$ is contained inside the subring generated by the classes of finite-dimensional 
 and positive (resp. negative)
 prefundamental representations.
\end{defn}
In other terms, a module $V$ in $\mathcal{O}$ belongs to $\mathcal{O}^+$ (resp. $\mathcal{O}^-$) if the highest $\ell$-weights of its simple constituents are all monomials in the following $\ell$-weights:
\begin{itemize}
\item[(i)] the $Y_{i,a}$'s with $i\in I$ and $a\in \C^{\times}$,
\item[(ii)] the $\Psi_{i,a}$'s (resp. the $(\Psi_{i,a}^{-1})$'s) with $i\in I$ and $a\in \C^{\times}$,
\item[(iii)] the $\mu$'s with $\mu \in \ttt^{\times}$.
\end{itemize}
The so-defined categories $\mathcal{O}^{\pm}$ are full monoidal subcategories of $\mathcal{O}$ which are especially interesting from the point of view of cluster algebras (see e.g.~\cite{hl2,bi} and the upcoming Remark \ref{rem:MonCatClust}). The $q$-character map is also remarkably well-behaved on the negative subcategory $\Om$.
\begin{prop}[see e.g.~{\cite[Theorem 7.6]{hl2}}]\label{prop:limitOm} Fix $V$ a simple object in $\Om$. Then, there are simple finite-dimensional $\uqg$-modules $(V_k)_{k\geq 1}$ with $\textstyle \lim_{k\rightarrow \infty} \overline{\chi}_q(V_k) = \overline{\chi}_q(V)$ as formal power series in $\Z[[A_{i,a}^{-1}]]_{i\in I,a\in \C^{\times}}$.
\end{prop}
In fact, with the above notation, we can define an inductive linear system
\begin{center}
$F_{k',k} : V_k \rightarrow V_{k'}$ \quad (with $k\leq k'$)
\end{center}
so that we have convergence of the action of the subalgebra $\widetilde{U}_q(\g)\subseteq \uqg$ generated by the set $\{x_{i,r}^+,k_i^{-1}x_{i,r}^-,k_i^{-1}\phi_{i,r}^{\pm}, k_i^{-1}\,|\, i\in I, r \in \Z\}$. (This subalgebra $\widetilde{U}_q(\g)$, defined in \cite{hj}, is called \textit{asymptotic algebra} and is strongly related to the shifted quantum affine algebras alluded to in Section \ref{sec:Intro}.)  After identifying the direct limit of this system with $V$, we recover the associated $\uqb$-action up to a choice for the eigenvalue of $k_1,\dots ,k_n$ on the highest $\ell$-weight vector of $V$. (This follows from the fact that $e_0,\dots ,e_n \in \widetilde{U}_q(\g)$, see \cite{hj}.) In other terms, every simple object of $\Om$ is the direct limit of a system of simple finite-dimensional $\uqg$-modules. \par
The maps $F_{k',k}$ defining the inductive system above are introduced in \cite{hj} (see also \cite{hl2}). To recall their construction, write $\Psi_k$ for the highest $\ell$-weight of the simple finite-dimensional module $V_k$. Then, for $k\leq k'$, $F_{k',k}$ is obtained as the composition
$$ V_k = L(\Psi_k) \arr L(\Psi_k)\otimes L(\Psi_{k'}\Psi_k^{-1}) \arr L(\Psi_{k'}) = V_{k'}$$
where the first map is defined by $v \mapsto v\otimes v_0$ for a fixed highest $\ell$-weight vector $v_0 \in L(\Psi_{k'}\Psi_k^{-1})$. The second map is deduced from Remark \ref{rem:L(PP')} and from a notable cyclicity property of tensor products of finite-dimensional fundamental representations
(see e.g.~\cite[Equation 3.21]{hj} and the references  therein).
\begin{rem}\label{rem:LnegW}
The above inductive system is studied in \cite{hj} for 
$ V = L_{i,a}^-$ and $V_k = W_{k,aq_i^{1-2k}}^{(i)} $. It is notably simple for $\g = \su$ as Proposition \ref{prop:limitOm} then follows from Examples \ref{ex:qcharW} and \ref{ex:qcharLmsu}. 
\end{rem}
Although the subcategories $\mathcal{O}^+$ and $\mathcal{O}^-$ are defined in a similar way, they are (as explained in Section \ref{sec:Intro}) instrinsically different from a representation-theoretic perspective. For example, Theorem \ref{prop:qcharprefond} shows that the normalized $q$-character of a negative prefundamental representation $L_{i,a}^-$ depends on the parameter $a\in \C^{\times}$ whereas this is not the case for the normalized $q$-character of $L_{i,a}^+$. This seemingly non-critical fact has deep consequences regarding the analysis done above and makes it in particular hard to extend Theorem \ref{prop:qcharprefond} to the category $\mathcal{O}^+$. (The module $L_{i,a}^+$ is also related, under technical conditions, to a limit of Kirillov--Reshitikhin modules, see \cite{hj}. However, the normalized $q$-characters of the latter modules depend on the spectral parameter $a$ and the limit underlying an analog of Theorem \ref{prop:qcharprefond} for $\mathcal{O}^+$ would thus need to erase this dependance. Such a limit cannot then be as naive as the one used for $\mathcal{O}^-$.) 
\par
Another symptom of the difference between the subcategories $\mathcal{O}^{+}$ and $\mathcal{O}^-$ can be observed when studying the \textit{shifted quantum affine algebras} mentioned in Section \ref{sec:Intro}. Indeed, the action of $\uqb$ on $L_{i,a}^-$ can be extended to the action of such a shifted algebra even though this is not possible for $L_{i,a}^+$ (cf.~\cite[Section 4]{hshift}). In fact, as we have already stated in Section \ref{sec:Intro}, the analog of positive prefundamental representations for shifted algebras are one-dimensional and are thus extremely far from being as complicated as the infinite-dimensional $\uqb$-modules $L_{i,a}^+$.
\par
Nevertheless, even with these differences, the subcategories $\mathcal{O}^{\pm}$ can be related on the level of Grothendieck rings. Indeed, it is shown in \cite[Sections 5-7]{hl2} that there is a ring isomorphism $$ D : K_0(\mathcal{O}^+) \simeq K_0(\mathcal{O}^-) $$ 
that sends classes of simple modules to classes of simple modules and is given by
$$ D[V_i(a)] = [V_i(a^{-1})] \quad \text{and}\quad D[L_{i,a}^+] = [L_{i,a^{-1}}^-]$$
on the classes of finite-dimensional fundamental and positive prefundamental representations (resp.). More precisely, let $V=L(\Psi_1)$ be a finite-dimensional simple $\uqb$-module. Then $$\textstyle \Psi_1= \mu\prod_{i\in I}\prod_{k=1}^{s_i} Y_{i,a_k^{(i)}}$$ for some $\mu\in \C^{\times}$ by Proposition \ref{prop:CharacSimFD} and $D[L(\Psi_1)] = [L(\Psi_2)]$ where $$ \textstyle \Psi_2 = \mu\prod_{i\in I}\prod_{k=1}^{s_i} Y_{i,(a_k^{(i)})^{-1}} $$
is obtained from $\Psi_1$ by mapping each variable $Y_{i,a}$ to $Y_{i,a^{-1}}$. In particular, $D$ sends classes of \textbf{finite-dimensional} $\uqb$-modules to classes of \textbf{finite-dimensional} $\uqb$-modules. We will call this map $D:K_0(\mathcal{O}^+)\simeq K_0(\mathcal{O}^-)$ \textit{Hernandez--Leclerc's duality}.
\begin{rem} 
 A natural question regarding the isomorphism $D$ is whether or not this duality extends to an automorphism of the Grothendieck ring $K_0(\mathcal{O})$ of the whole category $\mathcal{O}$. More elaborately, one could ask whether or not there is a autofunctor $\mathcal{D}_q$ of $\mathcal{O}$ with $D[M] = [\mathcal{D}_q(M)]$ for any $M$ in $\mathcal{O}$. These questions are partially answered in the next pages. Indeed we define in Section \ref{sec:Fq} a functor $\F_q$ relating the category $\mathcal{O}$ corresponding to $\uqgm$ to the one associated to $\uqg$ and consider (assuming that $q$ is a formal variable, see Section \ref{sec:FuncD}) the composition of $\F_q$ with a natural functor $\mathcal{B}_{q^{-1},q}$. The resulting composition $\mathcal{D}_q$ is an exact autofunctor of $\mathcal{O}$ and induces a ring automorphism of $K_0(\mathcal{O})$ which reproduces the duality $D$ on an interesting subring of $K_0(\mathcal{O}^{+})$ (see Theorem \ref{thm:DCat}). Moreover, letting aside the functorial interpretation of Hernandez--Leclerc's duality, we prove in Appendix \ref{app:K0D} that the map $D$ can be extended to a ring automorphism of all of the Grothendieck ring $K_0(\mathcal{O})$. This result does not tell however if the extended automorphism sends classes of simple modules to classes of simple modules (as does $D$ and the map induced from the functor $\mathcal{D}_q$). Our result also does not tell whether or not the extended map can be induced from a well-chosen autofunctor of $\mathcal{O}_q$.
\end{rem}
We now end this section by stating some important results about tensor products in $\mathcal{O}$. 
\begin{theorem}[{\cite[Theorem 4.11]{fh1}}]\label{thm:ProdPrefond}
Any tensor product of prefundamental representations of the same sign (that is positive or negative) is simple. 
\end{theorem}
The following theorem is well-known (see \cite[Theorem 3]{ch} or \cite[Propositions 9.2.4 and 9.5.3]{efk}). Note that we intrinsically use Proposition \ref{prop:CharacSimFD} in order to state it for $\uqb$-modules.
\begin{theorem}[see e.g.~\cite{ch,efk}]
\label{thm:genSim} Let $V$ and $W$ be irreducible finite-dimensional modules over $\uqb$ (or $\uqg$). Then $V(a)\otimes W$ is reducible for at most finitely many $a \in \C^{\times}$. 
\end{theorem}
Tensor products of simple finite-dimensional $\uqb$-modules are hence ``generically'' simple. The corresponding statement is nevertheless not true for the whole category $\mathcal{O}$ as, for example, there exists no $a\in\C^{\times}$ making $L_{1,1}^+(a)\otimes L_{1,1}^-$ simple when $\g = \su$ (see e.g.~\cite[Example 5.1]{her}, \cite{bjmst} or the upcoming Lemma \ref{lem:ProdTensLpm}).
\section{The functor $\F_q$}\label{sec:Fq}
This section is devoted to the study of the covariant functor $\F_q:\mathcal{O}_{q^{-1}}\arr \mathcal{O}_{q}$ alluded to in Section \ref{sec:Intro} (with $\mathcal{O}_{q^{\pm 1}}$ the category $\mathcal{O}$ corresponding to the quantum loop algebra $\uqgpm$). We prove in particular that this functor reverses tensor products and that it changes the sign of prefundamental representations. We also
 deduce the image $\F_q(V)$ of any object $V$ of $\mathcal{O}_{q^{-1}}$ as well as new relations in the ring $K_0(\mathcal{O}_q)$ and a functorial interpretation of the 
duality $D$. \par
As we will work simultaneously with the quantum loop algebras $\uqg$ and $\uqgm$, we will need to distinguish, using different notations, the concepts related to each of these algebras. This was in particular done above with the specification of the subscript $q^{\pm 1}$ on the category $\mathcal{O}_{q^{\pm 1}}$ of Section \ref{sec:O}. We will use the convention that a concept always refers to $\uqg$ unless explicitly specified (with notation or comments). Here are examples of these  ``new'' notations. \par We denote by $\overline{\overline{\textcolor{white}{\alpha}}} : P_{\mathbb{Q}}\arr \ttt^{\times}$ the injective morphism of groups obtained from the one $\overline{\textcolor{white}{\alpha}}$ of Section \ref{sec:O} by changing the quantum parameter to $q^{-1}$. This map induces (as in Section \ref{sec:O}) a partial order $\unlhd$ on $\ttt^{\times}$ with $\mu\unlhd \nu$ if and only if $\nu\mu^{-1}\in \overline{\overline{Q_+}}$. Clearly, $\overline{\overline{\omega}}=\overline{\omega}^{-1}$ for all $\omega \in P_{\mathbb{Q}}$ and it follows that the partial orders $\leq$ and $\unlhd$ are reciprocal (i.e.~$\mu\leq \nu$ if and only if $\nu\unlhd \mu$).
\subsection{First properties}\label{sec:Comon} The defining relations of $\uqgpm$ imply that the correspondence $$e_i\mapsto e_i, \qquad f_i\mapsto f_i, \quad \text{ and } \quad k_i^{\pm 1}\mapsto k_i^{\pm 1} \qquad (0\leq i\leq n)$$ induces an algebra anti-isomorphism $\Xi_q:\uqg\arr \uqgm$. 
Let us define $$\pp_q = S_{q^{-1}}\circ \Xi_q:\uqg\arr \uqgm$$
where $S_{q^{-1}}$ is the antipode of $\uqgm$ given in Section \ref{sec:DefQaff}. Then $\pp_q$ is an algebra isomorphism (since $S_{q^{-1}}$ is an algebra anti-isomorphism) and verifies
$$ \pp_q(e_i) = -k_i^{-1}e_i, \quad \pp_q(f_i) = -f_ik_i \ \text{ and } \ \pp_q(k_i)=k_i^{-1}$$
for $0\leq i\leq n$. This is the map appearing in Section \ref{sec:Intro}.\par
Note that $\pp_q(\uqb) \subseteq \uqbm$. It hence follows that the pullback $\pp_q^*$ induces a covariant functor from the category of all $\uqbm$-module to the category of all $\uqb$-modules. Write $\F_q$ for the restriction of this functor to the category $\mathcal{O}_{q^{-1}}$. Then $\F_q(V) = V$ as vector spaces for all $V$ in $\mathcal{O}_{q^{-1}}$. We will use the symbol $\star$ to distinguish the $\uqbm$-action on $V$ from the $\uqb$-action on $\F_q(V)$ (which is typically written using concatenation).
\begin{prop}\label{prop:FO} Let $V$ be in $\mathcal{O}_{q^{-1}}$. Then $\F_q(V)$ is in $\mathcal{O}_q$.
\end{prop}
\begin{proof} Remark that the weight spaces $V_{\mu^{-1}}$ and $\F_q(V)_{\mu}$ are naturally identified for $\mu\in \ttt^{\times}$ as $\pp_q(k_i) = k_i^{-1}$. The module $\F_q(V)$ is hence $\ttt$-diagonalizable and $\dim \F_q(V)_{\mu} = \dim V_{\mu^{-1}} < \infty$ whenever $\mu \in \ttt^{\times}$. It therefore suffices to show that the weights of $\F_q(V)$ are appropriately bounded above with respect to the order $\leq$ on $\ttt^{\times}$. \par For this, note that the fact that $V$ is in $\mathcal{O}_{q^{-1}}$ forces the set $P(V)\subseteq \ttt^{\times}$ to be appropriately bounded above for the reciprocal 
order $\unlhd$. We can thus find some $\lambda_1,\dots ,\lambda_s \in \ttt^{\times}$ such that $$\textstyle P(V)\subseteq \bigcup_{i=1}^s \{\mu\in \ttt^{\times}\,|\,\mu\unlhd \lambda_i\}=\bigcup_{i=1}^s\{\mu\in \ttt^{\times}\,|\, \lambda_i\leq \mu\}$$  and the above reasoning gives 
$$ \textstyle P(\F_q(V)) = \{\mu\in \ttt^{\times}\,|\, \mu^{-1}\in P(V)\} \subseteq \bigcup_{i=1}^s\{\mu \in \ttt^{\times}\,|\, \mu\leq \lambda_i^{-1}\}= \bigcup_{i=1}^s D(\lambda_i^{-1})$$
as wanted.
\end{proof}
We can hence consider $\F_q$ as a functor from $\mathcal{O}_{q^{-1}}$ to $\mathcal{O}_q$. It is trivially exact and preserves the dimension of modules since it is defined as a restriction of a pullback. Furthermore, as $\pp_q^{-1} = \pp_{q^{-1}}$, we have that $\F_q$ is invertible with inverse $\F_{q^{-1}}:\mathcal{O}_q\rightarrow \mathcal{O}_{q^{-1}}$. It then follows in particular that $\F_q$ preserves the irreducibility of modules. It is also compatible with spectral parameter shifts and tensor products in the sense of the following two results.\newpage
\begin{prop}\label{prop:Fshift} Fix $a\in \C^{\times}$ and let $V$ be in $\mathcal{O}_{q^{-1}}$. Then  $\F_q(V(a)) \simeq (\F_q(V))(a)$. 
\end{prop}
\begin{proof}
This follows from the fact that $\pp_q$ is compatible with the $\Z$-grading of $\uqgpm$ (that is $\tau_{a,q^{-1}}\circ \pp_q = \pp_q\circ \tau_{a,q}$ with $\tau_{a,q^{\pm 1}}$ the automorphism of $\uqgpm$ given in Section \ref{sec:DefQaff}).
\end{proof}
\begin{rem} The above proposition gives in particular $\F_q(V(q)) \simeq (\F_q(V))(q)$ which is somewhat counter-intuitive as $q$ does not ``play the same role'' for the algebras $\uqg$ and $\uqgm$.
\end{rem}
\begin{theorem}\label{thm:Comon} Let $V,W$ be in $\mathcal{O}_{q^{-1}}$ and consider the usual ``flip'' $\tau : V\otimes W \rightarrow W\otimes V$ given by $\tau(v\otimes w)=w\otimes v$. Then $\tau$ induces a $\uqb$-linear isomorphism $\F_q(V)\otimes \F_q(W)\arr \F_q(W\otimes V)$.
\end{theorem}
\begin{proof} 
It suffices to note that the antipode $S_{q^{-1}}$ of $\uqgm$ and the map $\Xi_q:\uqg\arr\uqgm$ introduced at the beginning of this section are respectively a coalgebra anti-isomorphism and a coalgebra isomorphism.
\end{proof}
Denote by $L_{q^{\pm 1}}(\Psi)$ the irreducible $U_{q^{\pm 1}}(\bor)$-module of highest $\ell$-weight $\Psi\in \ttt^{\times}_{\ell}$ and by $Y_{i,a}^{(q^{\pm 1})}$ the highest $\ell$-weight of the fundamental representation $V_i^{q^{\pm 1}}(a)$ of $U_{q^{\pm 1}}(\bor)$ (cf.~Example \ref{ex:KRsl2}). 
\subsection{Image of simple modules}\label{sec:ImSimple} Fix $\Psi\in \mathfrak{r}$. Then the image of the simple object $L_{q^{-1}}(\Psi)$ of $\mathcal{O}_{q^{-1}}$ through the functor $\F_q$ is a simple object of $\mathcal{O}_q$. We thus have
$$ \F_q (L_{q^{-1}}(\Psi)) \simeq L_q(\Psi^{\F_q})$$
for some (unique) $\Psi^{\F_q}\in \mathfrak{r}$. The goal of this subsection is to characterize the map $\Psi\mapsto \Psi^{\F_q}$. A naive strategy to do this would be to compute the action of the elements $\pp_q(\phi_{i,r}^+)$ (for $i\in I$ and $r\geq 0$) on a fixed highest $\ell$-weight vector $v$ of $L_{q^{-1}}(\Psi)$. (Remark that $v$ is a good candidate for being a highest $\ell$-weight vector for $ \F_q (L_{q^{-1}}(\Psi)) \simeq L_q(\Psi^{\F_q})$ since $\pp_q(e_i)\star v = -k_i^{-1}e_i\star v = 0$ for all $i\in I$.) It seems however difficult (if possible) to obtain an expression of these $\pp_q(\phi_{i,r}^+)$ in terms of the Drinfeld generators of $U_{q^{-1}}(\g)$ and our naive strategy is thus hardly applicable. Let us proceed differently and show first that the map $\Psi\mapsto \Psi^{\F_q}$ behaves well with respect to the structure of the subgroup $\mathfrak{r}\subseteq \ttt_{\ell}^{\times}$.\par 
We will use the following lemma which ties the constant parts of $\Psi$ and $\Psi^{\F_q}$. This lemma relies on the fact that $\mu < \varpi(\Psi)$ for any weight $\mu\neq\varpi(\Psi)$ of a module of highest $\ell$-weight $\Psi$. The latter fact can be proven using the triangular decomposition of $\uqb$ (see Section \ref{sec:DefQaff}).
\begin{lem}\label{lemma:varpiPsiF} Let $\Psi\in \mathfrak{r}$. Then $\varpi(\Psi^{\F_q})=(\varpi(\Psi))^{-1}$. In particular, $\mu^{\F_q} = \mu^{-1}$ for $\mu\in\ttt^{\times}$.
\end{lem}
\begin{proof} Since $\varpi(\Psi)\in P(L_{q^{-1}}(\Psi)) \subseteq \{\mu\in \ttt^{\times}\,|\,\mu \unlhd \varpi(\Psi)\}$, the proof of Proposition \ref{prop:FO} gives 
\begin{align*}
(\varpi(\Psi))^{-1} &\in P(L_q(\Psi^{\F_q}))
=\{\mu\in \ttt^{\times}\,|\,\mu^{-1}\in P(L_{q^{-1}}(\Psi))\} \subseteq  D((\varpi(\Psi))^{-1}).
\end{align*}
However $\varpi(\Psi^{\F_q}) \in P(L_q(\Psi^{\F_q}))\subseteq D(\varpi(\Psi^{\F_q}))$ and the lemma follows from the above fact.
\end{proof}
\begin{rem} Let $\Psi \in \mathfrak{r}$ satisfy $\varpi(\Psi) = \overline{\overline{\omega}}$ with $\omega\in P_{\mathbb{Q}}$. Then $\varpi(\Psi^{\F_q})=(\varpi(\Psi))^{-1} = \overline{\overline{\omega}}^{-1} = \overline{\omega}$ and the map $\Psi \mapsto \Psi^{\F_q}$ fixes the weight lattice $P_{\mathbb{Q}}$ (even if it inverts the elements of $\ttt^{\times}$).
\end{rem}
\begin{prop}\label{prop:Fgroup} The map $\Psi\mapsto \Psi^{\F_q}$ is a group automorphism of $\mathfrak{r}$.
\end{prop}
\begin{proof}
Take $\Psi,\Psi'\in\mathfrak{r}$ with $V=L_{q^{-1}}(\Psi)$ and $W= L_{q^{-1}}(\Psi')$ in $\mathcal{O}_{q^{-1}}$. Set also $\Phi = \Psi^{\F_q}(\Psi')^{\F_q}$ and write $\F = \F_q$ with $\mathcal{G}=\F_{q^{-1}}:\mathcal{O}_q\rightarrow \mathcal{O}_{q^{-1}}$. Then $\mathcal{G}$ is the inverse of $\F$ and is an exact functor which reverses tensor products by precedent results. We can hence use Remark \ref{rem:L(PP')} to deduce that $L_{q^{-1}}(\Phi^{\mathcal{G}}) = \mathcal{G}(L_q(\Phi))$ is a composition factor of $\mathcal{G}(L_q((\Psi')^{\F})\otimes L_q(\Psi^{\F}))\simeq V\otimes W$.\par It follows that $\Phi^{\mathcal{G}}$ occurs in the $q$-character $\chi_q(V\otimes W)=\chi_q(V)\chi_q(W)$ and must accordingly be of the form $\Phi^{\mathcal{G}} = \Psi_1\Psi_2$ for some $\ell$-weights $\Psi_1$ of $V$ and $\Psi_2$ of $W$. Lemma \ref{lemma:varpiPsiF} also gives
\begin{equation}\label{eq:varpiF}
\varpi(\Psi_1)\varpi(\Psi_2)=\varpi(\Phi^{\mathcal{G}}) = (\varpi(\Phi))^{-1} = (\varpi(\Psi^{\F})\varpi((\Psi')^{\F}))^{-1} = \varpi(\Psi)\varpi(\Psi').
\end{equation}
Suppose $\Psi_1 \neq \Psi$. Then the fact discussed above gives $\varpi(\Psi_1)<\varpi(\Psi)$ with $\varpi(\Psi_2)\leq \varpi(\Psi')$ and the resulting inequality $\varpi(\Psi_1)\varpi(\Psi_2)< \varpi(\Psi)\varpi(\Psi')$ contradicts \eqref{eq:varpiF}. Therefore $\Psi_1 = \Psi$. Similarly, $\Psi_2 = \Psi'$ and $\Psi^{\F}(\Psi')^{\F}=\Phi = (\Phi^{\mathcal{G}})^{\F} = (\Psi_1\Psi_2)^{\F} = (\Psi\Psi')^{\F}$ as wanted.
\end{proof}
Since the group $\mathfrak{r}$ is generated by $\ttt^{\times}$ and the highest $\ell$-weights of prefundamental representations, Lemma \ref{lemma:varpiPsiF} and Proposition \ref{prop:Fgroup} show that it is enough to characterize the $\ell$-weights $(\Psi_{i,a}^{\pm 1})^{\F_q}$ for $i\in I$ and $a \in \C^{\times}$ in order to characterize the action of $\F_q$ on $\mathcal{O}_{q^{-1}}$. For this, it is useful to study first 
 the image $\F_q(V)$ for $V$ a (finite-dimensional)
 fundamental representation of $\uqbm$ (see Example \ref{ex:KRsl2}). This relies on the notion of lowest $\ell$-weight.
\begin{defn}[{\cite[Section 3.6]{hj}}] Fix $\Psi = (\Psi_{i,r})_{i\in I,r\geq 0}\in \ttt^{\times}_{\ell}$ with $V$ a $\uqb$-module. Then $V$ has lowest $\ell$-weight $\Psi$ if there is a non-zero $v \in V$ satisfying $V = \uqb v$ with $U_q^-(\bor)v = 0$ as well as $\phi_{i,r}^+ v = \Psi_{i,r}v$ for all $i \in I$ and $r \geq 0$ (with $U_q^-(\bor)$ as in the end of Section \ref{sec:DefQaff}). 
\end{defn}
Results that hold for the highest $\ell$-weight modules $L_q(\Psi)$ typically also hold for their lowest $\ell$-weight counterpart. In particular, we have $\mu > \varpi(\Psi)$ for any weight $\mu\neq \varpi(\Psi)$ of a simple $\uqb$-module $V$ of \textbf{lowest $\ell$-weight} $\Psi$. This can be applied to finite-dimensional irreducible representations of $\uqb$. Indeed, the highest $\ell$-weight of such a representation has the form $$\Psi = {\textstyle \mu\prod_{i\in I}\prod_{k=1}^{s_i} Y^{(q)}_{i,a_k^{(i)}}}$$ with $\mu \in \ttt^{\times}$ (cf.~Proposition \ref{prop:CharacSimFD}) and the results of \cite{fm} imply that $L_q(\Psi)$ is of lowest $\ell$-weight
$$ \textstyle \Psi'=\mu\prod_{i\in I}\prod_{k=1}^{s_i} \Big(Y^{(q)}_{i^*,a_k^{(i)}q^{r^{\vee}h^{\vee}}}\Big)^{-1} $$
where 
\begin{itemize}[leftmargin = 4.2mm]
\item[--] $h^{\vee}$ is the dual Coxeter number of the finite-dimensional simple 
Lie algebra $\dot{\g}$ underlying $\g$,
\item[--] $r^{\vee}$ is the maximal number of edges connecting two vertices in the Dynkin diagram of $\dot{\g}$ and 
\item[--] $i^*\in I$ is defined by $w_0(\alpha_i)=-\alpha_{i^*}$ for $w_0$ the longest element in the Weyl group of $\dot{\g}$.
\end{itemize}
We will need another lemma (remark the similarity with Theorem \ref{prop:qcharprefond}). 
\begin{lem}\label{lem:weightsFond} Let $i\in I$. Then $\overline{\omega_i-\alpha_i}\in \ttt^{\times}$ is a weight of the fundamental representation $V=V_{i}^q(1)$ of $\uqb$ and the set $P(V)$ of weights for $V$ is contained in $\{\overline{\omega_i}\}\cup\overline{\omega_i-\alpha_i-Q^+}$.
\end{lem}
\begin{proof}
Suppose $\overline{\omega_i-\alpha_i}\not\in P(V)$ and take $v\in V$ a highest $\ell$-weight vector. Then $e_iv=0$ and $k_iv = q_iv$ as $\varpi(Y_{i,1}^{(q)}) = \overline{\omega_i}$. The induced $\uqg$-action on $V$ also verifies $f_i v \in V_{\overline{\omega_i-\alpha_i}}$ so that $f_iv = 0$ by hypothesis. Thus, by the defining relations of $\uqg$,
$$ \textstyle v = \left(\frac{k_i-k_i^{-1}}{q_i-q_i^{-1}}\right)v = [e_i,f_i]v = 0$$
but this implies that $v$ is not a highest $\ell$-weight vector. Hence $\overline{\omega_i-\alpha_i}\in P(V)$. The second assertion follows directly from \cite[Corollary 2.13]{h3}. 
\end{proof}
\begin{theorem}\label{thm:FondF} Fix $i\in I$ and set $\Psi = Y_{i,1}^{(q^{-1})}$. Then there is $\gamma_{i} \in \C^{\times}$ such that $\Psi^{\F_q} = Y_{i,\gamma_i}^{(q)}$.
\end{theorem}
\begin{rem} This is compatible with Lemma \ref{lemma:varpiPsiF} as $\varpi(Y_{i,\gamma_i}^{(q)}) = \overline{\omega_i}= (\overline{\overline{\omega_i}})^{-1} = (\varpi(Y_{i,1}^{(q^{-1})}))^{-1}$.
\end{rem}
\begin{proof}
Fix $V = L_{q^{-1}}(\Psi)$ in $\mathcal{O}_{q^{-1}}$ and note that, as $\dim \F_q(V) = \dim V < \infty$, the highest and lowest $\ell$-weights of $\F_q(V)$ must respectively have the form
$$\Psi^{\F_q} = \mu\textstyle \prod_{j\in I}\prod_{k=1}^{s_j} Y_{j,a_k^{(j)}}^{(q)} \quad \text{and} \quad \textstyle (\Psi^{\F_q})' = \mu \prod_{j\in I}\prod_{k=1}^{s_j} \Big(Y_{j^*,a_k^{(j)}q^{r^{\vee}h^{\vee}}}^{(q)}\Big)^{-1}$$
for some
 $\mu\in \ttt^{\times}$.
By Lemma \ref{lemma:varpiPsiF} (and the analogous result for simple lowest $\ell$-weight modules) these $\ell$-weights must satisfy the relations $\varpi(\Psi^{\F_q}) = (\varpi(\Psi))^{-1}$ and $\varpi((\Psi^{\F_q})') = (\varpi(\Psi'))^{-1}$ where $\Psi' = \Big(Y_{i^*,q^{-r^{\vee}h^{\vee}}}^{(q^{-1})}\Big)^{-1}$ is the lowest $\ell$-weight of $V$. Hence, as $\varpi(Y_{j,a}^{(q^{\pm 1})}) = \overline{\pm\omega_j}$, we have 
$$ \frac{\varpi(\Psi^{\F_q})}{\varpi((\Psi^{\F_q})')} = \overline{\omega} = \overline{\omega_i+\omega_{i^{*}}}=\frac{\varpi(\Psi')}{\varpi(\Psi)}$$
for $\omega = \sum_{j\in I}s_j(\omega_j+\omega_{j^*})$. However the group morphism $\overline{\textcolor{white}{\alpha}}:P_{\mathbb{Q}}\rightarrow \ttt^{\times}$ is injective and the fundamental weights are free in the weight lattice of $\dot{\g}$. There are thus only two cases left.
\begin{itemize}
\item[(Case 1)] Suppose $s_j = \delta_{i,j}$ for all $j\in I$. Then $\Psi^{\F_q} = \mu Y_{i,\gamma_i}^{(q)}$ for $\gamma_i = a_1^{(i)}\in \C^{\times}$ and the result follows from the relation $\varpi(\Psi^{\F_q}) = \varpi(\mu Y_{i,\gamma_i}^{(q)}) = \mu \overline{\omega_i} = \overline{\omega_i} = (\varpi(\Psi))^{-1}$.
\item[(Case 2)] Suppose $i\neq i^*$ with $s_j = \delta_{i^*,j}$ for all $j\in I$. Then $\Psi^{\F_q} = \mu Y_{i^*,a}^{(q)}$ for some $a\in \C^{\times}$ and $\varpi(\Psi^{\F_q}) = \mu \overline{\omega_{i^*}} = \overline{\omega_i} = \varpi(\Psi)^{-1}$ implies $\mu = \overline{\omega_i-\omega_{i^*}}$. Furthermore Lemma \ref{lem:weightsFond} (for the parameter $q^{-1}$) shows that $\overline{\overline{\omega_i-\alpha_i}} = \overline{\alpha_i-\omega_i} \in P(V)$. The inverse weight $\overline{\omega_i-\alpha_i}$ thus belongs to $P(\F_q(V))$ (see the proof of Proposition \ref{prop:FO}), but, by Lemma \ref{lem:weightsFond}, 
$$ P(\F_q(V)) = P(L_q(\Psi^{\F_q})) = \mu P(L_q(Y_{i^*,a}^{(q)})) \subseteq \{\overline{\omega_i}\}\cup \overline{\omega_i-\alpha_{i^*}-Q^+} $$
where we have used the relation $\mu = \overline{\omega_i-\omega_{i^*}}$. The injectivity of the group morphism $\overline{\textcolor{white}{\alpha}}:P_{\mathbb{Q}}\rightarrow \ttt^{\times}$ then gives $\alpha_i-\alpha_{i^*} \in Q^+$ but this is impossible as $i\neq i^*$ by hypothesis.
\end{itemize}
This concludes the proof since only the first case remains. 
\end{proof}
The following theorem uses again the shift $\gamma_i\in \C^{\times}$ appearing in Proposition \ref{thm:FondF}. Remark that this shift could \textit{a priori} depend on the choice of $i\in I$. 
\begin{theorem}\label{thm:preFondF}
Fix $i\in I$ and set $\Psi = \Psi_{i,1}$. Then $\Psi^{\F_q} = \Psi_{i,\gamma_i}^{-1}$.
\end{theorem}
\begin{proof} Applying $\F_q$ on the relation $\Psi =\overline{\omega_i}Y_{i,q_i^{-1}}^{(q^{-1})}(\Psi(q_i^{-2}))$ of $\mathfrak{r}$ and using Theorem \ref{thm:FondF} with Proposition \ref{prop:Fshift}, Proposition \ref{prop:Fgroup} and Lemma \ref{lemma:varpiPsiF} give
\begin{equation}\label{eq:PsiPrefondFond}
\Psi^{\F_q} = \overline{-\omega_i}Y_{i,q_i^{-1}\gamma_i}^{(q)}(\Psi^{\F_q}(q_i^{-2})).
\end{equation}
In particular, the rational functions $\Psi^{\F_q}_j(z)$ for $j\in I$ with $i\neq j$ satisfy $\Psi^{\F_q}_j(z) = \Psi^{\F_q}_j(zq_i^{-2})$ and must therefore be constant. Moreover Equation \eqref{eq:PsiPrefondFond} at $j=i$ gives
$$ (1-\gamma_iz)\Psi_i^{\F_q}(z) = (1-\gamma_i zq_i^{-2})\Psi_i^{\F_q}(zq_i^{-2})$$
and the rational function $(1-\gamma_iz)\Psi_i^{\F_q}(z)$ must hence also be constant. It thus follows that $\Psi^{\F_q} = \mu\Psi_{i,\gamma_i}^{-1}$ for some $\mu\in \ttt^{\times}$ but $\mu = \varpi(\Psi^{\F_q}) = (\varpi(\Psi))^{-1} = 1$ by Lemma \ref{lemma:varpiPsiF}. 
\end{proof}
\begin{cor} The functor $\F_q$ maps objects of $\mathcal{O}_{q^{-1}}^+$ (resp.~$\mathcal{O}_{q^{-1}}^-$) to objects of $\mathcal{O}_{q}^-$ (resp.~$\mathcal{O}_{q}^+$).
\end{cor}
\begin{proof}
This directly follows from Definition \ref{def:Opm}, Lemma \ref{lemma:varpiPsiF}, Proposition \ref{prop:Fgroup}, Theorem \ref{thm:FondF} and Theorem \ref{thm:preFondF} (with the fact that $\F_q$ is exact and preserves irreducibility of modules).
\end{proof}
We now prove that the shift $\gamma_i\in \C^{\times}$ appearing in $\Psi_{i,1}^{\F_q}=\Psi_{i,\gamma_i}^{-1}$ does not depend on $i\in I$. For this goal, note that $[V]\mapsto [\F_q(V)]$ gives a well-defined ring isomorphism $K_0(\mathcal{O}_{q^{-1}})\simeq K_0(\mathcal{O}_q)$ (also written $\F_q$) that sends the subring $K_0(\mathcal{O}_{q^{-1}}^{-})$ of $K_0(\mathcal{O}_{q^{-1}})$ to $K_0(\mathcal{O}_q^{+})$. Recall also from Section \ref{sec:qcaract} the ring isomorphism $D_q: K_0(\mathcal{O}_q^+)\simeq K_0(\mathcal{O}_q^-)$ of \cite{hl2} which acts on equivalence classes of fundamental representations as 
$$D_q[V_i^q(a)] = [V_i^q(a^{-1})].$$ Consider finally the ring automorphism $\mathcal{H}_q=\F_q\circ D_{q^{-1}}\circ \F_{q^{-1}}\circ D_{q}$ of $K_0(\mathcal{O}_q^+)$.
\begin{lem}\label{lem:Hq} Let $i\in I$ and $a \in \C^{\times}$. Then $\mathcal{H}_q[V_i^q(a)] = [V_i^q(a\gamma_i^2)]$.
\end{lem}
\begin{proof} As $\F_{q^{-1}}$ is the inverse of $\F_q$, Theorem \ref{thm:FondF} implies 
$$ \textstyle \F_{q^{-1}}(V_i^q(a^{-1})) \simeq (\F_{q^{-1}}\circ \F_q)(V_i^{q^{-1}}((a\gamma_i)^{-1})) = V_i^{q^{-1}}((a\gamma_i)^{-1})$$ 
and it follows that
$ \mathcal{H}_q [V_i^q(a)] = (\F_q\circ D_{q^{-1}})[V_i^{q^{-1}}((a\gamma_i)^{-1})]=[V_i^q(a\gamma_i^2)]$ as wanted.
\end{proof}
Fix now $i\in I$. The next lemma follows essentially from \cite{fm}.
\begin{lem}\label{lemma:redTP} There exists $a\in\C^{\times}$ making the tensor product $V_{1}^q(1)\otimes V_i^q(a)$ reducible. 
\end{lem}
\begin{proof} By \cite[Theorem 4.1]{fm}, the lowest $\ell$-weight $\Psi'$ of $V_1^q(1)$ factorizes as $$\textstyle \Psi' = Y_{1,1}^{(q)} \prod_{j\in I, b\in \C^{\times}} A_{j,b}^{-u_{j,b}}$$ where $\{u_{j,b}\}_{j\in I,b\in\C^{\times}}\subseteq \mathbb{N}$ has finite support and where $A_{j,b}$ is the $\ell$-weight introduced in the discussion following Theorem \ref{thm:qcaract}. The comments made before Theorem \ref{thm:FondF} thus imply
$$ \textstyle Y_{1,1}^{(q)} Y_{1^*,q^{r^{\vee}h^{\vee}}}^{(q)} = \prod_{j\in I,b\in \C^{\times}}A_{j,b}^{u_{j,b}}. $$
We want to show that $u_{i,b}> 0$ for some $b \in \C^{\times}$. Toward this goal, let $u_j = \sum_{b\in \mathbb{C}^{\times}} u_{j,b}\in \mathbb{N}$ (for $j\in I$) and remark that the constant part of the above equality gives 
\begin{equation}\label{eq:CPartFM} 
\textstyle \overline{\omega_1+\omega_{1^*}} = \overline{\sum_{j\in I} u_j\alpha_j}.
\end{equation}
The inverse Cartan matrix also produces a decomposition
$$\textstyle \omega_1+\omega_{1^*} = \sum_{j\in I}((C^{-1})_{j,1}+(C^{-1})_{j,1^{*}})\alpha_j.$$
and the injectivity of the group morphism $\overline{\textcolor{white}{\alpha}}:P_{\mathbb{Q}}\rightarrow \ttt^{\times}$ with \eqref{eq:CPartFM} hence imply $$u_i = (C^{-1})_{i,1}+(C^{-1})_{i,1^*}$$
(as the simple roots $\{\alpha_j\}_{j\in I}$ are free in the root lattice of $\dot{\g}$). In particular, the strict positivity of the inverse Cartan matrices of simple Lie algebras of finite type (see e.g.~\cite{wz}) implies $u_i>0$ and it follows that $u_{i,b} > 0$ for some $b\in \C^{\times}$ as claimed. The lemma is then a direct consequence of \cite[Lemma 2.6 and Theorem 6.7]{fm}.
\end{proof}
\begin{theorem}\label{thm:SpecShiftEq} The spectral shifts $\gamma_1$ and $\gamma_i$ are equal. 
\end{theorem}
\begin{proof} By Lemma \ref{lemma:redTP}, there is a non-trivial decomposition in $K_0(\mathcal{O}_q^+)$ of the form
$$ \textstyle [V_1^q(1)\otimes V_i^q(a)] = [M_1]+[M_2] $$
where $M_1,M_2$ are non-zero finite-dimensional $\uqb$-modules. Lemma \ref{lem:Hq} then gives
$$ \mathcal{H}_q[V_1^q(1)\otimes V_i^q(a)] =[V_1^q(\gamma_1^2)\otimes V_i^q(a\gamma_i^2)] = \mathcal{H}_q[M_1]+\mathcal{H}_q[M_2]$$
so that, by induction, 
\begin{equation}\label{eq:nTDK} 
\mathcal{H}_q^m[V_1^q(1)\otimes V_i^q(a)] =[V_1^q(\gamma_1^{2m})\otimes V_i^q(a\gamma_i^{2m})] = \mathcal{H}_q^m[M_1]+\mathcal{H}_q^m[M_2]
\end{equation}
whenever $m\in \mathbb{N}$. Fix such an integer $m$ and note that $\mathcal{H}_q^m[M_1]$ and $\mathcal{H}_q^m[M_2]$ are equivalence classes of finite-dimensional $\uqb$-modules. (Indeed both $D_{q}$ and the map induced from $\F_{q}$ send classes of finite-dimensional modules to classes of finite-dimensional modules, see Section \ref{sec:qcaract}.) Equation \eqref{eq:nTDK} thus shows that (after a spectral parameter shift)
$$V_1^q(a_m)\otimes V_i^q(a) \simeq (V_1^q(\gamma_1^{2m})\otimes V_i^q(a\gamma_i^{2m}))((\gamma_i^{2m})^{-1})$$
is reducible for $a_m = (\frac{\gamma_1}{\gamma_i})^{2m}$ and it follows from \cite[Proposition 6.15]{fm} that $a_m = q^{r_m}$ for a $r_m\in \Z$. In particular, if $\gamma_i\neq \gamma_1$, the quantity $$\textstyle \frac{\gamma_1}{\gamma_i} = \sqrt{a_1} = q^{\frac{r_1}{2}}$$ is not a root of unity and the reducibility of the tensor products $V_1^q(a_m)\otimes V_i^q(a)$ for $m\in \mathbb{N}$ contradicts Theorem \ref{thm:genSim}. This implies $\gamma_i=\gamma_1$ and concludes the proof.
\end{proof}
The next corollary follows from Proposition \ref{prop:Fshift}, Lemma \ref{lemma:varpiPsiF}, Proposition \ref{prop:Fgroup}, Theorem \ref{thm:preFondF} and Theorem \ref{thm:SpecShiftEq} (see also Remark \ref{rem:Psishift} for the notation used).
\begin{cor}\label{cor:F2} There is $\gamma\in \C^{\times}$ such that $\Psi^{\F_q} = (\Psi^{-1})(\gamma)$ for all $\Psi \in \mathfrak{r}$.\hfill $\qed$
\end{cor}
\begin{rem} The relation $\Psi^{\F_q} = (\Psi^{-1})(\gamma)$ agrees with Theorem \ref{thm:FondF} as $(Y_{i,1}^{(q^{-1})})^{-1} = Y_{i,1}^{(q)}$.
\end{rem}
\begin{example}\label{ex:FqLp} Take $\g = \su$. Using the basis $\{z_j\}_{j\geq 0}$ given in Example \ref{ex:Lpmsl2} for the negative prefundamental representation $L_{1,a}^{-,q^{-1}}$ of $U_{q^{-1}}(\bor)$, we get
$$\pp_q(k_1)\star z_j = k_1^{-1}\star z_j = q^{-2j}z_j, \quad \pp_q(e_1)\star z_j = -k_1^{-1}e_1\star z_j= -q^{2(1-j)}z_{j-1} \text{ and}
$$
$$\pp_q(e_0)\star z_j = -k_1e_0\star z_j = q^{2(j+1)}a(q-q^{-1})^{-1}q^{j-2}[j+1]_qz_{j+1} = a(q-q^{-1})^{-1}q^{3j}[j+1]_qz_{j+1}.$$\\[-2.5mm]
Using instead the basis $\{w_j\}_{j\geq 0}$ defined by $w_j = (-1)^j q^{j(j-1)}z_j$, we obtain
$$ \pp_q(k_1)\star w_j = q^{-2j}w_j, \quad \pp_q(e_1)\star w_j = w_{j-1}, \quad \pp_q(e_0)\star w_j = a(q^{-1}-q)^{-1}q^{j}[j+1]_qw_{j+1}$$
and $\F_q(L_{1,a}^{-,q^{-1}})$ is therefore isomorphic to the positive prefundamental representation $L_{1,a q^{-2}}^{+,q}$ of $\uqb$ (see again Example \ref{ex:Lpmsl2}). Thus $\gamma = q^{-2}$ if $\g = \su$. 
\end{example}
\begin{rem} The shift $\gamma$ underlying Corollary \ref{cor:F2} is somewhat difficult to compute explicitly for Lie algebras other that $\g=\su$ and its exact value may depend on the correspondence chosen between the Drinfeld--Jimbo and the Drinfeld generating sets of $\uqg$.
\end{rem}
Consider now the functor $\mathcal{G}_q:\mathcal{O}_{q^{-1}}\arr \mathcal{O}_q$ given by $\mathcal{G}_q = \tau_{\gamma^{-1},q}^*\circ \F_q$
with $\tau_{\gamma^{-1},q}^*$ the pullback by the automorphism $\tau_{\gamma^{-1},q}$ of $\uqg$. This functor is again exact and reverses tensor products. Its action on the category $\mathcal{O}_{q^{-1}}$ is also totally specified by the following corollary.
\begin{cor}\label{cor:F} Let $\Psi \in \mathfrak{r}$. Then $\mathcal{G}_q(L_{q^{-1}}(\Psi))\simeq L_q(\Psi^{-1})$.
\end{cor}
\begin{proof}
The corollary follows by applying the pullback $\tau_{\gamma^{-1},q}^*$ to Corollary \ref{cor:F2}.
\end{proof}
\subsection{Induced relations in the Grothendieck ring}\label{sec:FGroth}
The functor $\mathcal{G}_q$ defined above can be used to deduce new relations for the Grothendieck ring $K_0(\mathcal{O}_q)$. We illustrate this procedure on the celebrated $Q\widetilde{Q}$-system of \cite{fh2} which is recalled below.\newpage
\begin{theorem}[{\cite[Section 3.1]{fh2}}]\label{thm:QQtilde} For $i\in I$ and $a\in \C^{\times}$, define
\begin{equation}\label{eq:Psitilde} \tilde{\Psi}_{i,a}^ {(q)} =\Psi_{i,a}^{-1}\prod_{\{j\in I|C_{i,j}=-1\}}\Psi_{j,aq_i}\prod_{\{j\in I|C_{i,j}=-2\}}\Psi_{j,a}\Psi_{j,aq_i^2}\prod_{\{j\in I|C_{i,j}=-3\}} \Psi_{j,aq_i^{-1}}\Psi_{j,aq_i}\Psi_{j,aq_i^3}
\end{equation}
and let $X_{i,a}^{(q)} = L_q(\tilde{\Psi}_{i,a}^{(q)})$ be the associated simple object in $\mathcal{O}_q$.\par Denote $\chi_i^{(q)}=\chi(L_{i,1}^{+,q})$ the character of the prefundamental represention $L_{i,1}^{+,q}$ of $\uqb$.  We view this character and the one of $X_{i,a}^{(q)}$ as elements of $K_0(\mathcal{O}_q)$ by identifying the map $[\mu]\in \mathcal{E}$ of Section \ref{sec:qcaract} with the class of the invertible representation $[\mu]$ of $\uqb$. \par Define finally 
$$ \check{\chi}_{i,a}^{(q)}=\chi(X_{i,a}^{(q)})\left(\textstyle \left[\overline{\frac{\alpha_i}{2}}\right]-\left[\overline{\frac{\alpha_i}{2}}\right]^{-1}\right) $$
and consider $$Q_{i,a}^{(q)} = (\chi_i^{(q)})^{-1}[L_{i,a}^{+,q}]\quad\text{ and }\quad\tilde{Q}_{i,a}^{(q)} = (\check{\chi}_{i,a}^{(q)})^{-1}[X_{i,aq_i^{-2}}^{(q)}].$$  
Then the following relation holds 
\begin{align}\label{eq:QQtilde}
{\textstyle \left[\overline{\frac{\alpha_i}{2}}\right]}Q_{i,aq_i^{-1}}^{(q)}&\tilde{Q}_{i,aq_i}^{(q)}-{\textstyle\left[\overline{\frac{\alpha_i}{2}}\right]^{-1}}Q_{i,aq_i}^{(q)}\tilde{Q}_{i,aq_i^{-1}}^{(q)} =\\ &\prod_{\{j\in I|	C_{i,j}=-1\}}Q_{j,a}^{(q)}\prod_{\{j\in I|	C_{i,j}=-2\}}Q_{j,aq^{-1}}^{(q)}Q_{j,aq}^{(q)}\prod_{\{j\in I|C_{i,j}=-3\}}Q_{j,aq^{-2}}^{(q)}Q_{j,a}^{(q)}Q_{j,aq^2}^{(q)}.\nonumber
\end{align}
\end{theorem}
\begin{rem} As stated in \cite[Section 5.7]{fhr}, the proof given in \cite{fh2} for the $Q\widetilde{Q}$-system is incomplete and the variables $Q_{i,a}^{(q)}$ and $\tilde{Q}_{i,a}^{(q)}$ used there must be renormalized. The adequate renormalization has been done in Theorem \ref{thm:QQtilde}.
\end{rem}
\begin{rem} Equation \eqref{eq:QQtilde} originated in \cite{mrv1,mrv2} (see also \cite{mv1,mv2} for the yangian version of this equation) from the study of affine ${}^L\g$-opers (with ${}^L\g$ the Langlands dual of $\g$) and is deeply linked to the Bethe Ansatz equations associated to $\g$.
\end{rem} For a simply-laced Kac-Moody algebra $\g$, this equation reduces to
$$ {\textstyle \left[\overline{\frac{\alpha_i}{2}}\right]}Q_{i,aq_i^{-1}}^{(q)}\tilde{Q}_{i,aq_i}^{(q)}-{\textstyle \left[\overline{\frac{\alpha_i}{2}}\right]^{-1}}Q_{i,aq_i}^{(q)}\tilde{Q}_{i,aq_i^{-1}}^{(q)} = \prod_{j\sim i \text{ in } I} Q_{j,a}^{(q)}.$$
In particular, for $\g = \su$, we have $X_{i,a}^{(q)} = L_{i,a}^{-,q}$ and hence $\chi(X_{i,a}^{(q)}) = \chi_i^{(q)} = (1-[\overline{\alpha_1}]^{-1})^{-1}$ by Example \ref{ex:Lpmsl2} and Theorem \ref{prop:qcharprefond}. The equation above is thus in this case equivalent to the \textit{quantum Wronskian relation} of \cite{blz}, that is $$ [L_{1,aq^{-1}}^{+,q}][L_{1,aq^{-1}}^{-,q}]-[\overline{\alpha_1}]^{-1}[L_{1,aq}^{+,q}][L_{1,aq^{-3}}^{-,q}] = \chi(L_{1,1}^{+,q}).$$
Consider the ring isomorphism $K_0(\mathcal{O}_{q^{-1}}) \arr K_0(\mathcal{O}_q)$ given by $[V]\mapsto [\mathcal{G}_q(V)]$. 
\begin{cor}\label{cor:QQtilde2} For $i\in I$ and $a\in \C^{\times}$, let $\mathcal{X}_{i,a}^{(q)}$ be the simple object $\mathcal{G}_q (X_{i,a}^{(q^{-1})})$ of $\mathcal{O}_q$. Define
$$ \mathcal{Q}_{i,a}^{(q)} = (\chi_i^{(q)})^{-1}[L_{i,a}^{-,q}] \quad\text{ and }\quad \tilde{\mathcal{Q}}_{i,a}^{(q)} = (\widetilde{\chi}_{i,a}^{(q)})^{-1}[\mathcal{X}_{i,aq_i^2}^{(q)}]$$
with 
$\chi_i^{(q)}=\chi(L_{i,1}^{+,q})$ and $\widetilde{\chi}_{i,a}^{(q)}=\chi(\mathcal{X}_{i,a}^{(q)})\left(\textstyle \left[\overline{\frac{\alpha_i}{2}}\right]-\left[\overline{\frac{\alpha_i}{2}}\right]^{-1}\right)$ again seen as elements of $K_0(\mathcal{O}_q)$.\par\noindent Then the following relation holds
\begin{align}\label{eq:QQtilde2}
{\textstyle \left[\overline{\frac{\alpha_i}{2}}\right]}\mathcal{Q}_{i,aq_i}^{(q)}&\tilde{\mathcal{Q}}_{i,aq_i^{-1}}^{(q)}-{\textstyle\left[\overline{\frac{\alpha_i}{2}}\right]}^{-1}\mathcal{Q}_{i,aq_i^{-1}}^{(q)}\tilde{\mathcal{Q}}_{i,aq_i}^{(q)} \\&= \prod_{\{j\in I|C_{i,j}=-1\}}\mathcal{Q}_{j,a}^{(q)}\prod_{\{j\in I|C_{i,j}=-2\}}\mathcal{Q}_{j,aq^{-1}}^{(q)}\mathcal{Q}_{j,aq}^{(q)}\prod_{\{j\in I|C_{i,j}=-3\}}\mathcal{Q}_{j,aq^{-2}}^{(q)}\mathcal{Q}_{j,a}^{(q)}\mathcal{Q}_{j,aq^2}^{(q)}.\nonumber
\end{align}
\end{cor}
\begin{proof} Let $V$ be in $\mathcal{O}_{q^{-1}}$. Then the proof of Proposition \ref{prop:FO} and the definition of the functor $\mathcal{G}_q$ give $\dim V_{\mu}=\dim \F_q(V)_{\mu^{-1}}=\dim\mathcal{G}_q(V)_{\mu^{-1}}$. It therefore follows from Corollary \ref{cor:F} that the map $[V]\mapsto[\mathcal{G}_q(V)]$ sends $\chi(V)\in K_0(\mathcal{O}_{q^{-1}})$ to 
$$\textstyle  \sum_{\mu\in \ttt^{\times}}\dim (V_{\mu})[\mu^{-1}] = \sum_{\mu\in \ttt^{\times}}\dim (\mathcal{G}_q(V)_{\mu^{-1}})[\mu^{-1}] = \chi(\mathcal{G}_q(V))\in K_0(\mathcal{O}_q).$$ 
Hence the elements $\chi_i^{(q^{-1})}$ and $\check{\chi}_{i,a}^{(q^{-1})}$ defined for the parameter $q^{-1}$ using Theorem \ref{thm:QQtilde} are respectively sent to (cf.~Theorem \ref{prop:qcharprefond}, Corollary \ref{cor:F} and the beginning of Section \ref{sec:Fq})
$$\chi(L_{i,1}^{-,q})=\chi_i^{(q)} \quad \text{ and }\quad \chi(\mathcal{X}_{i,a}^{(q)})\left({\textstyle \left[\overline{\overline{\frac{\alpha_i}{2}}}\right]^{-1}-\left[\overline{\overline{\frac{\alpha_i}{2}}}\right]}\right)=\widetilde{\chi}_{i,a}^{(q)}.$$
Relation \eqref{eq:QQtilde} for $K_0(\mathcal{O}_{q^{-1}})$ thus maps to \eqref{eq:QQtilde2} since the equivalence classes $Q_{i,a}^{(q^{-1})}$ and $\tilde{Q}_{i,a}^{(q^{-1})}$ of Theorem \ref{thm:QQtilde} respectively map to $\mathcal{Q}_{i,a}^{(q)}$ and $\tilde{\mathcal{Q}}_{i,a}^{(q)}$ (again by Corollary \ref{cor:F}).
\end{proof}
For $\g = \su$, one can show that \eqref{eq:QQtilde2} amounts to 
$$ [L_{1,aq}^{-,q}][L_{1,aq}^{+,q}]-[\overline{\alpha_1}]^{-1}[L_{1,aq^{-1}}^{-,q}][L_{1,aq^3}^{+,q}] = \chi(L_{1,1}^{+,q}) $$
which is the quantum Wronskian relation given above up to a shift of the parameter $a$. This coincidence is however fortuitous and the relations given in Theorem \ref{thm:QQtilde} and Corollary \ref{cor:QQtilde2} are typically non-equivalent when $\g \neq \su$. For example, if $\g = \widehat{\mathfrak{sl}}_3$, Corollary \ref{cor:F} implies
$$\mathcal{X}_{1,a}^{(q)} = \mathcal{G}_q(X_{1,a}^{(q^{-1})}) \simeq L(\Psi_{1,a}\Psi_{2,aq^{-1}}^{-1}) = X_{2,aq^{-1}}^{(q)}$$
and \eqref{eq:QQtilde2} gives the relation
$$ (\chi(X_{2,aq^{-2}}^{(q)}))^{-1}[L_{1,aq}^{-,q}][X_{2,a}^{(q)}]-[\overline{\alpha_1}]^{-1}(\chi(X_{2,a}^{(q)}))^{-1}[L_{1,aq^{-1}}^{-,q}][X_{2,aq^2}^{(q)}] = \textstyle\frac{\chi_1^{(q)}}{\chi_2^{(q)}}(1-[\overline{\alpha_1}]^{-1})[L_{2,a}^{-,q}] $$
but the relations obtained from Theorem \ref{thm:QQtilde} have the form
$$ (\chi(X_{1,aq}^{(q)}))^{-1}[L_{1,aq^{-1}}^{+,q}][X_{1,aq^{-1}}^{(q)}] - [\overline{\alpha_1}]^{-1}(\chi(X_{1,aq^{-1}}^{(q)}))^{-1}[L_{1,aq}^{+,q}][X_{1,aq^{-3}}^{(q)}]=\textstyle\frac{\chi_1^{(q)}}{\chi_2^{(q)}}(1-[\overline{\alpha_1}]^{-1})[L_{2,a}^{+,q}]$$ 
for $i=1$ and 
$$ (\chi(X_{2,aq}^{(q)}))^{-1}[L_{2,aq^{-1}}^{+,q}][X_{2,aq^{-1}}^{(q)}] - [\overline{\alpha_2}]^{-1}(\chi(X_{2,aq^{-1}}^{(q)}))^{-1}[L_{2,aq}^{+,q}][X_{2,aq^{-3}}^{(q)}]=\textstyle\frac{\chi_2^{(q)}}{\chi_1^{(q)}}(1-[\overline{\alpha_2}]^{-1})[L_{1,a}^{+,q}]$$ 
for $i=2$. The last two relations clearly differ from the first one.\par
We now wish to apply $\mathcal{G}_q$ to the $QQ^*$-system of \cite[Section 6]{hl2}. This system also depends on parameters $i\in I$ with $a\in \C^{\times}$ and naturally describes the mutation of a remarkable 
cluster algebra associated to $K_0(\mathcal{O}_q^+)$ (see the upcoming Remark \ref{rem:MonCatClust}). It can be written as
\begin{equation}\label{eq:QQ*+}
[L_{i,a}^{\circ,q}][L_{i,a}^{+,q}] = [\overline{\omega_i}]\prod_{\{j\in I|C_{j,i}\neq 0\}}[L_{j,aq_j^{-C_{j,i}}}^{+,q}]+[\,\overline{\omega_i-\alpha_i}\,]\prod_{\{j\in I|C_{j,i}\neq 0\}}[L^{+,q}_{j,aq_j^{C_{j,i}}}]
\end{equation}
with $L_{i,a}^{\circ,q}$ the simple object of $\mathcal{O}_q^+$ with highest $\ell$-weight $\Psi_1 =Y_{i,aq_i^{-1}}^{(q)} \prod_{\{j\in I|C_{j,i}<0\}} \Psi_{j,aq_j^{-C_{j,i}}}.$ \par Similarly (see \cite[Section 7]{hl2}),
\begin{equation}\label{eq:QQ*-}
[L_{i,a}^{*,q}][L_{i,a}^{-,q}] = [\overline{\omega_i}]\prod_{\{j\in I|C_{j,i}\neq 0\}}[L_{j,aq_j^{C_{j,i}}}^{-,q}]+[\,\overline{\omega_i-\alpha_i}\,]\prod_{\{j\in I|C_{j,i}\neq 0\}}[L^{-,q}_{j,aq_j^{-C_{j,i}}}]
\end{equation}
with $L_{i,a}^{*,q}$ the simple object of $\mathcal{O}_q^-$ with highest $\ell$-weight $\Psi_2 =Y_{i,aq_i}^{(q)}\prod_{\{j\in I|C_{j,i}<0\}} \Psi_{j,aq_j^{C_{j,i}}}^{-1}.$
\begin{rem} For $\g = \su$, \eqref{eq:QQ*+} reduces to Baxter's QT-relation (cf.~\cite{fh1})
$$[V_1^q(aq^{-1})][L_{1,a}^{+,q}]=[\overline{\omega_1}][L_{1,aq^{-2}}^{+,q}]+[\overline{\omega_1}]^{-1}[L_{1,aq^2}^{+,q}].$$
\end{rem}
By Corollary \ref{cor:F}, the relation of $K_0(\mathcal{O}_{q^{-1}}^-)$ obtained from \eqref{eq:QQ*-} by inverting the parameter $q$ is sent to \eqref{eq:QQ*+} by the ring isomorphism $[V]\mapsto[\mathcal{G}_q(V)]$. We hence do not obtain new relations for $K_0(\mathcal{O}_q)$ in this case. However, as shown in \cite{her}, \eqref{eq:QQ*-} may be deduced from a short exact sequence of $\mathcal{O}_q^-$ describing the decomposition of the module $L_{i,a}^{*,q} \otimes L_{i,a}^{-,q}$ in simple factors. We state this more precisely in the next theorem.
\begin{theorem}[{\cite[Theorem 5.16]{her}}]\label{thm:QQ*} Let $i\in I$ and $a\in \C^{\times}$. Then there is a non-split short exact sequence in $\mathcal{O}_q^-$
$$ 0\arr \Big( [\,\overline{\omega_i-\alpha_i}\,]\otimes \bigotimes_{\{j\in I| C_{j,i}\neq 0\}} L_{j,aq_j^{-C_{j,i}}}^{-,q}\Big)\arr L_{i,a}^{*,q}\otimes L_{i,a}^{-,q} \arr \Big([\overline{\omega_i}]\otimes \bigotimes_{\{j\in I| C_{j,i}\neq 0\}} L_{j,aq_j^{C_{j,i}}}^{-,q}\Big)\arr 0 $$
where the extremal (non-zero) factors are simple modules.
\end{theorem}
Using the functor $\mathcal{G}_q$ on the exact sequence of $\mathcal{O}_{q^{-1}}$ given by Theorem \ref{thm:QQ*} (with quantum parameter $q^{-1}$) produces the following ``categorified version'' of the $QQ^*$-system \eqref{eq:QQ*+}. (Recall that a tensor product of prefundamental representations of the same sign is necessarily simple by Theorem \ref{thm:ProdPrefond} and can therefore be written in any order up to isomorphism. This is used implicitly below in order to deduce Corollary \ref{cor:QQ*} from Theorem \ref{thm:QQ*}.)
\begin{cor}\label{cor:QQ*} Let $i\in I$ and $a\in \C^{\times}$. There is a non-split short exact sequence in $\mathcal{O}_q^+$
$$ 0\arr \Big( [\,\overline{\omega_i-\alpha_i}\,]\otimes \bigotimes_{\{j\in I| C_{j,i}\neq 0\}} L_{j,aq_j^{C_{j,i}}}^{+,q}\Big)\arr L_{i,a}^{\circ,q}\otimes L_{i,a}^{+,q} \arr \Big([\overline{\omega_i}]\otimes \bigotimes_{\{j\in I| C_{j,i}\neq 0\}} L_{j,aq_j^{-C_{j,i}}}^{+,q}\Big)\arr 0 $$
where the extremal (non-zero) factors are simple modules. \hfill $\qed$
\end{cor}
\begin{rem} The proof given in \cite{her} for Theorem \ref{thm:QQ*} relies heavily on the analysis of the specialization at $u=1$ of the $R$-matrix (see Section \ref{sec:RMat}) $$R_{L_{i,a}^{*,q},L_{i,a}^{-,q}}(u) : L_{i,a}^{*,q}(u)\otimes L_{i,a}^{-,q} \arr L_{i,a}^{-,q}\otimes L_{i,a}^{*,q}(u)$$
and this proof thus cannot be adapted to Corollary \ref{cor:QQ*} before obtaining a general construction of such $R$-matrices for $\mathcal{O}_q^+$.
\end{rem}
\subsection{Functorial interpretation of Hernandez--Leclerc's duality}\label{sec:FuncD}
In this subsection, we understand the quantum parameter $q$ as being a formal variable and view $\uqg$ as an algebra over the algebraic closure $\mathbbm{k}_q$ of $\C(q)$ in $\bigcup_{m>0} \C((q^{1/m}))$. For another formal parameter $q'$, we have a natural isomorphism of $\C$-algebras $\mathbbm{k}_q\simeq \mathbbm{k}_{q'}$ that sends $q$ to $q'$ (this map would not be well-defined without the assumption that $q$ and $q'$ are formal variables). This isomorphism induces a second isomorphism of $\C$-algebras $$\varsigma_{q',q}:\uqg\arr U_{q'}(\g)$$ which sends elements of the Drinfeld generating set of $\uqg$ to the associated elements in the Drinfeld generating set of $U_{q'}(\g)$. 
\begin{rem}\label{rem:DJB} By adequately choosing the correspondences between the Drinfeld and Drinfeld--Jimbo generating sets of $\uqg$ and $U_{q'}(\g)$, one can ensure that the isomorphism $\varsigma_{q',q}$ also maps elements of the Drinfeld--Jimbo generating set of $\uqg$ to the corresponding elements in the Drinfeld--Jimbo generating set of $U_{q'}(\g)$. Such a choice is always possible (see e.g.~\cite[Section 2.2]{hj} for details) and will be assumed throughout this subsection (and in this subsection only).  
\end{rem} 
Define $\mathcal{B}_{q',q}$ as the covariant functor $U_{q'}(\bor)\otimes_{\uqb} - $ from the category of all $\uqb$-modules to the category of all $U_{q'}(\bor)$-modules where the (right) $\uqb$-action on $U_{q'}(\bor)$ is through the isomorphism $\varsigma_{q',q}$. This is an invertible functor (and is thus exact) since 
$$ \mathcal{B}_{q',q} \circ \mathcal{B}_{q,q'} = U_{q'}(\bor)\otimes_{\uqb} (\uqb \otimes_{U_{q'}(\bor)} -) \simeq U_{q'}(\bor) \otimes_{U_{q'}(\bor)} -$$
is canonically isomorphic to the identity functor of the category of all $\uqb$-modules. Moreover, $\varsigma_{q',q}$ induces a $\mathbbm{k}_{q'}$-algebra isomorphism $\mathbbm{k}_{q'}\otimes_{\mathbbm{k}_q} \uqb\simeq U_{q'}(\bor)$ and we can hence naturally identify the image $\mathcal{B}_{q',q}(V)$ of a $\uqb$-module $V$ with the $U_{q'}(\bor)$-module defined on $\mathbbm{k}_{q'}\otimes_{\mathbbm{k}_q} V$ via the action $x(1\otimes_{\mathbbm{k}_q}v) = 1\otimes_{\mathbbm{k}_q} \varsigma_{q',q}^{-1}(x)v$ (for any $v\in V$ and $x\in U_{q'}(\bor)$). Thus, $\mathcal{B}_{q',q}$ sends $\uqb$-modules of dimension $n$ over $\mathbbm{k}_q$ to $U_{q'}(\bor)$-modules of the same dimension over $\mathbbm{k}_{q'}$ (for $n\in \mathbb{N}\cup\{\infty\}$). It also preserves irreducibility as does any exact functor with exact inverse.
\begin{rem}\label{rem:BVact} Consider a $\uqb$-module $V$ with a fixed $\mathbbm{k}_q$-basis $\{v_j\}_{j\in J}\subseteq V$ and coefficients $\{e_{i,j,s}(q),k_{i,j,s}(q)\}_{i\in I, j,s \in J} \subseteq \mathbbm{k}_q$ such that 
$\textstyle e_i v_j = \sum_{s\in J} e_{i,j,s}(q) v_s \text{ and } k_i v_j = \sum_{s\in J} k_{i,j,s}(q) v_s.$
Then the $U_{q'}(\bor)$-action on the $\mathbbm{k}_{q'}$-basis $\{1\otimes_{\mathbbm{k}_q}v_j\}_{j\in J}$ of $\mathcal{B}_{q',q}(V)$ is (see Remark \ref{rem:DJB})
$$\textstyle e_i (1\otimes_{\mathbbm{k}_q} v_j) = \sum_{s\in J} e_{i,j,s}(q') (1\otimes_{\mathbbm{k}_q} v_s) \text{ and } k_i v_j = \sum_{s\in J} k_{i,j,s}(q') (1\otimes_{\mathbbm{k}_q} v_s).$$
\end{rem}
\begin{prop} Let $V$ be an object of $\mathcal{O}_q$. Then $\mathcal{B}_{q',q}(V)$ is in $\mathcal{O}_{q'}$.
\end{prop}
\begin{proof}
Fix $\mu(q) = (\mu_i(q))_{i\in I}\in (\mathbbm{k}_q^{\times})^{I}$. Then $\mathbbm{k}_{q'}\otimes_{\mathbbm{k}_q}V_{\mu(q)} =(\mathcal{B}_{q',q}(V))_{\mu(q')}$ (see Remark \ref{rem:BVact}) and it follows that $\mathcal{B}_{q',q}(V)$ is isomorphic to the direct sum of its weight spaces which are all finite-dimensional over $\mathbbm{k}_{q'}$. Consider now $\lambda_1(q),...,\lambda_s(q) \in (\mathbbm{k}_q^{\times})^{I}$ with $P(V) \subseteq \bigcup_{j=1}^s D(\lambda_j(q))$ and fix a weight $\omega(q') \in P(\mathcal{B}_{q',q}(V))\subseteq (\mathbbm{k}_{q'}^{\times})^{I}$. Then $\omega(q) \in P(V)$ by the precedent argument and there is hence $1\leq r\leq s$ such that $\omega(q)\leq_q \lambda_r(q)$ with $\leq_q$ the partial order on $\ttt^{\times} = (\mathbbm{k}_q^{\times})^I$ defined in Section \ref{sec:O}. Let us denote by $\leq_{q'}$ the partial order on $(\mathbbm{k}_{q'}^{\times})^{I}$ obtained from $\leq_q$ by changing the quantum parameter to $q'$. We want to show that $\omega(q')\leq_{q'} \lambda_r(q')$. For this goal, write $\lambda_r(q)(\omega(q))^{-1} =\mu(q) = (\mu_i(q))_{i\in I}$ so that $\omega(q)\leq_q \lambda_r(q)$ implies $$ \textstyle \mu_i(q) = q^{d_i\sum_{j\in I} m_j C_{ij}} = \overline{\sum_{j\in I}m_j\alpha_j}$$
for some $\{m_j\}_{j\in I}\subseteq \mathbb{N}$. Therefore $ \mu_i(q') = (q')^{d_i\sum_{j\in I} m_j C_{ij}}$ and $\mu(q')=\lambda_r(q')(\omega(q'))^{-1}$ is the image of $\sum_{j\in I}m_j\alpha_j \in Q_+$ by the map obtained from $\overline{\textcolor{white}{\alpha}} : P_{\mathbb{Q}} \rightarrow \ttt^{\times}$ by changing the quantum parameter $q$ by $q'$. This ends the proof.
\end{proof} 
The functor $\mathcal{B}_{q',q}$ thus gives rise to a functor from $\mathcal{O}_q$ to $\mathcal{O}_{q'}$ by restriction. Let us denote by $\mathfrak{r}_q$ the set of sequences $\Psi(q) = (\Psi_{i,r}(q))_{i\in I,r\geq 0}\subseteq\mathbbm{k}_q$ with $\varpi(\Psi(q))=(\Psi_{i,0}(q))_{i\in I} \in \ttt^{\times} = (\mathbbm{k}_q^{\times})^I$ and such that $\Psi_{i}(z,q) = \sum_{r\geq 0}\Psi_{i,r}(q)z^r$ is a rational function in $z$ for any $i\in I$.
Let us also write $L_q(\Psi(q))$ for the simple object of $\mathcal{O}_q$ of highest $\ell$-weight $\Psi(q)\in \mathfrak{r}_q$.
\begin{prop}\label{prop:B} Fix $\Psi(q)=(\Psi_{i,r}(q))_{i\in I,r\geq 0}\in\mathfrak{r}_q$. Then $\mathcal{B}_{q',q}(L_q(\Psi(q))) \simeq L_{q'}(\Psi(q'))$.
\end{prop}
\begin{proof} Fix a highest $\ell$-weight vector $v$ in $L_q(\Psi(q))$ . Then $1\otimes_{\mathbbm{k}_q}v$ is a highest $\ell$-weight vector with highest $\ell$-weight $\Psi(q')$ for the simple $U_{q'}(\bor)$-module $\mathcal{B}_{q',q}(L_q(\Psi(q)))$ as 
$$\phi_{i,r}^+(1\otimes_{\mathbbm{k}_q}v) = 1\otimes_{\mathbbm{k}_q}\phi_{i,r}^+v = 1\otimes_{\mathbbm{k}_q}\Psi_{i,r}(q)v = \Psi_{i,r}(q')(1\otimes_{\mathbbm{k}_q}v)$$ 
and $e_i(1\otimes_{\mathbbm{k}_q}v)=1\otimes_{\mathbbm{k}_q}e_iv=0$. This finishes the proof.
\end{proof}
\begin{lem} Let $V,W$ be $\uqb$-modules. Then $\mathcal{B}_{q',q}(V\otimes_{\mathbbm{k}_{q}} W) \simeq \mathcal{B}_{q',q}(V)\otimes_{\mathbbm{k}_{q'}} \mathcal{B}_{q',q}(W)$.
\end{lem}
\begin{proof} Use the obvious isomorphism $\mathbbm{k}_{q'}\otimes_{\mathbbm{k}_{q}} (V\otimes_{\mathbbm{k}_{q}} W) \simeq (\mathbbm{k}_{q'}\otimes_{\mathbbm{k}_{q}} V)\otimes_{\mathbbm{k}_{q'}}(\mathbbm{k}_{q'}\otimes_{\mathbbm{k}_{q}} W)$ (which can be easily seen to be $U_{q'}(\g)$-linear for the underlying actions).
\end{proof}
Consider the covariant autofunctor $\mathcal{D}_{q} = \mathcal{G}_q \circ \mathcal{B}_{q^{-1},q}$ of $\mathcal{O}_q$. By the above result, this is an invertible exact functor which reverses tensor products of modules and preserves irreducibility as well as dimensions (over $\mathbbm{k}_q$). It thus induces a ring automorphism of $K_0(\mathcal{O}_q)$ which sends equivalence classes of simple modules to equivalence classes of simple modules. Moreover, we have $\mathcal{D}_{q}(L_q(\Psi(q))) \simeq L_q((\Psi(q^{-1}))^{-1})$ by Corollary \ref{cor:F} and Proposition \ref{prop:B} so that 
$$ \mathcal{D}_q([\mu(q)]) \simeq [(\mu(q^{-1}))^{-1}],\quad \mathcal{D}_{q}(L_{i,a(q)}^{\pm,q}) \simeq L_{i,a(q^{-1})}^{\mp,q} \ \ \text{and}\ \ \mathcal{D}_{q}(V_i^{q}(a(q))) \simeq V_i^{q}(a(q^{-1}))$$
for $\mu(q)\in \ttt^{\times} = (\mathbbm{k}_q^{\times})^I$, $i\in I$ and $a(q)\in \mathbbm{k}_q$. In particular, 
$$ \mathcal{D}_q([\overline{\omega}]) \simeq [\overline{\omega}],\quad \mathcal{D}_{q}(L_{i,q^r}^{\pm,q}) \simeq L_{i,q^{-r}}^{\mp,q} \ \ \text{and}\ \ \mathcal{D}_{q}(V_i^{q}(q^r)) \simeq V_i^{q}(q^{-r})$$
for $\omega \in P_{\mathbb{Q}}$, $i\in I$ and $r\in \Z$. This observation leads to the following two definitions.
\begin{defn}\label{def:OZ} A $\ell$-weight $\Psi(q)=(\Psi_{i,r}(q))_{i\in I,r\geq 0}\in \mathfrak{r}_q$ is said to be \textit{$\ell$-integral} if 
\begin{itemize}
\item[(i)] Its constant part $\varpi(\Psi(q))$ belongs to the image of the morphism $\overline{\textcolor{white}{\alpha}} : P_{\mathbb{Q}} \rightarrow \ttt^{\times}$ and
\item[(ii)] The roots and poles of the rational function $\Psi_i(z,q)$ are contained in $q^{\mathbb{Z}}$ for any $i\in I$.
\end{itemize}
\end{defn}
\begin{defn}[inspired by {\cite{hl1}} and {\cite[Definition 4.1]{hl2}}]\label{def:OZ2} The category $\mathcal{O}_{\Z,q}$ is the full subcategory of $\mathcal{O}_q$ whose objects have simple constituents with $\ell$-integral highest $\ell$-weights.
\end{defn}
This is a monoidal subcategory as $\chi_q(V\otimes W) = \chi_q(V)\chi_q(W)$. Also, Definition \ref{def:OZ} and the above results show that $\mathcal{D}_q(L_q(\Psi(q))) \simeq L_q((\Psi(q^{-1}))^{-1})$ is in $\mathcal{O}_{\Z,q}$ for all $\ell$-integral $\ell$-weights $\Psi(q) \in \mathfrak{r}_q$. The functor $\mathcal{D}_q$ thus restricts to an autofunctor of the category $\mathcal{O}_{\Z,q}$. \par
Write now $\mathcal{O}_{\mathbb{Z},q}^{\pm} = \mathcal{O}_{\Z,q}\cap \mathcal{O}^{\pm}_q$ and recall Hernandez--Leclerc's duality $D_q:K_0(\mathcal{O}^+_q)\simeq K_0(\mathcal{O}_q^-)$ from Section \ref{sec:qcaract}. Then \cite[Theorem 7.9]{hl2} give
$$D_q^{\pm 1}[L_q(\Psi(q))]=[L_q((\Psi(q^{-1}))^{-1})]= [\mathcal{D}_q(L_q(\Psi(q)))]$$ for any $L_q(\Psi(q))$ in $\mathcal{O}_{\Z,q}^{\pm}$. In particular, the fact that both $D_q^{\pm 1}$ and $\mathcal{D}_q$ are trivially compatible with countable sums of equivalence classes of simple modules shows that $D_q$ restricts to a ring isomorphism $D_q: K_0(\mathcal{O}_{\Z,q}^+)\simeq K_0(\mathcal{O}_{\Z,q}^-)$ satisfying $ D_q^{\pm 1}[V] = [\mathcal{D}_q(V)]$ for every $V$ in $\mathcal{O}_{\Z,q}^{\pm}$. We write this result in a proper theorem. This is the main result of this subsection. 
\begin{theorem}\label{thm:DCat}
The functor $\mathcal{D}_q$ is involutive and restricts to a covariant autofunctor of $\mathcal{O}_{\Z,q}$ which is exact, reverses tensor products and preserves the irreducibility of modules as well as their dimension (over $\mathbbm{k}_q$). In addition, the ring automorphism of $K_0(\mathcal{O}_{\Z,q})$ induced from $\mathcal{D}_q$ extends the ring isomorphism $D_q:K_0(\mathcal{O}_q^+)\simeq K_0(\mathcal{O}_q^-)$ in the sense that 
$$ [\mathcal{D}_q(V)] = D_q^{\pm 1}[V] \text{ for any }V\text{ in }\mathcal{O}_{\mathbb{Z},q}^{\pm}. $$
We thus say that $\mathcal{D}_q$ categorifies Hernandez--Leclerc's duality $D_q$ on $\mathcal{O}_{\Z,q}^{\pm}$. \hfill $\qed$
\end{theorem}
\begin{rem}\label{rem:MonCatClust}
The categories $\mathcal{O}_{\Z,q}^{\pm}$ defined above are instances of monoidal categorifications of cluster algebras (see e.g.~\cite{kkop2,kkop3} and \cite[Section 8.4]{hshift} with the references therein). In this context, the map $K_0(\mathcal{O}_{\Z,q}^+)\simeq K_0(\mathcal{O}_{\Z,q}^-)$ induced by Hernandez--Leclerc's duality $D$ is a cluster algebra isomorphism and the above theorem states that this isomorphism is also the shadow of a functor $\mathcal{O}^+_{\Z,q} \arr \mathcal{O}^-_{\Z,q}$ between the categorifications. Furthermore, Theorem \ref{thm:DCat} realizes this peculiar functor as the restriction of an autofunctor $\mathcal{D}_q$ of the category $\mathcal{O}_{\Z,q}$ and implies that the map $K_0(\mathcal{O}_{\Z,q}^+)\simeq K_0(\mathcal{O}_{\Z,q}^-)$ can be extended to a ring involution of $K_0(\mathcal{O}_{\Z,q})$. This is the meaning of: ``$\mathcal{D}_q$ categorifies $D_q$ on $\mathcal{O}_{\Z,q}$''. 
\end{rem}
\begin{proof}
The only remaining thing to prove is the involutivity of $\mathcal{D}_q$. For this goal, remark that the functor $\mathcal{G}_q$ of Section \ref{sec:Fq} is canonically isomorphic to the functor $\uqb\otimes_{\uqbm}-$ where the (right) action of $\uqbm$ on $\uqb$ is defined as $x\cdot y = x(\tau_{\gamma_1,q}\circ \pp_{q^{-1}}(y))$ for $x\in \uqb$ and $y\in \uqbm$. It is then easy to construct a functorial isomorphism
$$ \mathcal{D}_q\circ\mathcal{D}_q(V) \simeq \uqb \otimes_{\uqbm} \uqbm \otimes_{\uqb} \uqb \otimes_{\uqbm} \uqbm \otimes_{\uqb} V \simeq V$$
for any $V$ in $\mathcal{O}_q$. This concludes the proof by the above comments.
\end{proof}
We end this section by stating the following result which is proved in Appendix \ref{app:K0D}. Observe that this result is true for $q$ a formal variable or $q\in \C^{\times}$ which is not a root of unity and that it answers a question of Hernandez--Leclerc about the possible extensions of their duality $D$.
\begin{theorem}\label{thm:DExt}
There exists an involutive ring isomorphism $\tilde{D}_q$ of $K_0(\mathcal{O}_q)$ with the property that $\tilde{D}_q([V]) = D_q^{\pm 1}([V])$ for any $V$ in $\mathcal{O}^{\pm}_q$.
\end{theorem}
\begin{rem} One may wonder if $\tilde{D}_q$ can be induced from some autofunctor of $\mathcal{O}_q$. It is also unclear whether or not it preserves classes of simple modules (see the end of Appendix \ref{app:K0D}). 
\end{rem}
\section{$R$-matrices in the category $\mathcal{O}$}\label{sec:RMat}
This section constructs $R$-matrices for the subcategory $\mathcal{O}^{+}$ of $\mathcal{O}$. We first recall the corresponding construction for finite-dimensional simple modules (through the universal $R$-matrix) with the (affine) $R$-matrices obtained for the subcategory $\mathcal{O}^-$ in \cite{her}. We then use the results of the latter paper and the functor $\F_q$ of Section \ref{sec:Fq} to obtain the $R$-matrices for $\mathcal{O}^+$ with a factorization of these braidings and a partial characterization of their singularities.
\par 
We will use the notation of Section \ref{sec:Qaff} and will mention explicitly when a notation should be understood inside the representation theory of $\uqgm$. (This will happen only when making use of the functor $\F_q$.) For example, the category $\mathcal{O}$ above is the category $\mathcal{O}_q$ of Section \ref{sec:Fq}.
\subsection{$R$-matrices for finite-dimensional representations}\label{sec:Rdimfin} 
The universal $R$-matrix $\mathcal{R}(u)$ is a remarkable element of the (slightly completed, see \cite{hbourb}) tensor product $(\uqg\hat{\otimes}\uqg)[[u]]$. It was introduced by Drinfeld in \cite{dr2} and admits a factorization in four components
\begin{equation}\label{eq:RmatfdFac}
\mathcal{R}(u) = \mathcal{R}^-(u)\mathcal{R}^{0}(u)\mathcal{R}^+(u)\mathcal{R}^{\infty}
\end{equation}
with $\mathcal{R}^{\pm}(u) \in (U_q^{\pm}(\bor)\hat{\otimes}U_q^{\mp}(\bor^-))[[u]]$. Here, $U_q(\bor^-)\subseteq \uqg$ is the subalgebra generated by the set $\{f_i,k_i^{\pm 1}\}_{i=0}^n$ and (as in Section \ref{sec:DefQaff})
$$U_q^{\pm}(\bor)= \uqb\cap U_q^{\pm}(\g)  \text{ with }  U_q^{\pm}(\bor^-) = U_q(\bor^-)\cap U_q^{\pm}(\g).$$
To describe the component $\mathcal{R}^{\infty}$ appearing in \eqref{eq:RmatfdFac}, let $(\cdot,\cdot)$ be the standard invariant symmetric bilinear form of $\dot{\g}$ (see e.g.~\cite[Chapter 2]{kac}) normalized such that $(\alpha_i,\alpha_j)= d_iC_{i,j}$ for $i,j\in I$. Recall also that $q = e^h$ for some $h\in \C^{\times}$ (see Section \ref{sec:DefQaff}). Then 
$\mathcal{R}^{\infty}=e^{-ht_{\infty}}$
where $t_{\infty}$ is the canonical element for the form $(\cdot,\cdot)$ in the tensor square of the Cartan subalgebra $\dot{\h}$ of $\dot{\g}$ (see \cite{da1}). In other terms, for $V,W$ objects in the category $\mathcal{O}$,
$$ \mathcal{R}^{\infty}(v\otimes w) = q^{-(\nu,\omega)}v\otimes w$$
if $v \in V$ and $w \in W$ are weight vectors of respective weights $(q^{(\nu,\alpha_i)})_{i\in I}$ and $(q^{(\omega,\alpha_i)})_{i\in I}$ with $\nu,\omega \in \C\otimes_{\Z} P$. (Note that the assumption just made on the weights of $v$ and $w$ is done without loss of generality as $q=e^h$. Indeed, fix $\mu=(\mu_i)_{i\in I} \in \ttt^{\times}$ and $\omega = \frac{1}{h}\sum_{i\in I} \frac{1}{d_i}\log(\mu_i)\omega_i \in \C\otimes_{\Z} P$ where we consider the principal branch of the logarithm. Then $\mu_i = e^{\log(\mu_i)}= q^{(\omega,\alpha_i)}$ for every $i\in I$ by definition of the standard form \cite[Chapter 2]{kac}.) \par An expression for the component $\mathcal{R}^0(u)$ of \eqref{eq:RmatfdFac} is known (see e.g.~\cite[Section 7]{fh1}) but will not be used in this paper. We will nevertheless need the fact that this component belongs to the subalgebra $(\uqb\hat{\otimes}U_q(\bor^-))[[u]]$. We will also consider the following explicit expressions which hold for $\g =\su$
\begin{center}
$\mathcal{R}^+(u) = \displaystyle \prod_{m\geq 0}^{\arr}\exp_q\big((q^{-1}-q)u^m x_{1,m}^+\otimes x_{1,-m}^-\big),$\\
$\mathcal{R}^-(u) = \displaystyle \prod_{m > 0}^{\leftarrow}\exp_{q^{-1}}\big((q^{-1}-q)u^mk_1^{-1}x_{1,m}^-\otimes x_{1,-m}^+k_1\big),$\\
$\mathcal{R}^0(u) = \displaystyle \exp\Big((q^{-1}-q)\sum_{m>0}\frac{mu^m}{[m]_q(q^m+q^{-m})}h_{1,m}\otimes h_{1,-m}\Big)$
 \end{center}
with $\exp_{q^{\pm 1}}(x) = \sum_{r\geq 0}q^{\pm \frac{r}{2}(1-r)}\frac{x^r}{[r]_q!}$ the $q$-exponential function.\par
The universal $R$-matrix is of fundamental importance in the theory of integrable systems as it gives a non-trivial (and somewhat universal) solution of the Yang--Baxter equation (cf.~\cite{hbourb} and the references therein). It can also act on any tensor product $V\otimes W$ of finite-dimensional irreducible $\uqg$-modules to produce a well-defined map $\Eval{\mathcal{R}(u)}{V\otimes W}{}:V(u)\otimes W\rightarrow V(u)\otimes W$ such that, for $\tau : V(u)\otimes W \rightarrow W\otimes V(u)$ the usual flip (as in Theorem \ref{thm:Comon}), 
\begin{equation}\label{eq:RM1} 
\tau\circ\Eval{\mathcal{R}(u)}{V\otimes W}{} : V(u)\otimes W \rightarrow W\otimes V(u)
\end{equation}
is a linear isomorphism with respect to the action of the algebra $U_{q,u}(\g) = \uqg \otimes \C(u)$. We call \textit{rational $R$-matrix}\footnote{The term ``rational'' is used here to distinguish the $R$-matrices described here from those of Section \ref{sec:ROm}.} the isomorphism $ R_{V,W}(u) $ obtained by renormalizing \eqref{eq:RM1} such that 
\begin{equation}\label{eq:NormR}
 R_{V,W}(u)(v\otimes w) = w\otimes v
 \end{equation}
for $v\in V$ and $w\in W$ two highest $\ell$-weight-vectors. It is well-known that there exists exactly one $U_{q,u}(\g)$-linear isomorphism $V(u)\otimes W\simeq W\otimes V(u)$ satisfying \eqref{eq:NormR} when $V,W$ are simple finite-dimensional $\uqg$-modules. The rational $R$-matrix $R_{V,W}(u)$ could thus also be defined as the unique such isomorphism. Moreover, as its name suggests, the map $R_{V,W}(u)$ is rational in the spectral parameter $u$ and has only a finite set of poles. This still holds for the inverse map since (by the unicity mentioned above)
\begin{equation}\label{eq:RInv} R_{V,W}^{-1}(u) = R_{W,V}(u^{-1})
\end{equation}
and the map $R_{V,W}(u)$ induces a $\uqg$-linear isomorphism $R_{V,W}(a) : V(a)\otimes W\arr W\otimes V(a) $ for any $a\in \C^{\times}$ which is neither a pole of $R_{V,W}(u)$ nor of $R_{V,W}^{-1}(u)$ (see \cite{her} and the references therein for more details about the many facts stated above).  
\begin{example}\label{ex:RmatKR} Take $\g = \su$ with $\{v_i\}_{i=0}^k$ the basis of $V_k = W_{k,a q^{1-2k}}^{(1)}$ given in Example \ref{ex:KRsl2}. Then the rational $R$-matrix $R_{k}(u):V_k(u)\otimes V_k \rightarrow V_k\otimes V_k(u)$ may be expressed as
\begin{align*}
R_{k}(u)(v_i\otimes v_j) = \sum_{\mu = \max(-i,j-k)}^{\min(j,k-i)}\left(\sum_{\lambda = \max(0,\mu)}^{\min(j,k-i)} \upsilon_{i,j,k}^{\mu,\lambda}(u)\gamma_{i,j,k}^{\mu,\lambda}(u) \right)v_{j-\mu}\otimes v_{i+\mu}
\end{align*}
with $\upsilon_{i,j,k}^{\mu,\lambda}(u)=u^{\lambda}q^{\mu(j-i-\mu)}\qbin{i+\lambda}{\lambda}{q}\qbin{j-\mu}{\lambda-\mu}{q}$ and with $\gamma_{i,j,k}^{\mu,\lambda}(u)$ given by
\begin{align*}
\textstyle \left(\prod_{s=1}^{\lambda} \frac{[k+1-i-s]_q}{[j+1-i-\lambda-s]_{q,u}}\right)\left(\prod_{s=1}^{\lambda-\mu} \frac{[k+\mu-j+s]_q}{[j-i-2\lambda+s]_{q,u}}\right)\left(\prod_{s=1-\lambda}^{\min(i,k-j)} \frac{[j-i-\lambda+s]_{q,u}}{[k+1-\lambda-s]_{q,u}}\right)\left(\prod_{s=1}^{j-\lambda} \frac{[s-k-1]_{q,u}}{[s-i-\lambda-1]_{q,u}}\right)
\end{align*}
where $[x]_{q,u} = \frac{q^x u - q^{-x}}{q-q^{-1}}$ for $x \in \Z$. In particular, for $k=1$, we recover the well-known matrix
\begin{align*}
\frac{1}{1-uq^2}{\small \begin{pmatrix}
1-uq^2 & 0 & 0 & 0\\
0 & (1-q^2)u & q(1-u) & 0 \\
0 & q(1-u)& 1-q^2 & 0 \\
0 & 0 & 0 & 1-uq^2
\end{pmatrix} }
\end{align*}
in the ordered base $(v_0\otimes v_0, v_0\otimes v_1, v_1\otimes v_0, v_1\otimes v_1)$ of $V_1(u)\otimes V_1$ (and of $V_1\otimes V_1(u)$).
\end{example}
As the components $\mathcal{R}^0(u)$ and $\mathcal{R}^{\pm}(u)$ belong to $(\uqb \hat{\otimes} U_q(\bor^-))[[u]]$, the universal $R$-matrix can in practice be specialized on any tensor product $V\otimes W$ with $V$ a simple object of $\mathcal{O}$ and $W$ a finite-dimensional simple $\uqg$-module. The specialization induces a $R$-matrix $R_{V,W}(u)$ (which is an isomorphism of modules over $U_{q,u}(\bor)=\uqb\otimes \C(u)$) and Equation \eqref{eq:RInv} allows us to deduce the reciprocal braiding $$ R_{W,V}(u):W(u)\otimes V\arr V\otimes W(u).$$ However, this method cannot be used when both $V$ and $W$ are infinite-dimensional modules in $\mathcal{O}$ and the construction of a $U_{q,u}(\bor)$-linear isomorphism $V(u)\otimes W\simeq W\otimes V(u)$ is then extremely non-trivial. Such an isomorphism may in fact not exist as the next lemma shows. 
\begin{lem}[{\cite[Example 5.2]{her}}, see also \cite{bjmst}]\label{lem:ProdTensLpm} Let $V = L_{1,1}^+$ and $W = L_{1,1}^-$ for $\g =\su$. Then $V(u)\otimes W$ and $W \otimes V(u)$ are not isomorphic as modules over $U_{q,u}(\bor)$. 
\end{lem}
Fortunately, the construction of $R$-matrices can be extended to the subcategories $\mathcal{O}^{\pm}$ of $\mathcal{O}$ by using results of \cite{her} with the functor $\F_q$ of Section \ref{sec:Fq}. This is done in the next subsection. Note nevertheless that the existence of braidings for $\mathcal{O}^{\pm}$ should be expected from the (already mentioned) fact that tensor products of prefundamental representations of the same sign can be rewritten in any order (up to isomorphism). 
\subsection{Affine $R$-matrices}\label{sec:ROm} The strategy considered in \cite{her} for the construction of $R$-matrices in the subcategory $\mathcal{O}^-$ relies on a well-chosen inductive system. To outline this strategy, fix $V$ and $W$ simple objects of $\mathcal{O}^-$ with $(V_k)_{k\geq 1},(W_k)_{k\geq 1}$ finite-dimensional simple $\uqg$-modules such that $\lim_{k\rightarrow \infty}\overline{\chi}_q(V_k) = \overline{\chi}_q(V)$ and $\lim_{k\rightarrow \infty}\overline{\chi}_q(W_k)=\overline{\chi}_q(W)$ (see Proposition \ref{prop:limitOm}).\par
Note $\Psi_k$ and $\Psi_k'$ the highest $\ell$-weights of $V_k$ and $W_k$. Then, for $k\leq \ell$, we may consider the composition 
\begin{align*}
V_k(u)\otimes W_k &= L(\Psi_k)(u)\otimes L(\Psi_k') \arr L(\Psi_k)(u)\otimes L(\Psi'_k)\otimes L(\Psi_{\ell}\Psi_k^{-1})(u)\otimes L(\Psi'_{\ell}(\Psi'_k)^{-1})\\
&\arr L(\Psi_k)(u)\otimes L(\Psi_{\ell}\Psi_k^{-1})(u)\otimes L(\Psi'_k) \otimes L(\Psi'_{\ell}(\Psi'_k)^{-1})\\
&\arr L(\Psi_{\ell})(u)\otimes L(\Psi'_{\ell}) = V_{\ell}(u)\otimes W_{\ell}
\end{align*}
where the first and last arrows are analogous to the maps used for the definition of the system $\{F_{\ell,k}\}_{0\leq k\leq \ell}$ of Section \ref{sec:qcaract} and where the second arrow is the inverse of the rational $R$-matrix 
$$ L(\Psi_{\ell}\Psi_k^{-1})(u) \otimes L(\Psi'_k) \xrightarrow{\sim} L(\Psi'_k)\otimes L(\Psi_{\ell}\Psi_k^{-1})(u). $$
We denote this composition by $G_{\ell,k}(u)$. There is also an analogous map
$$ G_{\ell,k}'(u) : W_k\otimes V_k(u) \rightarrow W_{\ell} \otimes V_{\ell}(u). $$
By \cite[Section 5]{her}, the sets $\{G_{\ell,k}(u)\}_{0\leq k\leq \ell}$ and $\{G_{\ell,k}'(u)\}_{0\leq k\leq \ell}$ define inductive linear systems and allow the convergence of the action of the (asymptotic) subalgebra $\widetilde{U}_q(\g)$ of Section \ref{sec:qcaract} on the respective direct limits. This is enough to define a $\uqb$-action on these limits with
$$ \lim_{k \arr \infty} V_k(u)\otimes W_k \simeq V(u)\otimes W \quad \text{ and }\quad \lim_{k \arr \infty} W_k(u)\otimes V_k \simeq W(u)\otimes V.$$
\begin{lem}[{\cite[Section 5]{her}}]\label{lemma:IndRkOm} With the above notation, let $R_k(u):V_k(u)\otimes W_k \arr W_k\otimes V_k(u)$ be the rational $R$-matrix of Section \ref{sec:Rdimfin}. Then, for $k\leq \ell$,
$$ G_{\ell,k}'(u) \circ R_k(u) = R_{\ell}(u)\circ G_{\ell,k}(u). $$
\end{lem}
Write $G_{\infty,k}(u) : V_k(u)\otimes W_k \arr V(u)\otimes W$ and $G_{\infty,k}'(u):W_k\otimes V_k(u)\arr W \otimes V(u)$ for the canonical maps associated to the systems $\{G_{\ell,k}(u)\}_{0\leq k\leq \ell}$ and $\{G_{\ell,k}'(u)\}_{0\leq k\leq \ell}$. Then the above lemma and the usual universal property of direct limits give a unique map $$R(u) : V(u)\otimes W \rightarrow W\otimes V(u)$$ making the following diagram commute \\[-2.25mm]
\begin{equation}\label{eq:CommDiagRmat}
\begin{tikzcd}[row sep = 1.5em, column sep = 3em]
V_k(u)\otimes W_k \arrow[dd,"G_{\ell,k}(u)",swap]\arrow[rrr,"R_k(u)"]\arrow[dr,"G_{\infty,k}(u)",swap] & & & W_k\otimes V_k(u)\arrow[dd,"G_{\ell,k}'(u)"]\arrow[dl,"G_{\infty,k}'(u)"]\\
& V(u)\otimes W\arrow[r,dashed,"\exists! R(u)"] & W \otimes V(u) & \\
V_{\ell}(u)\otimes W_{\ell}\arrow[rrr,"R_{\ell}(u)",swap]\arrow[ur,"G_{\infty,\ell}(u)",swap] & & & W_{\ell}\otimes V_{\ell}(u)\arrow[ul,"G_{\infty,\ell}'(u)"]
\end{tikzcd}
\end{equation}\\[-0.5mm]
We say that this (unique) induced map $R(u)$ is an \textit{affine $R$-matrix}. It acts rationally on the product $V(u)\otimes W$ with at most countably infinite poles. (This follows from the commutative diagram above with the fact that $G_{\ell,k}(u)$ and $G_{\ell,k}'(u)$ are rational maps.)
\begin{theorem}[{\cite[Section 5]{her}}]\label{thm:ROmLin}
The application $R(u)$ is a linear isomorphism for the action of $U_{q,u}(\bor) = \uqb\otimes \C(u)$. It also induces a $\uqb$-linear isomorphism $V(a)\otimes W \simeq W\otimes V(a)$ whenever $a\in \C^{\times}$ is neither a pole of $R(u)$ nor of its inverse.
\end{theorem}
Braidings for the subcategory $\mathcal{O}^-$ can therefore be obtained from the $R$-matrices of finite-dimensional representations. This is nevertheless highly non-trivial as the maps $G_{\ell,k}(u)$ and $G_{\ell,k}'(u)$ defining the inductive systems above are not $\uqb$-linear. They however verify
$$ \displaystyle x\cdot (v\otimes w) = \lim_{k\arr \infty} G_{\infty,k}(u)(x\cdot G_{\infty,k}^{-1}(u)(v\otimes w))$$ 
and 
$$\displaystyle x\cdot (w\otimes v) = \lim_{k\arr \infty} G_{\infty,k}'(u)(x\cdot (G_{\infty,k}'(u))^{-1}(w\otimes v))$$
for any $v\in V(u)$, $w\in W$ and any $x$ in the intersection of $\uqb$ with the subalgebra $\widetilde{U}_q(\g)$ of Section \ref{sec:qcaract}. In particular, the limits appearing in these equalities are well-defined.
\begin{example}\label{ex:RmatLm} Let $\g=\su$ and recall the basis $\{v_{i,k}\}_{i=0}^k$ of $V_k = W_k= W_{k,aq^{1-2k}}^{(1)}$ described in Example \ref{ex:KRsl2}. Then the image of $v_{i,k}\otimes v_{j,k}$ under $G_{\ell,k}(u):V_k(u)\otimes W_k \arr V_{\ell}(u)\otimes W_{\ell}$ is 
$$ \sum_{\nu = 0}^{\min(j,\ell-k)} {\textstyle u^{\nu}q^{\nu(j-i-\nu)}\qbin{i+\nu}{\nu}{q}\left(\prod_{s=1}^{j-\nu}\frac{[s-k-1]_{q,u}}{[s-\ell-1]_{q,u}}\right)\left(\prod_{s=1}^{\nu}\frac{[s+k-\ell-1]_{q}}{[j-\ell-s]_{q,u}}\right)(v_{i+\nu,\ell}\otimes v_{j-\nu,\ell})}$$
where again $[x]_{q,u} = \frac{q^x u-q^{-x}}{q-q^{-1}}$. Also, the limit of the system $\{G_{\ell,k}(u)\}_{0\leq k\leq \ell}$ is isomorphic to $V(u)\otimes W$ with $V = W= L_{1,a}^-$. Let us consider the basis $\{z_i\}_{i\geq 0}\subseteq V$ of Example \ref{ex:Lpmsl2}. Then the image of $v_{i,k}\otimes v_{j,k}$ under the map $G_{\infty,k}(u):V_k(u)\otimes W_k \rightarrow V(u)\otimes W$ is given by
$$ \sum_{0\leq \nu\leq j}{\textstyle u^{\nu}(q^{-1}-q)^{j-\nu}q^{-\nu(i+k)}q^{\frac{1}{2}(j-\nu)(j+3\nu-1)}\qbin{i+\nu}{\nu}{q}\left(\prod_{s=1}^{j-\nu}[s-k-1]_{q,u}\right)
}(z_{i+\nu}\otimes z_{j-\nu}).$$
One can finally verify that the unique map $R(u)$ making the diagram \eqref{eq:CommDiagRmat} commutative is 
$$R(u)(z_i\otimes z_j) = \sum_{0\leq \nu\leq i}{\textstyle\frac{1}{u^{i}} (q-q^{-1})^{i-\nu}q^{\frac{1}{2}(i-\nu)(i-2j+3\nu-1)}\qbin{i+j-\nu}{j}{q} \left(\prod_{s=\nu-i+1}^0 [s]_{q,u}\right)}(z_{\nu}\otimes z_{i+j-\nu})$$
and that this map is indeed a $U_{q,u}(\bor)$-linear isomorphism. Remark that it has no poles on $\C^{\times}$ and is hence holomorphic in the spectral parameter $u$. It therefore produces a $\uqb$-linear isomorphism $L_{1,a}^-\otimes L_{1,b}^{-}\simeq L_{1,b}^{-}\otimes L_{1,a}^-$ for any $b \in \C^{\times}$ (with $u$ evaluated at $b/a$).
\end{example}
Let us now emphasize the quantum parameter as in Section \ref{sec:Fq} and assume that $V,W$ are simple objects of the positive subcategory $\mathcal{O}^+_q$. Then the functor $\F_{q^{-1}}$ of the above-mentioned section sends $V$ and $W$ to simple objects $V'=\F_{q^{-1}}(V)$ and $W' = \F_{q^{-1}}(W)$ of the negative subcategory $\mathcal{O}^-_{q^{-1}}$ of $\mathcal{O}_{q^{-1}}$. There is thus an affine $R$-matrix 
$$ R'(u) : W'(u)\otimes V' \arr V'\otimes W'(u)$$ 
which can also be considered as a $U_{q,u}(\bor)$-linear isomorphism from the module $\F_q\big(W'(u)\otimes V'\big)$ to $\F_q\big(V'\otimes W'(u)\big)$ (as the functor $\F_q$ acts as the identity on morphisms). \par Moreover, as $\F_{q^{-1}}$ is the inverse of $\F_q$, Proposition \ref{prop:Fshift} and Theorem \ref{thm:Comon} give
$$  \F_q\big(W'(u)\otimes V'\big) \simeq (\F_q\circ\F_{q^{-1}})(V)\otimes ((\F_q\circ\F_{q^{-1}})(W))(u) \simeq V\otimes W(u)$$
and, similarly, $\F_q\big(V'\otimes W'(u)\big) \simeq W(u)\otimes V$. Making explicit the isomorphisms used above and inverting the spectral parameter $u$, we deduce that (with the notation of Theorem \ref{thm:Comon})
\begin{equation}\label{eq:RmatOp}
 R(u) = \tau \circ R'(u^{-1}) \circ \tau : V\otimes W(u^{-1}) \arr W(u^{-1})\otimes V
\end{equation}
is a $U_{q,u}(\bor)$-linear isomorphism. The automorphism $\tau_{u,q}$ of Section \ref{sec:DefQaff} then enables us to view this map as an isomorphism from $V(u)\otimes W $ to $W\otimes V(u)$. We will call the latter isomorphism an \textit{affine $R$-matrix} just like the braidings induced from \eqref{eq:CommDiagRmat}. It acts rationally on $V(u)\otimes W$ with the same set of poles as the map $R'(u^{-1})$. This set is hence at most countably infinite by the above comments.\par 
Let us summarize the results of the last paragraphs in a proper theorem. 
\begin{theorem}\label{thm:RmatOp} Let $V$ and $W$ be simple objects of $\mathcal{O}^+_q$. Then $R(u) = \tau \circ R'(u^{-1}) \circ \tau$ produces a $U_{q,u}(\bor)$-linear isomorphism from $V(u)\otimes W$ to $W\otimes V(u)$ where $$R'(u): (\F_{q^{-1}}(W))(u)\otimes \F_{q^{-1}}(V)\arr \F_{q^{-1}}(V)\otimes (\F_{q^{-1}}(W))(u)$$ is the affine $R$-matrix of $\mathcal{O}_{q^{-1}}^-$ induced from diagram \eqref{eq:CommDiagRmat}. \hfill $\qed$
\end{theorem}
\begin{example}\label{ex:RmatLp} Let $\g = \su$ with $\{w_i\}_{i\geq 0}$ the basis for $V = W = L_{1,a}^{+,q}$ given in Example \ref{ex:Lpmsl2}. Then it follows from Example \ref{ex:FqLp}, Example \ref{ex:RmatLm} and Theorem \ref{thm:RmatOp} that the affine $R$-matrix $R(u):V(u)\otimes W \rightarrow W\otimes V(u)$ sends $w_i\otimes w_j$ to
$$ \sum_{0\leq \nu\leq j}{\textstyle u^j(q^{-1}-q)^{j-\nu}q^{-\frac{1}{2}(j-\nu)(j-2i+3\nu-1)}\qbin{i+j-\nu}{i}{q}\left(\prod_{s=\nu-j+1}^0 [s]_{q^{-1},u^{-1}}\right)}\varrho_{i,j,\nu} w_{i+j-\nu}\otimes w_{\nu}.$$
where the factor $\varrho_{i,j,\nu} = q^{i(i-1)+j(j-1)-(i+j-\nu)(i+j-\nu-1)-\nu(\nu-1)} = q^{-2(i-\nu)(j-\nu)}$ comes from the isomorphism $V\simeq \F_q(L_{1,a q^2}^{-,q^{-1}})$ of Example \ref{ex:FqLp}. The latter expression can be simplified as 
$$ \sum_{0\leq \nu\leq j}{\textstyle u^{\nu}(q^{-1}-q)^{j-\nu}q^{-\frac{1}{2}(j-\nu)(j+2i-\nu-1)}\qbin{i+j-\nu}{i}{q}\left(\prod_{s=\nu-j+1}^0 [s]_{q,u}\right)}w_{i+j-\nu}\otimes w_{\nu}$$
and it is then easy to verify that $R(u)$ is a $U_{q,u}(\bor)$-linear isomorphism.
\end{example}
We now end this section by giving some basic properties of the affine $R$-matrices of $\mathcal{O}^{\pm}$. 
\begin{prop} Fix $V,W$ simple objects of $\mathcal{O}_q$. Assume that both $V$ and $W$ belong to $\mathcal{O}_q^{-}$ (or $\mathcal{O}_q^+$) and let $R_{V,W}(u):V(u)\otimes W\rightarrow W\otimes V(u)$ be the affine $R$-matrix defined above. Then
\begin{itemize}
\item[(i)] $R_{V,W}(u)$ satisfies the normalization condition \eqref{eq:NormR},
\item[(ii)] The affine $R$-matrix $R_{W,V}(u):W(u)\otimes V\rightarrow V\otimes W(u)$ is equal to $R_{V,W}^{-1}(u^{-1})$.
\end{itemize}
Also $R_{V,W}(u)$ is the rational $R$-matrix of Section \ref{sec:Rdimfin} if both $V$ and $W$ are finite-dimensional.
\end{prop}
\begin{proof} Part (i) is easily proven. For part (ii), assume that $V,W$ are objects of $\mathcal{O}^-_q$ and write
\begin{center}
$G_{\ell,k}(u):V_k(u)\otimes W_k \arr V_{\ell}(u)\otimes W_{\ell}$, \qquad $G_{\ell,k}'(u):W_k\otimes V_k(u)\arr W_{\ell}\otimes V_{\ell}(u),$ \\
$H_{\ell,k}(u):W_k(u)\otimes V_k \arr W_{\ell}(u)\otimes V_{\ell},$ \qquad $H_{\ell,k}'(u):V_k\otimes W_k(u) \arr V_{\ell}\otimes W_{\ell}(u)$
\end{center}
for the direct systems underlying the definition of the affine $R$-matrices $R_{V,W}(u)$ and $R_{W,V}(u)$. A straightforward comparison of these inductive systems yields
$$G_{\ell,k}(u^{-1})=H_{\ell,k}'(u) \quad \text{and} \quad G_{\ell,k}'(u^{-1}) = H_{\ell,k}(u)$$
so that
$G_{\infty,k}(u^{-1}) = H'_{\infty,k}(u)$ and $G_{\infty,k}'(u^{-1})=H_{\infty,k}(u)$.
Part (ii) for $\mathcal{O}_q^-$ then follows from Equation \eqref{eq:RInv} and from the unicity underlying the definition of the affine $R$-matrices in \eqref{eq:CommDiagRmat}. The corresponding result for $\mathcal{O}^+_q$ is afterwards a direct consequence of the construction \eqref{eq:RmatOp}. Finally, the last part of the theorem follows from the unicity of rational $R$-matrices.
\end{proof}
\subsection{Factorization through stable maps}\label{sec:Stab} A fundamental problem in the theory of quantum loop algebras is concerned with tensor products of $\ell$-weight vectors. Indeed, the notion of $\ell$-weight vectors does not behave well with respect to the coproduct $\Delta$ of $\uqb$ and tensor products of such vectors for given representations of $\uqb$ are typically not $\ell$-weight vectors for the tensor product of these representations. This gives rise to many technical difficulties and makes in particular more involved the study of the monoidal structure of $\mathcal{O}=\mathcal{O}_q$. There is however a way to compensate for these difficulties using the notion of stable maps. \par
Consider the ordering $\preceq$ on $\ttt^{\times}\times\ttt^{\times}$ induced from the one $\leq$ of $\ttt^{\times}$ by $(\omega_1,\omega_2)\preceq (\omega_1',\omega_2')$ if and only if $\omega_1\omega_2 = \omega_1'\omega_2'\text{ and } \omega_1\leq \omega_1'$. Then, given $V$ and $W$ modules in $\mathcal{O}$ with $\ell$-weight vectors $v\in V$ and $w \in W$, one can define the following subspace of $V(u)\otimes W$:
$$ (v\otimes w)_{\prec} = \sum_{(\omega_1,\omega_2)\prec (\varpi(\Psi),\varpi(\Psi'))} V(u)_{\omega_1}\otimes W_{\omega_2}$$
with $\Psi$ and $\Psi'$ the respective $\ell$-weights of $v$ and $w$. Using this notation, we can follow \cite{her} and define a stable map for $V$ and $W$ as a $\C$-linear isomorphism 
$$S_{V,W}(u):V(u)\otimes W \arr V(u)\otimes W$$
which is rational in the spectral parameter $u$ \textbf{and} is such that $S_{V,W}(u)(v\otimes w)$ belongs to the subspace $(v\otimes w)_{\preceq}=v\otimes w+(v\otimes w)_{\prec}$ \textbf{and} is a $\ell$-weight vector of $\ell$-weight $\Psi(u)\Psi'$ whenever $v\in V(u)$ and $w\in W$ are $\ell$-weight vectors of respective $\ell$-weight $\Psi(u)$ and $\Psi'$. Remark that these stable maps conjecturally generalize the (geometric) stable envelopes defined in \cite{mo} using equivariant $K$-theory of Nakajima quiver varieties (for $\g$ of type ADE). \par 
Such a map always exists but may not be unique. It is nevertheless uniquely determined if $V$ and $W$ are irreducible with at least one of them belonging to the subcategory $\mathcal{O}^-$. Moreover, in this case, the (unique) map $S_{V,W}(u)$ induces a rational $\uqhp$-linear isomorphism (cf.~\cite{her}) $$ S_{V,W}(u):V(u)\otimes_d W \rightarrow V(u)\otimes W$$
where $V(u)\otimes_d W$ is the $\uqhp$-module obtained by using the \textit{Drinfeld coproduct} $$\Delta_d:\uqhp\arr\uqhp\otimes\uqhp$$
which is defined by $\Delta_d(k_i)=k_i\otimes k_i$ and $\Delta_d(h_{i,r})=h_{i,r}\otimes 1+1\otimes h_{i,r}$.\par
A deep result of Hernandez is the following factorization for the affine $R$-matrices of $\mathcal{O}^-$.
\begin{theorem}[{\cite[Section 5]{her}}]\label{thm:RmatFacOm} Fix $V,W$ simple objects of $\mathcal{O}^-$ with $R_{V,W}(u)$ the corresponding affine $R$-matrix (obtained as in Section \ref{sec:ROm}). Then 
\begin{equation}\label{eq:FacRmatOm} R_{V,W}(u) = S_{W,V}(u^{-1})\circ \tau\circ \alpha(u)\circ S_{V,W}^{-1}(u)
\end{equation}
for some (unique) $\uqhp$-linear automorphism  $\alpha(u)$ of the module $V(u)\otimes_d W$. 
\end{theorem}
The map $\alpha(u)$ above can be constructed explicitly using the abelian component $\mathcal{R}^0(u)$ of the universal $R$-matrix $\mathcal{R}(u)$. Equation \eqref{eq:FacRmatOm} may thus be seen as an explicit factorization for $R_{V,W}(u)$. This is a remarkable result since it simplifies greatly the computation of braidings for $\mathcal{O}^-$. (Using \eqref{eq:CommDiagRmat} for the explicit computation of an affine $R$-matrix of $\mathcal{O}^-$ is typically very hard.) 
It can also be used in more abstract proofs (see e.g.~the proof of \cite[Theorem 5.16]{her}).\par 
We now finish this paper by giving a similar factorization for the affine $R$-matrices of $\mathcal{O}^+$. For this, let us emphasize again the quantum parameter $q$ and fix simple objects $V,W$ in $\mathcal{O}^+_q$. Then Theorems \ref{thm:RmatOp} and \ref{thm:RmatFacOm} give the following factorization for the $R$-matrix $R_{V,W}(u)$
\begin{align*} R_{V,W}(u) &= \tau \circ S_{V',W'}(u)\circ \tau \circ \alpha(u^{-1})\circ S_{W',V'}^{-1}(u^{-1}) \circ \tau\\
&= (\tau \circ S_{V',W'}(u) \circ \tau)\circ\tau \circ (\tau \circ \alpha(u^{-1})\circ\tau) \circ (\tau \circ S_{W',V'}^{-1}(u^{-1})\circ \tau)
\end{align*}
with $V' = \F_{q^{-1}}(V)$, $W' = \F_{q^{-1}}(W)$ and $\alpha(u)$ a $U_{q^{-1}}(\mathfrak{h}^+)$-linear automorphism of $W'(u)\otimes_d V'$. (The stable maps above are also $U_{q^{-1}}(\mathfrak{h}^+)$-linear as $V',W'$ are simple objects of $\mathcal{O}_{q^{-1}}$.) \par 
Let us write $\beta(u) = \tau\circ \alpha(u^{-1})\circ \tau$ and $ \mathcal{S}_{V,W}(u) = \tau \circ S_{W',V'}(u^{-1})\circ \tau$ so that 
\begin{equation}\label{eq:FacRmatOp2}
R_{V,W}(u) = \mathcal{S}_{W,V}(u^{-1}) \circ \tau \circ \beta(u) \circ \mathcal{S}_{V,W}^{-1}(u).
\end{equation}
Then the automorphism $\tau_{u,q}$ of Section \ref{sec:Qaff} and the functor $\F_q$ of Section \ref{sec:Fq} (with Theorem \ref{thm:Comon}) allow us to view the map $\mathcal{S}_{V,W}(u)$ above as being a rational isomorphism from $\F_q(V'(u)\otimes_d W')$ unto $V(u)\otimes W$ which is linear\footnote{Remark that $\tau:V'(u)\otimes_d W'\arr W'\otimes_d V'(u)$ is $U_{q^{-1}}(\mathfrak{h}^+)$-linear by definition of the coproduct $\Delta_d$.} for the action of the subalgebra $\mathfrak{X}_q=\pp_{q^{-1}}(U_{q^{-1}}(\mathfrak{h}^+))$ of $\uqb$. The map $\beta(u)$ may also be seen as a $\mathfrak{X}_q$-linear automorphism of $\F_q(V'(u)\otimes_d W')$. \par
We wish to compare the factorizations \eqref{eq:FacRmatOm} and \eqref{eq:FacRmatOp2}. In that perspective, define $$\Delta_{d,q}'=(\pp_q^{-1} \otimes \pp_q^{-1})\circ \Delta_{d,q^{-1}}\circ \pp_q : \mathfrak{X}_q \arr \mathfrak{X}_q\otimes \mathfrak{X}_q$$ where $\Delta_{d,q^{-1}}$ is the Drinfeld coproduct associated to the algebra $U_{q^{-1}}(\mathfrak{h}^+)$. Let us furthermore denote $V(u)\otimes'_{d}W$ and $W\otimes'_{d}V(u)$ the $\mathfrak{X}_q$-modules induced from the coproduct $\Delta'_{d,q}$. Then it follows from the (trivial) relation $(\pp_q\otimes \pp_q)\circ \Delta_{d,q}'=\Delta_{d,q^{-1}} \circ \pp_q$ that the $\mathfrak{X}_q$-modules $V(u)\otimes_d' W$ and $\F_q(V(u)\otimes_d W)$ are naturally isomorphic. We can thus view the factorization \eqref{eq:FacRmatOp2} as 
\begin{equation}\label{eq:FacRMatOp3}
V(u)\otimes W \xrightarrow{\mathcal{S}_{V,W}^{-1}(u)} V(u)\otimes'_d W \xrightarrow{\beta(u)} V(u)\otimes_d' W \xrightarrow{\,\tau\,} W \otimes'_d V(u) \xrightarrow{\mathcal{S}_{W,V}(u^{-1})} W\otimes V(u)
\end{equation}
Note that all maps appearing in \eqref{eq:FacRMatOp3} are $\mathfrak{X}_q$-linear. We hence have that
\begin{itemize}
\item[$\bullet$] $\mathcal{S}_{V,W}(u)$ is a rational $\mathfrak{X}_q$-linear isomorphism of $V(u)\otimes_d' W$ unto $V(u)\otimes W$,
\item[$\bullet$] $\beta(u)$ is a $\mathfrak{X}_q$-linear automorphism of $V(u)\otimes_d' W$.
\end{itemize}
Moreover, the map $\mathcal{S}_{V,W}(u)$ verifies a triangularity property similar to that of $S_{V,W}(u)$. Indeed, let $v \in V$ and $w\in W$ be $\ell$-weight vectors of $\ell$-weights $\Psi$ and $\Psi'$ (respectively). Note also $\preceq'$ the partial ordering on $\ttt^{\times}\times \ttt^{\times}$ defined by $(\omega_1,\omega_2)\preceq'(\omega_1',\omega_2')$ if and only if $\omega_1\omega_2=\omega_1'\omega_2'$ and $\omega_1\unlhd \omega_1'$ (where $\unlhd$ is the partial ordering of Section \ref{sec:Fq}). Then the triangularity property of the stable map $S_{W',V'}(u^{-1})$ (of $U_{q^{-1}}(\bor)$) gives 
$$ \mathcal{S}_{V,W}(u)(v\otimes w) \in \tau(w\otimes v+(w\otimes v)_{\prec'})$$
where the subspace $(w\otimes v)_{\prec'}\subseteq W\otimes V(u)$ is defined by $$(w\otimes v)_{\prec'} = \sum_{(\omega_1,\omega_2)\prec'(\varpi(\Psi'),\varpi(\Psi))}W_{\omega_1}\otimes V(u)_{\omega_2}.$$
In addition, an easy computation (using the fact that $\leq$ and $\unlhd$ are reciprocal) shows that $$ \tau(w\otimes v)_{\prec'}\subseteq (v\otimes w)_{\prec}$$
so that $\mathcal{S}_{V,W}(u)(v\otimes w) \in v\otimes w+(v\otimes w)_{\prec}=(v\otimes w)_{\preceq}$. This is the said triangularity property. \par We thus conclude that the factorizations given in \eqref{eq:FacRmatOm} and \eqref{eq:FacRmatOp2} are perfectly analogous, the only real difference between them being the use of the subalgebra $\mathfrak{X}_q$ of $\uqb$ for \eqref{eq:FacRmatOp2} instead of the usual Cartan--Drinfeld subalgebra $\uqhp$ (that is used for \eqref{eq:FacRmatOm}). 
\begin{example} Consider $\g = \su$ and $V = L_{1,a}^{-,q}$ with the basis $\{z_i\}_{i\geq 0}\subseteq V$ of Example \ref{ex:Lpmsl2}. Then the image of $z_i\otimes z_j$ under the stable map $S_{V,V}(u):V(u)\otimes_d V\arr V(u)\otimes V$ is (cf.~\cite{her})
$$ S_{V,V}(u)(z_i\otimes z_j) = \sum_{0\leq \lambda \leq j} {\textstyle u^{\lambda}(q^{-1}-q)^{-\lambda} \qbin{i+\lambda}{\lambda}{q}}\,\frac{q^{\frac{\lambda}{2}(1-3\lambda)}q^{\lambda(j-2i)}}{\prod_{s=1}^{\lambda}[j-i-s]_{q,u}}\,z_{i+\lambda}\otimes z_{j-\lambda}.$$
The inverse map is easily computed to be
$$ S_{V,V}^{-1}(u)(z_i\otimes z_j) = \sum_{0\leq \lambda \leq j} {\textstyle u^{\lambda}(q-q^{-1})^{-\lambda} \qbin{i+\lambda}{\lambda}{q}}\,\frac{q^{\frac{\lambda}{2}(1-3\lambda)}q^{\lambda(j-2i)}}{\prod_{s=1}^{\lambda}[j-i-\lambda+1-s]_{q,u}}\,z_{i+\lambda}\otimes z_{j-\lambda} $$
and one can express the affine $R$-matrix $R_{V}(u):V(u)\otimes V\arr V\otimes V(u)$ of Example \ref{ex:RmatLm} as $$R_{V}(u) = S_{V,V}(u^{-1})\circ\tau\circ \alpha(u)\circ S_{V,V}^{-1}(u)$$ with $\alpha(u)$ the $U_q(\mathfrak{h}^+)$-linear automorphism of $V(u)\otimes_d V$ given by
$$ \alpha(u)(z_i\otimes z_j) = \frac{(-1)^j}{u^{i}} q^{\frac{1}{2}(i-j)(i+j-1)}(q - q^{-1})^{i - j}\,\frac{\prod_{s=1}^i[j-s+1]_{q,u}}{\prod_{s=1}^j [s-i-1]_{q,u}}\,z_i\otimes z_j.$$
The corresponding factorization for the affine $R$-matrix $R_{W}(u):W(u)\otimes W\arr W\otimes W(u)$ of Example \ref{ex:RmatLp} (with $W = L_{1,a}^{+,q}$) is 
$$ R_W(u) = \mathcal{S}_{W,W}(u^{-1})\circ \tau\circ \beta(u)\circ \mathcal{S}_{W,W}(u) $$
where the maps $\mathcal{S}_{W,W}(u)$, $\mathcal{S}_{W,W}^{-1}(u)$ and $\beta(u)$ are respectively given (on the basis $\{w_i\}_{i\geq 0}\subseteq W$ of Example \ref{ex:Lpmsl2}) by 
$$ \mathcal{S}_{W,W}(u)(w_i\otimes w_j) = \sum_{0\leq \lambda \leq i} {\textstyle (q-q^{-1})^{-\lambda} \qbin{j+\lambda}{\lambda}{q}}\,\frac{q^{-\frac{\lambda}{2}(1-2i+\lambda)}}{\prod_{s=1}^{\lambda}[i-j-s]_{q,u}}\,w_{i-\lambda}\otimes w_{j+\lambda},$$ 
$$ \mathcal{S}_{W,W}^{-1}(u)(w_i\otimes w_j) = \sum_{0\leq \lambda \leq i} {\textstyle (q^{-1}-q)^{-\lambda} \qbin{j+\lambda}{\lambda}{q}}\,\frac{q^{-\frac{\lambda}{2}(1-2i+\lambda)}}{\prod_{s=1}^{\lambda}[i-j-\lambda+1-s]_{q,u}}\,w_{i-\lambda}\otimes w_{j+\lambda}$$
and finally
$$ \beta(u)(w_i\otimes w_j)=(-u)^i q^{-\frac{1}{2}(j-i)(j+i-1)}(q^{-1} - q)^{j - i}\,\frac{\prod_{s=1}^j [i-s+1]_{q,u}}{\prod_{s=1}^i [s-j-1]_{q,u}}\,w_i\otimes w_j.$$
(Note that we have implicitly used the isomorphism $W \simeq \F_q(L_{1,a q^{2}}^{-,q^{-1}})$ of Example \ref{ex:FqLp}.)
\end{example}
We end this paper by remarking that the braidings obtained for the categories $\mathcal{O}_q^{\pm}$ (in this section and in \cite{her}) could possibly be used in order to extend the notion of generalised quantum affine Schur--Weyl duality introduced by Kang--Kashiwara--Kim in \cite{kkk} (see also \cite{fuj}). Indeed, the typical method underlying the construction of these dualities is to construct first a quiver-Hecke (or Khovanov--Laura--Rouquier) algebra using the poles of a finite collection $\{R_{i,j}(u):V_i(u)\otimes V_j\arr V_j\otimes V_i(u)\}_{i,j\in \mathcal{J}}$ of $R$-matrices of finite-dimensional simple real (i.e.~of tensor square simple) modules $\{V_i\}_{i\in \mathcal{J}}$. A naive question is therefore whether we can construct analogous dualities and algebras by replacing the finite-dimensional $\{V_i\}_{i\in \mathcal{J}}$ by (well-behaved enough) objects of $\mathcal{O}^{\pm}_q$. This is an interesting problem as the generalized dualities of Kang--Kashiwara--Kim are extremely important in the study of the finite-dimensional representation theory of quantum affine algebras (see e.g.~\cite{fuj,kko}). Extending these dualities could thus lead to highly non-trivial results regarding the subcategories $\mathcal{O}^{\pm}_q$. \par 
 
 Finally, a natural question arising from our work is whether or not the functor $\F_q$ of Section \ref{sec:Fq} can be extended to the setting of shifted quantum affine algebras (and to their truncations, as defined in \cite{hshift}). This can be done easily by using the functors introduced by Hernandez in \cite[Section 7]{hshift} which relate the category $\mathcal{O}_{\mu}$ of modules over the shifted algebra $U_q^{\mu}(\g)$ with the category $\mathcal{O}_q$ (linked to $\uqb$). It is however not clear if the resulting functor is exact or if it behaves well with respect to the fusion product of $U_q^{\mu}(\g)$ (defined in \cite[Section 5]{hshift}).

\appendix
\section{Extension of Hernandez--Leclerc's duality}\label{app:K0D}
This appendix is devoted to the proof of Theorem \ref{thm:DExt}. We fix $q\in \C^{\times}$ which is not a root of unity and use the notation of Section \ref{sec:Fq}. One could however work with $q$ a formal variable and replace the base field $\C$ by $\mathbbm{k}_q$ as in Section \ref{sec:FuncD}. Denote by $\mathcal{E}_{\ell,q^{\pm 1}}$ the set defined in Section \ref{sec:qcaract} for the parameter $q^{\pm 1}$. (Recall that the definition of $\mathcal{E}_{\ell,q}$ involves the order $\leq$.)\par 
Consider the involutive group automorphism $\vartheta$ of $\mathfrak{r}$ given on the highest $\ell$-weights of positive prefundamental and invertible representations by $\Psi_{i,a}\mapsto \Psi_{i,a^{-1}}$ and $\mu \mapsto \mu^{-1}$ ($i\in I$, $a\in\C^{\times}$ and $\mu \in \ttt^{\times}$). Fix $c : \mathfrak{r}\arr \Z$ in $\mathcal{E}_{\ell,q^{-1}}$ and denote by $I_q(c)$ the composite map $I_q(c) = c\circ\vartheta$. 
\begin{prop} The map $I_q(c):\mathfrak{r}\arr \Z$ belongs to $\mathcal{E}_{\ell,q}$.
\end{prop}
\begin{proof}
Since $c$ belongs to $\mathcal{E}_{\ell,q^{-1}}$, we can find $\lambda_1,...,\lambda_r\in \ttt^{\times}$ such that $$\textstyle \{\varpi(\Psi)\,|\,\Psi\in\mathfrak{r},\, c(\Psi)\neq 0\}\subseteq \bigcup_{i=1}^r\{\mu\in\ttt^{\times}\,|\,\mu\unlhd\lambda_i\}$$
where $\unlhd$ is the order of Section \ref{sec:Fq}. Hence $(I_q(c))(\Psi)=(c\circ \vartheta)(\Psi)\neq 0$ for some $\Psi \in \mathfrak{r}$ implies $\varpi(\vartheta(\Psi)) \unlhd \lambda_i$ for at least one $i\in \{1,...,r\}$. However, $\varpi(\vartheta(\Psi))=(\varpi(\Psi))^{-1}$ (as one can easily check on the highest $\ell$-weights of positive prefundamental and invertible representations) and it thus follows that $\varpi(\Psi) \leq \lambda_i^{-1}$ (as the orders $\leq$ and $\unlhd$ are mutually reciprocal). Therefore,
$$ \textstyle \{\varpi(\Psi)\,|\,\Psi\in\mathfrak{r},\, (I_q(c))(\Psi)\neq 0\}\subseteq \bigcup_{i=1}^r\{\mu\in\ttt^{\times}\,|\,\mu\leq\lambda_i^{-1}\}=\bigcup_{i=1}^r D(\lambda_i^{-1}) $$
and $I_q(c)$ satisfies the first property underlying the definition of $\mathcal{E}_{\ell,q}$. For the second one, fix $\mu\in \ttt^{\times}$ and use the precedent results to deduce that the set
$$ \textstyle \{\Psi\in \mathfrak{r}\,|\,(I_q(c))(\Psi)\neq 0,\, \varpi(\Psi) = \mu\} = \vartheta(\{\Psi\in\mathfrak{r}\,|\, c(\Psi)\neq 0, \, \varpi(\Psi)=\mu^{-1} \})$$
is finite (as $c$ belongs to $\mathcal{E}_{\ell,q^{-1}}$). This ends the proof.
\end{proof} 
The correspondence $I_q:c\mapsto c\circ\vartheta$ thus induces a map from $\mathcal{E}_{\ell,q^{-1}}$ to $\mathcal{E}_{\ell,q}$. This is a ring isomorphism. Indeed, $I_{q^{-1}}= I_q^{-1}$ as $\vartheta$ is involutive and the only non-trivial thing to verify is that $I_q$ is compatible with convolution products. This follows from the computation 
\begin{align*}
 (I_q(c_1\cdot c_2))(\Psi) = \sum_{\Psi_1\Psi_2 = \vartheta(\Psi)}c_1(\Psi_1)c_2(\Psi_2) =  \sum_{\Psi_1\Psi_2 {=} \Psi}c_1(\vartheta(\Psi_1))c_2(\vartheta(\Psi_2)) = (I_q(c_1)\cdot I_q(c_2))(\Psi)
\end{align*}
where the substitutions $\Psi_1 \mapsto \vartheta(\Psi_1)$ and $\Psi_2\mapsto \vartheta(\Psi_2)$ were used for the second equality.\par Write $[\Psi]$ as in Section \ref{sec:qcaract} for the element of $\mathcal{E}_{\ell,q^{-1}}\cap\mathcal{E}_{\ell,q}$ given by $[\Psi](\Psi') = \delta_{\Psi,\Psi'}$. Then $I_q([\Psi]) = [\vartheta(\Psi)]$ for any $\Psi \in \mathfrak{r}$. We will construct a ring isomorphism $\tilde{I}_q:K_0(\mathcal{O}_q)\simeq K_0(\mathcal{O}_{q^{-1}})$ making the following diagram commute: 
\begin{equation}\label{eq:DiagD} \tag{$\star$}
\begin{tikzcd}
K_0(\mathcal{O}_{q^{-1}}) \ar[r, "\tilde{I}_q"] \ar[d,"\chi_q",swap] & K_0(\mathcal{O}_{q}) \ar[d,"\chi_q"] \\
\mathcal{E}_{\ell,{q^{-1}}} \ar[r,"I_q"]& \mathcal{E}_{\ell,q}
\end{tikzcd}
\end{equation}
By Theorem \ref{thm:qcaract}, it is enough to show that $I_q\circ \chi_q(K_0(\mathcal{O}_{q^{-1}})) \subseteq \chi_q(K_0(\mathcal{O}_{q}))$. For this, consider the following lemma which is a direct consequence of Theorem \ref{prop:qcharprefond} and the definition of $\vartheta$.
\begin{lem}\label{lem:IqPos} Let $\mu\in \ttt^{\times}$, $i\in I$ and $a\in \C^{\times}$. Then 
$$
I_q\circ \chi_q([\mu]) = \chi_q([\mu^{-1}])\text{ and }I_q\circ \chi_q(L_{i,a}^{+,q^{-1}}) = \chi_q(L_{i,a^{-1}}^{+,q}).\eqno\qed
$$
\end{lem}
We denote $W_{k,a}^{(i),q^{\pm 1}}$ the Kirillov--Reshitikhin module over $U_{q^{\pm 1}}(\g)$ with parameters $k\in \mathbb{N}_{>0}$, $i\in I$ and $a\in \C^{\times}$ (see Example \ref{ex:KRsl2}). We will need another lemma.
\begin{lem}\label{lem:IqW} Fix $k\in \mathbb{N}_{>0}$, $i\in I$, $a\in \C^{\times}$ and $W = W_{k,a}^{(i),q^{-1}}$. Then $I_q\circ \overline{\chi}_q(W) = \overline{\chi}_q(W_{k,a^{-1}}^{(i),q})$.
\end{lem}
\begin{proof}
Consider the algebra isomorphism $\iota_q : \uqg\arr \uqgm$ defined by\footnote{This is the composition of the isomorphisms defined in \cite[Proposition 2.2]{hj} and \cite[Proposition 1.6]{c}.} $\iota_q(x_{j,r}^{\pm}) = x_{j,-r}^{\pm}$ and $\iota_q(\phi_{j,\pm m}^{\pm})= \phi_{j,\mp m}^{\mp}$ for $j\in I$, $r\in \Z$ and $m\in \mathbb{N}$. Let also $V= \iota_q^*(W)$ be the pullback of $W$ by $\iota_q$. Then a slightly modified version of the proof of \cite[Lemma 4.10]{hSig} gives $\chi_q(V)=I_q\circ\chi_q(W)$. \par Write
$$\Psi = Y_{i,a}^{(q^{-1})}Y_{i,aq_i^{-2}}^{(q^{-1})}\dots Y_{i,aq_i^{2(1-k)}}^{(q^{-1})}$$
for the highest $\ell$-weight of $W$ and take a corresponding highest $\ell$-weight vector $v\in W$. Then, for any $j\in I$, we have $\iota_q(e_j)\star v=e_j\star v= 0$ (where $\star$ denotes the $\uqbm$-action on $W$) and $$\iota_q(\phi_j^+(z))\star v = \phi_j^-(z^{-1})\star v = \phi_j^+(z^{-1}) \star v = \Psi_j(z^{-1})v = \left\{\begin{array}{ll}
q_i^{k}\frac{1-a^{-1} z q_i^{-1}}{1-a^{-1} z q_i^{2k-1}} & \text{if }i=j,\\
1 & \text{else}
\end{array} \right.$$
where the second equality follows from \cite[Theorem 4.10]{hshift}. It hence follows that $V\simeq W_{k,a^{-1}}^{(i),q}$ as $\iota_q^*$ preserves irreducibility (since it is exact and has an exact inverse). \par Therefore, using the relation $I_q([Y_{i,a}^{(q^{-1})}]) = [\overline{\omega_i}][\Psi_{i,a^{-1}q_i^{-1}}][\Psi_{i,a^{-1}q_i}^{-1}] = [Y_{i,a^{-1}}^{(q)}]$, we get
$$ \overline{\chi}_q(W_{k,a^{-1}}^{(i),q})=[(Y_{i,a^{-1}}^{(q)}Y_{i,a^{-1}q_i^2}^{(q)}\dots Y_{i,a^{-1}q_i^{2(k-1)}}^{(q)})^{-1}]I_q(\chi_q(W)) = I_q([\Psi^{-1}]\chi_q(W)) = I_q\circ\overline{\chi}_q(W)$$
and thus the result. 
\end{proof}
\begin{cor}\label{cor:IqNeg} Fix $i\in I$ and $a\in \C^{\times}$. Then $I_q\circ \chi_q(L_{i,a}^{-,q^{-1}}) = \chi_q(L_{i,a^{-1}}^{-,q})$.
\end{cor}
\begin{proof}
Recall the $\ell$-weight $A_{i,a}$ defined after Theorem \ref{thm:qcaract} and note $A_{i,a}'$ the $\ell$-weight obtained after changing the parameter $q$ by $q^{-1}$ in the definition of $A_{i,a}$. Then $I_q(A_{i,a}')= A_{i,a^{-1}}$ and it follows that $I_q$ induces a continuous function (for the usual topology on formal series) from $\Z[[(A_{j,b}')^{-1}]]_{j\in I,b\in\C^{\times}}$ to $\Z[[(A_{j,b})^{-1}]]_{j\in I,b\in\C^{\times}}$. By Remark \ref{rem:LnegW}, we thus obtain
\begin{align*}
I_q \circ \chi_q(L_{i,a}^{-,q^{-1}}) &=
 I_q([\Psi_{i,a}^{-1}]{\textstyle \lim_{k\arr\infty}}\overline{\chi}_q(W_{k,aq_i^{2k-1}}^{(i),q^{-1}})) = [\Psi_{i,a^{-1}}^{-1}]{\textstyle \lim_{k\arr\infty}} (I_q\circ\overline{\chi}_q (W_{k,aq_i^{2k-1}}^{(i),q^{-1}}))\\ &= [\Psi_{i,a^{-1}}^{-1}]{\textstyle \lim_{k\arr\infty}} \overline{\chi}_q(W_{k,a^{-1}q_i^{1-2k}}^{(i),q}) = [\Psi_{i,a^{-1}}^{-1}]\overline{\chi}_q(L_{i,a^{-1}}^{-,q}) = \chi_q(L_{i,a^{-1}}^{-,q})
\end{align*}
where we have used Lemma \ref{lem:IqW}. This ends the proof.
\end{proof}
Define $K_{0,q}$ as the subring of $K_0(\mathcal{O}_q)$ generated by the equivalence classes $[\mu]$, $[V_{i}(a)]$ and $[L_{i,a}^{\pm}]$ for $\mu\in \ttt^{\times}$, $i\in I$ and $a \in \C^{\times}$. Then $K_{0,q}$ contains the subrings $K_0^{\pm} \subseteq K_0(\mathcal{O}^{\pm}_q)$ introduced in \cite[Section 5C]{hl2}. This observation and the work of \cite{hl2} foreshadow the following result (which is interesting in its own right).
\begin{prop}\label{prop:K0q} An element of $K_0(\mathcal{O}_q)$ is an at most countable sum of elements of $K_{0,q}$.
\end{prop}
\begin{proof} For each $\ups\in\mathfrak{r}$, let us fix a factorisation 
$$\textstyle \ups=\varpi(\ups)\left(\prod_{s=1}^{m_{\ups}} \Psi_{i_s(\ups),a_s(\ups)}\right)\left(\prod_{s=1}^{n_{\ups}} \Psi_{j_s(\ups),b_s(\ups)}^{-1}\right)\left(\prod_{s=1}^{r_{\ups}} Y_{\ell_s(\ups),c_s(\ups)}^{(q)}\right).$$
Note that $[L_q(\ups)]$ appears exactly once in the decomposition in $K_0(\mathcal{O}_q)$ of the class
$$\textstyle [M_{\ups}] = \left(\prod_{s=1}^{m_{\ups}} [L_{i_s(\ups),a_s(\ups)}^{+,q}]\right)\left(\prod_{s=1}^{n_{\ups}} [L_{j_s(\ups),b_s(\ups)}^{-,q}]\right)\left(\prod_{s=1}^{r_{\ups}} [V^{q}_{\ell_s(\ups)}(c_s(\ups))]\right)\in K_{0,q}$$
in terms of classes of simple modules. Moreover the other simple classes $[L_q(\Psi_1)]$ appearing in this decomposition all verify $\varpi(\Psi_1)< \varpi(\ups)$. We thus have a decomposition of the form
\begin{equation}\label{eq:DecEt1}
[L_q(\ups)] = [M_{\ups}]- \sum_{\substack{\Psi_1\in\mathfrak{r} \\ \varpi(\Psi_1) < \varpi(\ups)}} N_{\Psi_1,\ups}[L_q(\Psi_1)]
\end{equation}
where $N_{\Psi_1,\ups} \in \mathbb{N}$ is the multiplicity of $[L_q(\Psi_1)]$ in $[M_{\ups}]$.\par
Repeating the above for all the classes $[L_q(\Psi_1)]$ appearing in \eqref{eq:DecEt1} and iterating, we get 
\begin{equation}\label{eq:DecFin}
[L_q(\ups)] = [M_{\ups}]+\sum_{r\geq 1}(-1)^r\sum_{\substack{\Psi_1,...,\Psi_r\in\mathfrak{r}\\\varpi(\Psi_r)<\dots<\varpi(\Psi_1)<\varpi(\ups)}}N_{\Psi_r,\Psi_{r-1}}\dots N_{\Psi_2,\Psi_1}N_{\Psi_1,\Omega}[M_{\Psi_r}].
\end{equation} 
We want to show that the right-hand side of \eqref{eq:DecFin} gives a well-defined countable sum in $K_{0,q}$. For that goal, fix a strictly increasing sequence $\mu_r<\dots<\mu_1<\varpi(\ups)$ in $\ttt^{\times}$ and note that (as the weight-spaces of $M_{\Omega}, M_{\Psi_1}, \dots ,M_{\Psi_r}$ are finite-dimensional) there are at most finitely many $\Psi_1,\dots,\Psi_r\in \ttt^{\times}$ with both $\varpi(\Psi_s) = \mu_s$ for every $1\leq s\leq r$ and $N_{\Psi_r,\Psi_{r-1}}\dots N_{\Psi_2,\Psi_1}N_{\Psi_1,\ups}\neq 0$. In particular, the contribution to \eqref{eq:DecFin} of a given $[M_{\Psi}]$ with $\varpi(\Psi) < \varpi(\ups)$ is finite since the set $\{\mu\in\ttt^{\times}\,|\,\varpi(\Psi)\leq\mu\leq\varpi(\ups)\}$ has finite cardinality by definition of the order $\leq$ on $\ttt^{\times}$ (and hence contains finitely many chains of the form $\varpi(\Psi)<\mu_r<\dots<\mu_1<\varpi(\ups)$). In addition, as $D(\varpi(\ups)) = \{\mu\in\ttt^{\times}|\,\mu\leq \varpi(\ups)\}$ is countably infinite, there are countably many sequences of the form $\mu_r<\dots<\mu_1<\varpi(\ups)$ in $\ttt^{\times}$. There are thus also at most countably many $\Psi\in \mathfrak{r}$ such that $[M_{\Psi}]$ contributes to \eqref{eq:DecFin} and it follows that the above decomposition indeed gives a well-defined countable sum in $K_{0,q}$.\par 
The proposition is hence true for classes of simple objects of $\mathcal{O}_q$. Let us now fix an arbitrary element $X$ of $K_0(\mathcal{O}_q)$. Then $X$ can be decomposed as
$ X = \sum_{i=1}^s r_i [V_i]$
for some $s\in\mathbb{N}_{>0}$, some $r_i \in \Z$ and some objects $V_i$ of $\mathcal{O}_q$. Furthermore, for every $1 \leq i\leq s$, we can decompose the equivalence class $[V_i]$ as an at most countable sum of simple classes $[L_q(\Psi)]$ with coefficients in $\mathbb{N}$. (Note that this sum is at most countable since $V_i$ has at most countably many $\ell$-weights.) We can thus use the above results to express each $V_i$ (and finally $X$) as an at most countable sum of elements of $K_{0,q}$. This concludes the proof.
\end{proof}
\begin{cor} There is a ring isomorphism $\tilde{I}_q:K_0(\mathcal{O}_{q^{-1}})\simeq K_0(\mathcal{O}_q)$ making \eqref{eq:DiagD} commute.
\end{cor}
\begin{proof}
Note that $I_q \circ \chi_q(K_{0,q^{-1}}) \subseteq \chi_q(K_{0,q})$ by Lemma \ref{lem:IqPos}, Lemma \ref{lem:IqW} and Corollary \ref{cor:IqNeg}. Proposition \ref{prop:K0q} and the (obvious) compatibility of the morphisms $\chi_q$ and $I_q\circ\chi_q$ with countable sums of elements of $K_{0,q^{-1}}$ in $K_0(\mathcal{O}_{q^{-1}})$ therefore give $I_q\circ \chi_q(K_0(\mathcal{O}_{q^{-1}})) \subseteq \chi_q(K_0(\mathcal{O}_q))$. The corollary then follows from the injectivity of the $q$-character map (see Theorem \ref{thm:qcaract}).
\end{proof}
Let $\tilde{D}_q$ be the ring automorphism of $K_0(\mathcal{O}_q)$ given by $\tilde{D}_q([V]) = \tilde{I}_q[\mathcal{G}_{q^{-1}}(V)]$ on equivalence classes of objects of $\mathcal{O}_q$. Then the above results and Corollary \ref{cor:F2} give $\tilde{D}_q([\mu]) = [\mu]$ with
$$\tilde{D}_q([V_i^q(a)])=[V_i^q(a^{-1})] \ \ \text{and}\ \ \tilde{D}_q[L_{i,a}^{\pm,q}]=[L^{\mp,q}_{i,a^{-1}}].$$
Thus $\tilde{D}_q(X)=D_q(X)$ (and $\tilde{D}_q(X)=D_q^{-1}(X)$) for all $X$ in the subring $K_{0}^{+}\subseteq K_0(\mathcal{O}_q^{+})$ (resp. $K_{0}^{-}\subseteq K_0(\mathcal{O}_q^{-})$) of \cite[Section 5C]{hl2}. This and \cite[Proposition 5.16]{hl2} (with the compatibility of $\tilde{D}_q$ and $D_q^{\pm 1}$ with countable sums of elements of $K_{0}^{\pm}$) then imply the following theorem. This is the main result of this appendix.
\begin{theorem} The restriction of $\tilde{D}_q$ to $K_0(\mathcal{O}_q^{\pm})$ is equal to $D_q^{\pm 1}$. \hfill $\qed$
\end{theorem}
An important property of Hernandez--Leclerc's duality (which is related to its interpretation in terms of monoidal categorifications of cluster algebras) is that it sends equivalence classes of simple modules to equivalence classes of simple modules. The proof of this property in \cite[Section 7B]{hl2} relies in an essential way on Proposition \ref{prop:limitOm} and is thus hard to extrapolate to the case of the ring automorphism $\tilde{D}_q$. The following question therefore remains unanswered.
\begin{question} Does $\tilde{D}_q$ send classes of simple modules to classes of simple modules?
\end{question}
We conclude this appendix by noting that $\tilde{D}_q$ is involutive, as announced in Theorem \ref{thm:DExt}. Indeed $\tilde{D}_q\circ\tilde{D}_q$ fixes the classes of invertible, fundamental and prefundamental representations in $K_0(\mathcal{O}_q)$. It thus fixes the subring $K_{0,q}\subseteq K_0(\mathcal{O}_q)$ and must therefore be equal to the identity map of $K_0(\mathcal{O}_q)$ by Proposition \ref{prop:K0q} (since it is compatible with countable sums).

\let\oldaddcontentsline\addcontentsline
\renewcommand{\addcontentsline}[3]{}

 \let\addcontentsline\oldaddcontentsline

\end{document}